\documentclass[a4paper,10pt]{article}
\usepackage[utf8]{inputenc} 

\usepackage{wrapfig}
\usepackage{amsmath} 
\usepackage{amsthm} 
\usepackage{amssymb} 
\usepackage{enumerate} 
\usepackage{esint} 
\usepackage{pgf,tikz} 
\usetikzlibrary{arrows} 
 \usepackage{yfonts} 
 \usepackage{mathrsfs} 
 \usepackage{graphicx}
\usepackage{caption}
\usepackage{subcaption}

\newcommand{\C}{{\mathbb C}}       
\newcommand{\R}{{\mathbb R}}       
\newcommand{\N}{{\mathbb N}}       
\newcommand{\Z}{{\mathbb Z}}       
\newcommand{\D}{{\mathbb D}}

\newcommand{\dist}{{\rm dist}}

\newcommand{\real}{{\rm Re \,}}
\newcommand{\imag}{{\rm Im}}
\newcommand{\rf}[1]{{(\ref{#1})}}
\newcommand{\supp}{{\rm supp}}

\newcommand{\Beurling}{{\mathbf B}}
\newcommand{\Cauchy}{{\mathbf C}}

\newcommand{\norm}[1]{{\left\|{#1}\right\|}}
\newcommand{\modulus}[2]{\omega_{#1}#2}

\definecolor{ffffff}{rgb}{1.0,1.0,1.0}
\definecolor{qqqqff}{rgb}{0.0,0.0,1.0}
\definecolor{ffqqqq}{rgb}{1.0,0.0,0.0}
\definecolor{zzzzqq}{rgb}{0.6,0.6,0.0}
\definecolor{marronet}{rgb}{0.6,0.2,0}
\definecolor{negre}{rgb}{0,0,0}
\definecolor{vermell}{rgb}{0.8,0.05,0.05}
\definecolor{blau}{rgb}{0.2,0.1,1}
\definecolor{blauclar}{rgb}{0.,0.,1.}
\definecolor{grisfosc}{rgb}{0.25098039215686274,0.25098039215686274,0.25098039215686274}
\definecolor{verd}{rgb}{0.1,0.6,0.1}
\definecolor{taronja}{rgb}{0.9,0.6,0.05}
\definecolor{vermellclar}{rgb}{1.,0.,0.}
\definecolor{verdet}{rgb}{0,0.8,0.1}
\definecolor{blauverd}{rgb}{0,0.4,0.2}
\definecolor{grisclar}{rgb}{0.6274509803921569,0.6274509803921569,0.6274509803921569}

\newtheorem{theorem}{Theorem}
\newtheorem*{theorem*}{Theorem}
\newtheorem{lemma}[theorem]{Lemma}
\newtheorem{claim}[theorem]{Claim}

\newtheorem{corollary}[theorem]{Corollary}
\newtheorem*{corollary*}{Corollary}
\newtheorem{proposition}[theorem]{Proposition}
\newtheorem{definition}[theorem]{Definition}

\newtheorem{example}[theorem]{Example}
\newtheorem{remark}[theorem]{Remark}

\numberwithin{subsection}{section}
\numberwithin{theorem}{section}
\numberwithin{equation}{section}
\numberwithin{figure}{section}

\usepackage[affil-it]{authblk}

\title{Characterization for stability in planar conductivities}

\author{Daniel Faraco, Mart\'i Prats
\thanks{U\-ni\-ver\-si\-dad Au\-t\'o\-no\-ma de Ma\-drid - ICMAT, Spain: \texttt{daniel.faraco@uam.es}, \texttt{marti.prats@uam.es}. The authors were funded by the European Research
Council under the grant agreement 307179-GFTIPFD and MTM2011-28198 and they acknowledge financial support from the Spanish Ministry of Economy and Competitiveness, through the ``Severo Ochoa'' Programme for Centres of Excellence in R\&D (SEV-2015-0554). The second author was partially funded by AGAUR - Generalitat de Catalunya (2014 SGR 75) as well.}}

\begin{document}
\maketitle
\bibliographystyle{alpha}

\begin{abstract} 
We find a complete characterization for sets of uniformly strongly elliptic and isotropic conductivities with stable recovery in the $L^2$ norm when the data of the Calder\'on Inverse Conductivity Problem is obtained in the boundary of a disk and the conductivities are constant in a neighborhood of its boundary.
To obtain this result, we present  minimal a priori assumptions which turn out to be sufficient for  sets of conductivities to have stable recovery in a bounded and rough domain.  
The condition is presented in terms of the integral moduli of continuity of the coefficients involved and their ellipticity bound as conjectured by Alessandrini in his 2007 paper, giving explicit quantitative control for every pair of conductivities.
\end{abstract}

\renewcommand{\abstractname}{}
\begin{abstract}

\end{abstract}

\section{Introduction}
Let $\gamma$ be a strongly elliptic, isotropic conductivity coefficient in a bounded domain $\Omega\subset \C$, that is $\gamma:\C \to \R_+$ with $\supp(\gamma-1 )\subset \overline\Omega$ and both $\gamma$ and its multiplicative inverse $\gamma^{-1}$  bounded above by $K<\infty$ modulo null sets, which we summarize as
$$\gamma \in \mathcal{G}(K,\Omega).$$
 For $1\leq p \leq \infty$, the set $\mathcal{G}(K,\Omega)$ is a metric space when endowed with the $L^p$-distance 
\begin{equation*}
\dist_{p}^\Omega(\gamma_1,\gamma_2)= \norm{\gamma_1-\gamma_2}_{L^p(\Omega)}.
\end{equation*}
The conductivity inverse problem, proposed in 1980 by Alberto Calder\'on (see \cite{CalderonInverse}), consists in determining $\gamma$ from boundary measurements. This measurements are samples of the Dirichlet-to-Neumann (DtN) map $\Lambda_\gamma: H^{1/2}(\partial\Omega)\to H^{-1/2}(\partial\Omega)$ which sends a function $f$ to $\gamma\frac{\partial u}{\partial \nu}$, being $\nu$ the outward unit  normal  vector of $\partial\Omega$ and $u$ the solution of the Dirichlet boundary value problem
\begin{equation}\label{eqConductivity}
\begin{cases}
	\nabla \cdot (\gamma \nabla u)=0, \\
	u_{|\partial\Omega}=f.
\end{cases}
\end{equation}
See Section \ref{secDirect} for the precise weak formulation of this equation and the DtN map.


In dimension 2, after the milestones \cite{Nachman} and \cite{BrownUhlmann}, Astala and P\"aiv\"arinta showed  in \cite{AstalaPaivarinta} that every pair of conductivity coefficients $\gamma_1,\gamma_2 \in \mathcal{G}(K,\Omega)$ satisfies that $\Lambda_{\gamma_1}=\Lambda_{\gamma_2}$ if and only if $\gamma_1=\gamma_2$.  In higher dimensions, there are uniqueness results which require some a priori regularity of $\gamma$ (see \cite{CaroRogers, Haberman, HabermanTataru, BrownTorres, SylvesterUhlmann}). 

Nevertheless, for the problem to be well-posed  the inverse map $\Lambda_\gamma \mapsto \gamma$ should be continuous in some sense, to have a chance that in real life situations if our sample differs slightly from the  DtN map of a certain body, then the solution we get is a good approximation of the conductivity of that body. Alessandrini showed in \cite{AlessandriniStable} that this is not possible in the $L^{\infty}$-distance unless one imposes a priori conditions. 
In fact, $G$ convergence gives explicit examples of highly oscillating sequences such that convergence of the DtN map
does not imply the convergence of the conductivities in any $\dist_{p}^\Omega$ distance (see \cite{AlessandriniCabib,FaracoKurylevRuiz}).

In \cite{BarceloFaracoRuiz}  this problem was solved assuming that the space of conductivities is $L^{\infty} \cap C^\alpha(\Omega)$ by a cunning adaptation of Astala and P\"aiv\"arinta arguments, to obtain stability in the $L^\infty$-distance for Lipschitz domains.  Later on, in \cite{ClopFaracoRuiz}   that result was extended to non-smooth conductivities, as long as they belong to a fractional Sobolev space $H^\alpha(\Omega)$ and showing stability with respect to the $L^p$-distance with $p<\infty$.  Estimates for rough domains were obtained in \cite{FaracoRogers}.  
Previous remarkable steps can be found in \cite{AlessandriniSingular, Liu}.

The purpose of the present paper is to characterize the subsets  $\mathcal{F}\subset \mathcal{G}(K,\Omega)$ such that the inverse map is $L^2$-stable for these conductivities. The condition studied is to have a uniform integral modulus of continuity of exponent $p$  for $1\leq p<\infty$, under which $L^2$-stability is shown, the condition being necessary when $\Omega=\D$ and the conductivities are constant in a neighborhood of the boundary. Incidentally, every uniformly elliptic conductivity has a bounded integral modulus of continuity for every finite exponent.

\begin{definition}\label{defModulus}
An increasing function $\omega:\R_+\to \R_+$ with $\lim_{t\to0} \omega(t)= 0$ is called {\em modulus of continuity}.

Let $f : \R^d \to \C$ be a measurable function and let $0< p \leq \infty$. We define its integral modulus of continuity of exponent $p$, or $p$-modulus for short, as
\begin{equation*}
\modulus{p}{f}(t) := \sup_{|y|\leq t} \norm{ f - \tau_y f}_{L^p} \mbox{\quad\quad for $0\leq t \leq \infty$},
\end{equation*}
where we wrote $\tau_y f(x) =  f(x-y)$.
\end{definition}
Note that $\modulus{p}{f}$ is increasing by definition. If $f\in L^p$ with $1\leq p<\infty$, then  $\modulus{p}{f}$  is a modulus of continuity by the Kolmogorov-Riesz Theorem (see \cite[Theorem 5]{HancheOlsenHolden}, for instance). However this is not true for $L^\infty$, since  $\lim_{t\to0} \modulus{\infty}{f}(t)= 0$ is equivalent to uniform continuity of $f$. 

Next we define the families of conductivities under study:
\begin{definition}\label{defFamiliesOfModuli}
Let $\gamma \in \mathcal{G}(K,\Omega)$ for $1\leq K<\infty$. Let $0< p \leq\infty$ and assume  that for a certain modulus of continuity  $\omega$ we have the pointwise bound $\modulus{p}{\gamma}\leq \omega$. Then we say that $\gamma \in\mathcal{G}(K,\Omega, p,\omega)$.
\end{definition}

\begin{definition}
Let $0<p\leq\infty$. We say that a  family of conductivities $\mathcal{F}\subset \mathcal{G}(K, \Omega)$ is $L^p$-stable for (recovery in) $\Omega$ if the map $\Lambda_\gamma \mapsto \gamma$ is uniformly continuous in $\Lambda(\mathcal{F})$ with respect to the (semi)distance $\dist_p^\Omega$ in $\mathcal{F}$. 
\end{definition}

The main result of the paper is summarized in the following theorem.
\begin{theorem}\label{theoBigUp}
Let $K\geq 1$, let $r_0<1$ and let $\mathcal{F}\subset \mathcal{G}(K, r_0\D)$. The family $\mathcal{F}$ is $L^2$-stable for  $\D$ if and only if there exists a modulus of continuity $\omega$ such that $\mathcal{F}\subset \mathcal{G}(K,r_0\D, 2, \omega)$.
\end{theorem}

 Theorem \ref{theoBigUp} is divided in two parts. The necessity of the condition is based on a compactness argument and it requires that the conductivities involved coincide near the boundary:
\begin{theorem}\label{theoSharp}
Let $K\geq 1$, $0 < s<\infty$ and $r_0<1$. Let $\mathcal{F}\subset \mathcal{G}(K, r_0\D)$ be an $L^s$-stable family of conductivities for recovery in $\D$. 

Then for every $0< p<\infty$ there exists a modulus of continuity $\omega$ such that $\mathcal{F}\subset \mathcal{G}(K,\Omega, p, \omega)$. 
\end{theorem}

As pointed out by Alessandrini in \cite{AlessandriniOpen}, whenever the forward map is continuous in a compact subset of $\mathcal{G}(K,\Omega)$ with the $L^p$-distance, the inverse map is continuous by trivial arguments (and it is well defined by \cite{AstalaPaivarinta}). Given any set $D$ compactly supported in a domain $\Omega$, one can check that the forward mapping is continuous in every $L^p$ by using the higher integrability of the gradient and Lemma \ref{lemCompareLambdas} below. The family $\mathcal{G}(K,D,p,\omega)$ is compact in the metric space $\mathcal{G}(K,\Omega)$ endowed with the $L^p$ distance by the Kolmogorov-Riesz criterion and, thus, it is always $L^p$ stable for $\Omega$, finishing the proof of the sufficiency, but this reasoning does not provide us with information on the modulus of continuity of the inverse mapping. 

In Theorem \ref{theoMainTheorem} below we give a quantitative version of this result and we do it without any assumption on the continuity of the forward mapping, solving all the questions posed in \cite{AlessandriniOpen} on the minimal a priori assumptions for stability in dimension 2. The argument is quite deeper, and it works in a rather general setting, with no conditions on the boundary of the domain (improving \cite{FaracoRogers}):
\begin{theorem}\label{theoMainTheorem}
Let $K\geq 1$, let $2K< p <\infty$, let $\Omega$ be a bounded domain and let $\omega$ be a modulus of continuity. Then the family  $\mathcal{G}(K,\Omega, p, \omega)$ is  $L^2$-stable for $\Omega$. That is, there exists a modulus of continuity $\eta$  depending only on $K$, $\Omega$, $p$ and $\omega$ so that  every pair of coefficients $\gamma_1, \gamma_2\in {\mathcal{G}}$ satisfies that
$$ \dist_2^\Omega (\gamma_1,\gamma_2) \leq \eta \left( \norm{ \Lambda_{\gamma_1} - \Lambda_{\gamma_2} }_{H^{1/2}(\partial\Omega)\to H^{-1/2}(\partial\Omega)} \right). $$

Moreover, if $\omega$ is upper semi-continuous, there exist constants $C_{\kappa,p}$, $C_K$, $b_{\kappa,p}$ and $\alpha_K$ such that 
\begin{equation*}
\eta(\rho)
	\lesssim_{\kappa,p}  \omega\left(\frac{C_K}{|\log(\rho)|^\frac1K}\right)^{b_{\kappa,p}} +  \frac{1}{|\log(\rho)|^{\alpha_K}} + \omega\left(  C_{\kappa,p} \,  \omega\left(\frac{C_K}{|\log(\rho)|^\frac1K}\right)^{b_{\kappa,p}} +  \frac{C_K}{|\log(\rho)|^{\alpha_K}}\right).
\end{equation*}
\end{theorem}

Note that if $\omega$ is not upper semi-continuous, then $\omega^*(t):=\inf_{s>t}\omega(s)$ is upper semi-continuous and $B^\omega_p\subset B^{\omega^*}_p$.
 Theorem \ref{theoMainTheorem} can be stated replacing $2K< p <\infty$  by $0< p <\infty$ via a quick interpolation argument with $L^\infty$. However, we will use along the proof the integral modulus of continuity of exponent $p$ satifying the first restriction, see Section \ref{secDecay} for the details.
A similar replacement can be done for the estimates on the distance, using interpolation and H\"older inequalities, to obtain $L^s$-stability for $0<s<\infty$, leading to 
$$ \dist_s^\Omega (\gamma_1,\gamma_2) \lesssim_{s,|\Omega|} \eta \left( \norm{ \Lambda_{\gamma_1} - \Lambda_{\gamma_2} }_{H^{1/2}(\partial\Omega)\to H^{-1/2}(\partial\Omega)} \right)^\frac2s. $$
 Similar reasonings can be used to rewrite Theorem \ref{theoBigUp} in these terms.

To show Theorem \ref{theoMainTheorem}, we follow the scheme of \cite{BarceloFaracoRuiz} and \cite{ClopFaracoRuiz}: first we define Complex Geometric Optics Solutions (CGOS) for a certain Beltrami equation as in \cite{AstalaPaivarinta}, and we show that the decay rate is indeed a function of the modulus of continuity. Then we derive $L^\infty$-stability of the CGOS from the scattering transform and finally we use an interpolation argument to deduce stability for the conductivities. 

The decay argument follows the line of \cite{ClopFaracoRuiz}, but we need to study the interaction of quasiconformal mappings, the Beurling transform and the modulus of continuity

 Then we use the approach of Nachman and study the equation in the $k$ variable.  As in \cite{BarceloFaracoRuiz} we are able to use topological arguments to obtain estimates for the 
solutions $u(z,k)$ using their asymptotics in both variables. The proof here follows closely 
\cite{BarceloFaracoRuiz} but we present here shorter and clearer arguments. In particular we are able to keep track of the evolution of the various moduli of continuity.

 Finally, the interpolation argument presented at the end of the paper is inspired by the one in \cite{ClopFaracoRuiz}, but we cannot use Sobolev interpolation. Instead, we need to solve some intermediate steps which are of interest by themselves and which may shed some light on the original argument.

At this point, we want to draw the attention of the reader to the question of the regularity of quasiconformal mappings in relation with their Beltrami coefficient. In \cite{ClopFaracoMateuOrobitgZhong}, \cite{ClopFaracoRuiz}, \cite{CruzMateuOrobitg}, \cite{BaisonClopOrobitg} and \cite{PratsBeltrami} the authors establish Sobolev regularity of the principal solution of the $\R$-linear Beltrami equation in terms of the regularity of the  pair of Beltrami coefficients. Here we show that the $p$-modulus of the derivative of a quasiregular solution is locally controlled by the $p$-modulus of its Beltrami coefficient and the natural H\"older regularity of the mapping by means of Caccioppoli inequalities: 

\begin{theorem}\label{theoCaccioppoliModulusModified}
Let $\mu,\nu \in L^\infty$ be compactly supported with $\norm{|\mu|+|\nu|}_{L^\infty} \leq \kappa< 1$. Let $f$ be a quasiregular solution to 
$$	\bar\partial f = \mu \,  \partial f + \nu \,  \overline{\partial f }. $$
Let $1<p<p_\kappa$ satisfy that $\kappa \norm{\Beurling}_{L^p\to L^p}<1$, let $r\in [p, p_\kappa)$ and let $q$ be defined by $\frac1p=\frac1q+\frac1r$. Then, for every real-valued Lipschitz function $\varphi$ compactly supported in $\D$, we have that
\begin{equation*}
\modulus{p}{(\varphi \bar\partial f)}(t)
	 \leq  C_{\kappa,r,p} \norm{f \nabla \varphi}_{L^r}\left(  \modulus{q}{\mu}(t)+  \modulus{q}{\nu}(t)\right)  +C_{\kappa,p,\varphi} \, \norm{f}_{W^{1+p} ((t+1)\D)}  |t|^{1-\frac2p}.
\end{equation*}
\end{theorem}

Still regarding the regularity of quasiconformal mappings we establish another result of independent interest which deals with the $p$-modulus of a function when precomposed with a quasiconformal mapping $\phi$. In Lemma \ref{lemCompositionLebesgue} it is shown that for convenient indices $p$ and $q$ the $q$-modulus of $f\circ \phi$ can be controlled in terms of $\modulus{p}{f}$. The result obtained is clearly non-sharp, as the results in Sobolev spaces obtained in \cite{HenclKoskelaComposition} and \cite{OlivaPrats} show. Optimal estimates would be highly appreciated.

There is another open question which is of interest.  It may happen that convenient modifications of the arguments presented here lead to an expression such as 
\begin{equation}\label{eqConjecture}
 \dist_s^\Omega (\gamma_1,\gamma_2) \lesssim \norm{\gamma_1-\gamma_2}_{B^\omega_p} \,\eta \left(\frac{\norm{ \Lambda_{\gamma_1} - \Lambda_{\gamma_2} }_{H^{1/2}(\partial\Omega)\to H^{-1/2}(\partial\Omega)} }{\norm{\gamma_1-\gamma_2}_{B^\omega_p}}\right).
 \end{equation}
In \cite{AlessandriniVessella} the authors present a technique which allows to deduce Lipschitz stability whenever we restrict ourselves to a finitely generated family of conductivities, as long as  \rf{eqConjecture} holds. Some steps can be done to modify Sections \ref{secLinear} and \ref{secNonlinear} accordingly but, as it happened already in the quest for uniqueness, obstacles appear when trying to derive decay in the conductivity CGOS in terms of $\norm{\gamma_1-\gamma_2}_{B^\omega_p}$. If these obstacles could be overcome, the authors are convinced that an expression in the spirit of \rf{eqConjecture} would be obtained  and the techniques in \cite{AlessandriniVessella} could be applied in this context.

Let us give a final thought to end this introduction. Given a body with conductivity coefficient $\gamma$, if we measure $\widetilde{\Lambda}$ in an experimental setting and it does not coincide with $\Lambda_{\gamma}$ for any $\gamma$, or if it does coincide with $\Lambda_\gamma$ but the corresponding $\gamma$ does not have the regularity which we assumed a priori, then we have to deal  with the additional problem of projecting our sample to the space of admissible DtN maps, say $\widetilde{\Lambda}\mapsto \Lambda_{\widetilde{\gamma}}$. In other words, it is convenient to find a regularization strategy for the Dirichlet-to-Neumann map. There are several approaches to this problem, each of them valid with some extra a priori assumptions. We refer the reader to \cite{KnudsenLassasMuellerSiltanen,AstalaMuellerPaivarintaSiltanen} and references therein for that particular question. It remains open to find a regularization strategy that fits with the present paper approach.

The paper is organized as follows: after the preliminaries below and a word on the direct problem in Section \ref{secDirect},  Theorem \ref{theoSharp} is shown in Section \ref{secSharp}. 
In Section \ref{secBackground} the strategies of some previous works to deal with uniqueness and stability of Calder\'on's inverse problem are recalled. In particular, Section \ref{secQuasiregular} contains some general results on quasiconformal and quasiregular mappings to be used along the present paper, Section \ref{secInverse} introduces the big picture of the inverse problem, Section \ref{secReduction} details the reduction to the unit disk and Section \ref{secCGOS} is devoted to present the CGOS and some other tools introduced by previous authors. 

Following the approach in \cite{AstalaPaivarinta}, in Section \ref{secDecay} it is seen how the modulus of continuity determines the decay of CGOS. The interplay of the moduli with certain operators such as the Beurling and the Fourier transforms is introduced in Section \ref{secInteractionCGOS}. Then, some decay properties of the solutions to a linear equation  is shown  in Section \ref{secLinear}, to reduce the decay of the CGOS for the Beltrami equation to the previous case in Section \ref{secNonlinear}. To end this part, in Section \ref{secDecayConductivity} it is  checked that this information can be brought to the conductivity equation. 

In Section \ref{secScatteringToUInfty} the stability from the scattering transform to the CGOS is studied, presenting first the main argument and a delicate topological argument in Section \ref{secGauge} and then completing the details of the proof in Section \ref{secDetails}. Finally, Section \ref{secFinal} is devoted to show the Theorem \ref{theoMainTheorem}, via some Caccioppoli inequalities presented in Section \ref{secCaccioppoli} and a final interpolation argument in Section \ref{secInterpolation}.

\subsection{Notation}
Given a distribution $f\in D'(\C)$ we will denote its gradient $\nabla f =(\partial_x f, \partial_y f)$, that is, a pair of distributions given by the usual partial distributional derivatives of $f$. Whenever the derivatives coincide with $L^1_{loc}$ functions, we will use the expression $|\nabla f|$ to denote the function $|\partial_x f| + |\partial_y f|$. We will adopt the Wirtinger notation for derivatives as well, $\partial f:= \frac12(\partial_x f - i \partial_y f)$ and $\bar \partial f:= \frac12(\partial_x f + i \partial_y f)$. Integration with respect to the Lebesgue measure will be indicated by $dm$, whereas we reserve the notation $dz$ (or $d\xi$, $d\zeta$ and so on) for the line integration with form $dx + i dy$, where $z=x+iy$. 

Let $\Omega\subset \C$ be a domain. For $0<s<1$, we say that a function $u\in \dot C^s(\Omega)$ if the seminorm $\norm{u}_{\dot C^s(\Omega)}:=\sup_{x,y \in \Omega} \frac{|u(x)-u(y)|}{|x-y|^s}<\infty$. We say that $u\in C^s(\Omega)$ if $\norm{u}_{C^s(\Omega)}:=\norm{u}_{L^\infty(\Omega)}+\norm{u}_{\dot C^s(\Omega)}<\infty$.

We say that a distribution $u \in \dot W^{1,p}(\Omega)$  if its weak derivatives in $\Omega$ are $L^1_{loc}(\Omega)$ functions and the Sobolev homogeneous seminorm
$\norm{u}_{\dot W^{1,p}(\Omega)} : = \norm{\nabla u}_{L^p(\Omega)}<\infty$.
We say that $u\in L^1_{loc}(\Omega)$ is in the (non-homogeneous) Sobolev space $W^{1,p}(\Omega)$ if $\norm{u}_{W^{1,p}(\Omega)} : = \norm{u}_{L^p(\Omega)}+\norm{u}_{\dot W^{1,p}(\Omega)}<\infty$. We denote $H^1(\Omega):=W^{1,2}(\Omega)$. We say that $u\in H^1_0(\Omega)$ if, in addition, there exists a collection of smooth functions compactly supported on $\Omega$, say $\{u_j\}_{j=0}^\infty\subset C^\infty_c(\Omega)$, such that  $\lim_{j\to\infty}\norm{u-u_j}_{H^1(\Omega)}=0$. 

As usual, we define $H^{1/2}(\partial\Omega):=H^1(\Omega)/H^1_0(\Omega)$, that is, the quotient space (see \cite[Theorem 3.13]{Schechter}, for instance), inducing a trace operator ${\rm tr}_{\partial\Omega}: H^1(\Omega) \to H^{1/2}(\partial\Omega)$ which is the class (coset) function, and we call   its dual $H^{-1/2}(\partial\Omega)$ (see \cite[Theorem 2.10]{Schechter}). When the domain is regular enough, this trace spaces have an equivalent formulation in terms of Bessel potentials (see \cite{TriebelTheory}, for instance).

Given a modulus of continuity $\omega$, we define $B^\omega_p$ as the collection of $L^p$ functions $u$ which have the seminorm
$$\norm{u}_{\dot B^\omega_p} := \sup_{t>0} \frac{\omega_p u(t)}{\omega(t)}<\infty.$$
The collection $B^\omega_p$ is a Banach space when endowed with the norm $\norm{u}_{ B^\omega_p}:=\norm{u}_{L^p} + \norm{u}_{\dot B^\omega_p}$.  We will use as well the classical homogeneous Besov seminorms 
$$\norm{u}_{\dot B^s_{p,q}} := \left(\int_0^1 \left(\frac{\omega_p u(t)}{t^{s}}\right)^q \frac{dm(t)}{t}\right)^\frac1q<\infty,$$
and their non-homogeneous counterpart. For more information, we refer the reader to \cite[Section 1.11.9]{TriebelTheoryIII}. Note that if $\omega(t)=t^s$ with $0<s<1$ and $p<\infty$, then $B^\omega_p$ coincides with the Besov space $B^s_{p,\infty}$, and $B^\omega_\infty=C^s$. 

\renewcommand{\abstractname}{Acknowledgements}
\begin{abstract}
The authors would like to thank Alberto Ruiz for several discussions on the matter and Giovanni Alessandrini for his enlighting feedback after the first draft of this paper was released.
\end{abstract}

\section{Compactness of the collection of DtN maps}
\subsection{The direct problem}\label{secDirect}

Consider $K<\infty$, and $\gamma \in \mathcal{G}(K,\Omega)$. Given $w\in H^1(\Omega)$ and its trace class $f={\rm tr}_{\partial\Omega} w \in H^{1/2}(\partial\Omega)$, we say that $u_\gamma \in w + H^1_0 (\Omega)$ is a weak solution to \rf{eqConductivity} if 
\begin{equation}\label{eqWeakDirichlet0}
\int_\Omega \gamma \nabla u_\gamma \cdot \nabla v =0  \mbox{\quad\quad for every }v\in H^1_0(\Omega).
\end{equation}
Lax-Milgram Theorem (see \cite[Theorem 6.2.1]{Evans}, for instance) precisely grants the existence and uniqueness of such a solution and we have the energy estimate 
\begin{equation*}
\norm{u_\gamma }_{H^1(\Omega)}\leq  C_K \norm{f}_{H^{1/2}(\partial\Omega)}.
\end{equation*}

The image of $f$ by the DtN map is defined as the trace of $\gamma \nabla u_\gamma$ in the direction normal to the boundary. In the weak context, this means that $\Lambda_\gamma f$ is in the dual of $H^{1/2}(\partial\Omega)$, that is,  $H^{-1/2}(\partial\Omega)$ in the following sense:
 \begin{definition}\label{defDtN}
 Let $f,g\in H^{1/2}(\partial\Omega)$ and assume that $g$ has representative $G\in H^1(\Omega)$. Then, the DtN map of $f$ is defined by acting on $g$ as
\begin{equation*}
\langle \Lambda_\gamma f, g\rangle : = \int_\Omega \gamma \nabla u_\gamma \cdot \nabla G \, dm ,
\end{equation*}
where $u_\gamma$ is the solution to \rf{eqWeakDirichlet0}.
\end{definition}
Note that \rf{eqWeakDirichlet0} implies that this definition does not depend on the particular choice of $G$. With a convenient pick, it follows that
\begin{equation*}
\norm{\Lambda_\gamma}_{H^{1/2}(\partial\Omega)\to H^{-1/2}(\partial\Omega)} \leq  C_{K,|\Omega|}.
\end{equation*}
%
%
By standard techniques, we have the following well-known theorem.
\begin{theorem}
The map $\gamma \mapsto \Lambda_\gamma$ is bounded and continuous from  $\mathcal{G}(K,\Omega)$
 to
 $$\mathcal{L}_{1/2,-1/2}(\Omega):=\left\{\Lambda : H^{1/2}(\partial\Omega) \to H^{-1/2}(\partial\Omega)\right\},$$
with  $\dist_\infty^\Omega$ in $\mathcal{G}(K,\Omega)$ and the topology induced by the usual norm in $\mathcal{L}_{1/2,-1/2}(\Omega)$.
\end{theorem}

\subsection{Proof of Theorem \ref{theoSharp}}\label{secSharp}

First we compare the DtN map of two conductivities which coincide near the boundary. To simplify notation, given $\gamma_1,\gamma_2 \in \mathcal{G}(K,\Omega)$ we write $\Lambda_j$ for $\Lambda_{\gamma_j}$.
\begin{lemma}\label{lemCompareLambdas}
Let $\Omega$ be a bounded domain with an open subset $U\subset\Omega$, let $f_1,f_2\in H^\frac12(\partial\Omega)$ and $\gamma_1,\gamma_2 \in \mathcal{G}(K,\Omega)$ whose difference is supported in $U$. If $u_j$ satisfies \rf{eqConductivity} with conductivity $\gamma_j$ and boundary condition $f_1$ for $j\in \{1,2\}$, then for every $F\in H^1(\Omega)$ with trace $f_2$ we have that
$$|\langle (\Lambda_1-\Lambda_2)f_1,f_2\rangle|\leq C_K \norm{\nabla u_2}_{L^2(U)} \norm{\nabla F}_{L^2(\Omega)}.$$
\end{lemma}

\begin{proof}
Let $F\in H^1(\Omega)$ be a representative of $f_2$. Note that 
\begin{align*}
\langle (\Lambda_1-\Lambda_2)f_1,f_2\rangle
	& =  \int_\Omega (\gamma_1 \nabla u_1-\gamma_2\nabla u_2) \cdot \nabla F \, dm\\
	& =  \int_\Omega \gamma_1 \nabla \left(u_1- u_2\right) \cdot \nabla F \, dm+\int_\Omega (\gamma_1-\gamma_2) \nabla u_2 \cdot \nabla F \, dm.
\end{align*}
Taking absolute values,
\begin{align*}
\left|\langle (\Lambda_1-\Lambda_2)f_1,f_2\rangle\right|
	& \leq  K \norm{\nabla \left(u_1- u_2\right)}_{L^2(\Omega)} \norm{ \nabla F }_{L^2(\Omega)}+ 2K \norm{\nabla u_2}_{L^2(U)} \norm{\nabla F}_{L^2(U)}.
\end{align*}


Note that $u_1$ and $u_2$ have trace $f_1$ by definition, so $u_1- u_2\in W^{1,2}_0(\Omega)$.  Therefore, using \rf{eqWeakDirichlet0} for $u_1$ and $u_2$ we obtain
\begin{align*}
\norm{\nabla \left(u_1- u_2\right)}_{L^2(\Omega)}^2
	&  = \int_\Omega \nabla \left(u_1- u_2\right)\cdot \nabla \left(u_1- u_2\right)\, dm 
		\leq K \int_\Omega \gamma_1 \nabla \left(u_1- u_2\right)\cdot \nabla \left(u_1- u_2\right)\, dm \\
	& =  K  \left| \int_{\Omega} (\gamma_1-\gamma_2)\nabla  u_2 \cdot \nabla \left(u_1- u_2\right)  \, dm\right|\\
	& \leq  K\norm{\gamma_1-\gamma_2}_{L^\infty(U)} \norm{\nabla u_2}_{L^2(U)}\norm{\nabla \left(u_1- u_2\right)}_{L^2(U)} .
\end{align*}
Thus, dividing by $\norm{\nabla \left(u_1- u_2\right)}_{L^2(\Omega)} $ we obtain
\begin{align*}
\norm{\nabla \left(u_1- u_2\right)}_{L^2(\Omega)}
	& \leq  C_K \norm{\nabla u_2}_{L^2(U)} .
\end{align*}

Summing up,
\begin{align*}
\left|\langle (\Lambda_1-\Lambda_2)f_1,f_2\rangle\right|
	& \leq  C_K  \norm{\nabla u_2}_{L^2(U)} \norm{ \nabla F }_{L^2(\Omega)}+ C_K \norm{\nabla u_2}_{L^2(U)} \norm{\nabla F}_{L^2(U)},
\end{align*}
and the lemma follows.
\end{proof}

Next we consider the $L^2(\frac{d\theta}{2\pi})$-orthonormal  family of spherical harmonics in $\partial\D$ given by
$$ f_{jp}(e^{i\theta}) =\begin{cases} 
	c \cos(j\theta) & \mbox{if } 0< j<\infty \mbox{ and }p = 1, \\ 
	c \sin(j\theta) & \mbox{if } 0< j<\infty \mbox{ and }p = -1, \\ 
	1 & \mbox{if } j=0 \mbox{ and } p=0, 
	\end{cases}$$ 
 with $c=\frac\pi2$. We consider their harmonic extensions
to $\C$ expressed in polar coordinates as $P_{jp}(re^{i\theta}) = r^j f_{jp}(e^{i\theta})$. Note that $P_{j,1}(z)=c \, \real (z^{j})$ and $P_{j,-1}(z)=c\, \imag (z ^{j})$.

By a change of variables, we have that 
\begin{equation*}
\norm{P_{jp}}_{L^2(r_0\D)}=\left( \int_0^{r_0} r^{2j} \int_{\partial\D} |f_{jp}(e^{i\theta})|^2\, d\theta \, r dr\right)^\frac12 = \frac{1}{\sqrt{2j+2}} r_0^{j+1} \norm{f_{jp}}_{L^2(\partial\D)} = \frac{1}{c \sqrt{2j+2}} r_0^{j+1},
\end{equation*}
and for $j>0$, since $\partial (P_{j,1} + i P_{j,-1})= cjz^{j-1}$ and $ \partial (P_{j,1} - i P_{j,-1})= 0$, we have that
\begin{equation}\label{eqNablaHarmonic}
\norm{\partial P_{j,p}}_{L^2(r_0\D)} \approx j \left(\norm{ P_{j-1,1}}_{L^2(r_0\D)}+\norm{ P_{j-1,-1}}_{L^2(r_0\D)}\right)  \approx j^\frac12 r_0^{j} .
\end{equation}
One can argue analogously for $\bar\partial$. 

\begin{lemma}\label{lemHarmonicsUnderControl}
Let $0<r_0<1$ and let $\gamma \in \mathcal{G}(K,r_0\D)$. Denote $\Gamma(\gamma)=\Lambda_\gamma-\Lambda_0$, where $\Lambda_0$ denotes the DtN map corresponding to the conductivity $\gamma_0\equiv 1$ in $\partial \D$. Then
$$|\langle \Gamma(\gamma) f_{jp}, f_{kq}\rangle| \leq C_K \sqrt{jk} \, r_0^{\max\{j,k\}}  \leq C_{K,r_0}.$$
\end{lemma}
\begin{proof}
By the symmetry of the DtN maps, we can assume that $j\geq k$.
Using Lemma \ref{lemCompareLambdas} and estimate \rf{eqNablaHarmonic} we get that
$$|\langle (\Lambda_\gamma-\Lambda_0)f_{jp},f_{kq} \rangle|\leq C_K \norm{\nabla P_{jp}}_{L^2(r_0\D)} \norm{\nabla P_{kq}}_{L^2(\D)}\leq C_K \sqrt{jk} \, r_0^{j}.$$
\end{proof}


Next we show that even though $\mathcal{G}(K,r_0\D)$ is not compact, its image $\Gamma(\mathcal{G}(K,r_0\D))$ is a compact set in $\mathcal{L}_{s,-s}(\D)$, following  the  approach given at \cite[Lemma 3]{MandacheInstability}.
\begin{lemma}\label{lemGCompactInL}
Let $r_0<1$ and $s\in \R$. Then $\Gamma(\mathcal{G}(K,r_0\D))$ is a totally bounded subset of $ \mathcal{L}_{s,-s}(\D)$.
\end{lemma}
\begin{proof}
Let 
$$X_s:=\left\{ (a_{jpkq}) : \, \norm{(a_{jpkq}) }_{X_s} := \sup_{j,p,k,q} (1+\max\{j,k\})^{2s+2}|a_{jpkq}|<\infty\right\}.$$
The embedding $X_{-s} \subset \mathcal{L}_{s,-s}$ (for $s\geq 0$) is shown to be continuous in \cite{MandacheInstability}. Thus, it is enough to show $\Gamma(\mathcal{G}(K,r_0\D))$ is totally bounded in $X_{0}$. By the preceding lemma, the embedding  $\Gamma(\mathcal{G}(K,r_0\D))  \subset X_{0}$ is clear. It remains to see that it admits a finite $\delta$-cover for every $\delta>0$ in the $X_{0}$ norm. 

Let $0<\delta< e^{-1}$ and let $\ell_{\delta}$ be the smallest integer satisfying that for $\ell\geq \ell_{\delta}$ the estimate
\begin{equation}\label{eqEllDeltaS}
C_K (1+\ell)^{3} r_0^{\ell}\leq \delta
\end{equation}
holds. 
 Let $\delta ' =\delta (1+ \ell_{\delta})^{-2}$, let
$$ Y_{\delta} := \delta' \Z \cap [-C_{K,r_0},C_{K,r_0}],$$
where the constant is given by Lemma \ref{lemHarmonicsUnderControl}, and
$$Y := \left\{ (b_{jpkq}) : b_{jpkq} \in Y_{\delta} \mbox{ whenever } \max\{j,k\} \leq \ell_{\delta} \mbox{ and } b_{jpkq}=0 \mbox{ otherwise}\right\}.$$
Then $Y$ has finitely many elements and for $(a_{jpkq})\in \Gamma(\mathcal{G}(K,r_0\D))$ there exists $(b_{jpkq})\in Y$ with $\norm{(a_{jpkq}-b_{jpkq})}_{X_{0}}\leq \delta$. Indeed, if $\max\{j,k\}>\ell_{\delta}$, then fix $b_{jpkq}=0$ and by Lemma \ref{lemHarmonicsUnderControl} and \rf{eqEllDeltaS} we get that
$$(1+\max\{j,k\})^{2}|a_{jpkq}|\leq C_K (1+\max\{j,k\})^{3}  r_0^{\max\{j,k\}}\leq \delta.$$
 If $\max\{j,k\}\leq \ell_{\delta}$ instead, since $|a_{jpkq}| <C_{K,r_0}$  (see Lemma \ref{lemHarmonicsUnderControl} again) we can choose $b_{jpkq}\in Y_{\delta}$ with $|a_{jpkq}-b_{jpkq}|\leq \delta'$. Thus, we obtain
$$(1+\max\{j,k\})^{2}|a_{jpkq}-b_{jpkq}|\leq (1+\max\{j,k\})^{2}  \delta ' \leq \delta.$$
\end{proof}

\begin{proof}[Proof of Theorem \ref{theoSharp}]
Let $K\geq 1$, $1\leq s \leq \infty$ and $r_0<1$. Let $\mathcal{F}\subset \mathcal{G}(K, r_0\D)$ be an $L^s$-stable family of conductivities for $\D$. By Lemma \ref{lemGCompactInL} the image $\Lambda(\mathcal{F})$ is totally bounded in $\mathcal{L}_{1/2,-1/2}(\D)$. If the recovery map is uniformly continuous in $\Lambda(\mathcal{F})$, then $\mathcal{F}$ must be totally bounded as well in $L^s$. 
By the Kolmogorov-Riesz Theorem (see \cite[Theorem 5]{HancheOlsenHolden}, for instance), for $1\leq s<\infty$ there exists a modulus of continuity $\omega$ such that $\mathcal{F}\subset \mathcal{G}(K,\D, s, \omega)$. 

Given any $s < p <\infty$, $L^p$-stability of $\mathcal{F}$ follows after proper interpolation with $L^\infty$. Given $0< p < s$, H\"older's inequality serves to find $L^p$-stability as well.
\end{proof}

\begin{remark}\label{remNotCompactlySupported}
To end this section, let us remark that the condition $r_0<1$ is crucial. Otherwise total boundedness of $\Lambda(\mathcal{F})$ in the previous proof is not satisfied.
\end{remark}
\begin{proof}
Indeed, in \cite[Theorem 4.9]{FaracoKurylevRuiz}, one shows that the family of conductivities
$$\gamma_R(z):= 1+ 2\chi_{R\D \setminus R^2\D}(z)$$
satisfies that
$$\langle (\Lambda_R-\Lambda_0) e^{ij\theta},e^{ik\theta}\rangle = \delta^k_j |j|m_{|j|},$$
with 
$$m_x = 4\frac{R^{2x}-R^{4x}}{4-3R^{2x} + 2R^{4x}}.$$
Note that $m_x$, depends only on $R^x$ and, therefore, its maximum in $x\in(0,\infty)$ does not depend on $0<R<1$. Moreover, the function has only one maximum, with $m_x$ tending to 0 both for  $x\to 0$ and $x\to \infty$, being increasing before the maximum is obtain, decreasing after that and continuous in $(0,\infty)$. For $0<R_n<1$, choosing an appropriate $R_n< R_{n+1}^2<1$ close enough to one and using \rf{eqNablaHarmonic}, we can obtain that
$$\max_{j} \frac{ \left|\langle (\Lambda_{R_n}-\Lambda_{R_{n+1}}) e^{ij\theta},e^{ij\theta}\rangle \right|}{\norm{f_{jp}}_{H^{1/2}(\partial\D)}^2} \approx  \frac{j_0}{j_0} = 1.$$
By induction we have obtained a sequence of conductivities $\gamma_n:=\gamma_{R_n}$ such that $\norm{\Lambda_{\gamma_n}-\Lambda_{\gamma_m}} \approx C$ for every $m\neq n$ and $\norm{\gamma_n}_{L^\infty}=2$. In particular, 
\begin{itemize}
\item $\mathcal{F}:= \{\gamma_n\}_{n=1}^\infty$ is (tautologically) $\infty$-stable for $\D$.
\item $\Lambda(\mathcal{F})$ is not totally bounded in $\mathcal{L}_{1/2,-1/2}(\D)$.
\item $\mathcal{F}$ is not totally bounded in $L^\infty(\D)$.
\end{itemize}
\end{proof}

\section{Background for the inverse problem}\label{secBackground}
Next we recall the reader the basic background of the present work. First we will give a word on quasiconformal mappings, then we will recall some related previous results in the stability issue, later we provide a short argument to reduce Theorem \ref{theoMainTheorem} to the case $\Omega=\D$, and finally we will introduce the notation and some results on CGOS from \cite{AstalaPaivarinta} and \cite{BarceloFaracoRuiz}.

\subsection{Quasiconformality}\label{secQuasiregular}
\begin{definition}\label{defFourier}
We define the (non-unitary) Fourier transform of a function $f \in L^1$ as
$$\widehat{f}(\xi)=\int_\C e_{\bar{\xi}}(z) f(z)\, dm(z), $$
where $e_{\xi}(z):= e^{i(\xi z + \overline{\xi z})}= e^{2 i \bar{\xi} \cdot z}$. The Fourier transform extends to $L^2$ as an isometry, and it can be defined in tempered distributions via the Parseval identity
$$\int g \overline{\widehat{f}}=\int \widehat{g} \overline{f},$$
see \cite[Chapter 4.1]{AstalaIwaniecMartin}.
\end{definition}

\begin{definition}\label{defBeurling}
We define the Beurling transform of $f\in L^2$ as $\widehat {\Beurling f} (\xi) = \frac{\bar \xi}{\xi}\widehat f (\xi)$. It coincides with the principal value Calder\'on-Zygmund convolution operator
$$\Beurling f(z) = - \frac1\pi \lim_{\varepsilon \to 0} \int_{|w-z|>\varepsilon} \frac{f(w)}{(z-w)^2}\,dm(w), $$
and therefore it extends to $L^p$ for $1<p<\infty$ (see \cite[Corollary 4.1.1 and Theorem 4.5.3]{AstalaIwaniecMartin}). 

We define the Cauchy transform of $f\in \mathcal{S}$ as 
$$\Cauchy f(z) = \frac1\pi \int_{\C} \frac{f(w)}{(z-w)}\,dm(w). $$
For $p\geq 2$, with a slight modification of the kernel one can extend $\Cauchy$ to $L^p$ modulo constants, obtaining that for $1<p<2$, $\frac1{p^*}=\frac1p-\frac12$ and $2<q<\infty$, the Cauchy transform acts as a bounded operator 
\begin{equation}\label{eqCauchyBounded}
\Cauchy : L^p \to L^{p^*}, \mbox{\quad\quad\quad\quad} 
\Cauchy : L^2 \to BMO \mbox{\quad\quad and \quad\quad} 
\Cauchy : L^{q} \to \dot{C}^{1-\frac2{q}}.
\end{equation}
For $f\in L^p$ with $1<p<\infty$, the identities
\begin{equation*}
\partial \Cauchy f = \Beurling f \mbox{\quad\quad and \quad\quad} \bar \partial \Cauchy f =  f
\end{equation*}
hold. Regarding compactly supported functions, for $1<p\leq 2$ and $2< q<\infty$ we have that 
\begin{equation}\label{eqCauchyCompactlyBounded}
\Cauchy\circ \chi_B: \left\{f\in L^p(B): \int_B f \, dm = 0 \right\} \to W^{1,p}(\C), \mbox{\quad\quad and \quad\quad} 
\Cauchy\circ\chi_B: L^q(B) \to W^{1,q}(\C),
\end{equation}
with operator norm depending on $p$ and the radius of the ball (see \cite[Section 4.3.2]{AstalaIwaniecMartin}).
\end{definition}

 The following result is known as the Measurable Riemann Mapping Theorem.
\begin{theorem}[see {\cite[Theorem 5.3.2]{AstalaIwaniecMartin}}]\label{theoMeasurableRiemann}
Let $\mu\in L^\infty$ be compactly supported in $\overline{\D}$, with $\norm{\mu}_{L^\infty}\leq \kappa$. Then there exists a unique $W^{1,2}_{loc}$ solution $f$ to the Beltrami equation
\begin{equation}\label{eqBeltramiEquation}
\begin{cases}
	\bar\partial f (z) = \mu (z) \partial f (z) \mbox{\quad\quad a.e. } z\in\C, \\
	f(z) - z = \mathcal{O}_{z\to \infty}(1/z),
\end{cases}
\end{equation}
which is called the {\em principal solution} to \rf{eqBeltramiEquation} and, moreover, it is homeomorphic.

In addition, 
\begin{equation}\label{eqPrincipalFromCauchy}
f= z + \Cauchy (\bar\partial f),
\end{equation}
and its $\bar \partial$-derivative can be obtained by the  Neumann series
\begin{equation}\label{eqPrincipalNeumann}
\bar\partial f = (I-\mu\Beurling)^{-1} (\mu) = \sum_{n=0}^{\infty} (\mu\Beurling)^n \mu = \mu + \mu\Beurling\mu + \mu \Beurling\mu\Beurling\mu +\cdots ,
\end{equation}
which is convergent in $L^p$ as long as $\kappa\norm{\Beurling}_{p,p}<1$.
\end{theorem}

Note that $\norm{\Beurling}_{2,2}=1<\frac1\kappa$, so there is an open  interval of exponents $p$ such that $\bar\partial f \in L^p$ which contains $2$. Thus, the Cauchy transform in \rf{eqPrincipalFromCauchy} can be computed by the expression in Definition \ref{defBeurling}. Although the range of convergence of the Neumann series is not clear yet, it is well known that 
\begin{equation*}
I-\mu\Beurling \mbox{\quad\quad is invertible for } p_\kappa '<p<p_\kappa
\end{equation*}
(see \cite{AstalaIwaniecSaksman}), where $p_\kappa = 1+\frac1\kappa = \frac{2K}{K-1} = p_K$ is called the critical exponent.

 Given $\kappa<1$, we will write  
\begin{equation}\label{eqKappaK}
K=\frac{1+\kappa}{1-\kappa}.
\end{equation}
In general, a $K$-quasiregular mapping is a function $f \in W^{1,2}_{loc}(\Omega)$  satisfying that $\bar\partial f = \mu \partial f $ with Beltrami coefficient essentially bounded by  $\norm{\mu}_{L^\infty}\leq \kappa$, where $K$ and $\kappa$ are related by \rf{eqKappaK}. We say that $f$ is $K$-quasiconformal if, in addition, it is a homeomorphism between domains. Quasiregular mappings have the following self-improvement property, in the form of a Caccioppoli inequality:
\begin{theorem}[see {\cite[Theorem 5.4.2]{AstalaIwaniecMartin}}]\label{theoSelfImprovement}
Let $\kappa<1$, let $q\in(p_\kappa', p_\kappa)$ and let $f\in W^{1,q}_{loc}(\Omega)$ for a planar domain $\Omega$ satisfy the distortion inequality
$$|\bar\partial f| \leq \kappa |\partial f|$$
almost everywhere. Then $f\in W^{1,p}_{loc}$ for every $p<p_\kappa$. In particular, $f$ is locally H\"older continuous, and for every $s \in (p_\kappa',p_\kappa)$, we have the Caccioppoli estimate
\begin{equation}\label{eqCaccioppoli}
\norm{\varphi Df}_{L^s}\leq C_{\kappa,s}\norm{f\nabla\varphi}_{L^s}
\end{equation}
whenever $\varphi$ is a  Lipschitz function compactly supported in $\Omega$.
\end{theorem}

\subsection{Previous results on the inverse problem}\label{secInverse}
In Section \ref{secDirect} we have seen that the forward map $\Lambda:\gamma\mapsto\Lambda_\gamma$ is bounded and continuous. Uniqueness for the inverse problem was solved by Astala and P\"aiv\"arinta in \cite{AstalaPaivarinta}.

The problem we face is stability. Barcel\'o, Faraco and Ruiz in \cite{BarceloFaracoRuiz}, showed continuity of the inverse mapping when we assume uniform H\"older continuity on the conductivities and the boundary measurements are taken in Lipschitz domains.
\begin{theorem*}[Stability for continuous conductivities]
Let $\Omega$ be a Lipschitz domain and let $0<s<1$. Then the family $\mathcal{G}(K,\Omega, \infty, t^s) = \left\{ \gamma \in\mathcal{G}(K,\Omega) : \norm{\gamma}_{\dot C^s} \leq 1 \right\}$ is  $L^\infty$-stable for $\Omega$,
and the recovery map has modulus of continuity 
$$\eta(\rho):= \frac1{\left|\log \left(\min \left\{\frac12,\rho \right\}\right)\right|^{b_s}}$$
for $\rho$ small enough.
\end{theorem*}
According to Mandache's result in \cite{MandacheInstability}, there is no hope to improve qualitatively the modulus of continuity obtained above.

Here, there is the extra assumption that the conductivities are (H\"older) continuous. Clop, Faraco and Ruiz addressed the question of stability for non-continuous conductivities with a priori Sobolev stability of fractional smoothness as close to zero as needed in \cite{ClopFaracoRuiz}. In that case, $L^\infty$ stability cannot be reached, but one gets $L^2$ stability nevertheless.
\begin{theorem*}[Stability for non-continuous conductivities]
Let $\Omega$ be a Lipschitz domain and let $0<s<1$. Then the family ${\mathcal{G}}(K,\Omega, 2, t^s) = \left\{ \gamma \in\mathcal{G}(K,\Omega) : \norm{\gamma}_{\dot B^s_{2,\infty}} \leq 1 \right\}$ is  $L^2$-stable for $\Omega$,
and the recovery map has modulus of continuity 
$$\eta(\rho):= \frac1{\left|\log \left(\min \left\{\frac12,\rho \right\}\right)\right|^{c \min\{s,1/2\}^2}}$$
for $\rho$ small enough.
\end{theorem*}

In \cite{FaracoRogers} the regularity conditions on $\Omega$ were severely reduced.

\subsection{Reduction to the unit disk}\label{secReduction}
In this section we explain how to reduce Theorem \ref{theoMainTheorem} to the case $\Omega=\D$. By rescaling, we can assume that $\overline{\Omega}\subset\D$. We want to check that the mapping 
$$
\begin{matrix}
\Lambda(\mathcal{G}(K, \Omega, p, \omega)) &  \to & \left(\mathcal{G}(K, \Omega, p, \omega), L^2\right)\\
\Lambda_\gamma	& \mapsto & \gamma 
\end{matrix}
$$
is uniformly continuous. By \cite[Theorem 3.6]{ClopFaracoRuiz}, every pair of conductivities $\gamma_1,\gamma_2\in\mathcal{G}(K, \Omega)$ satisfies the estimate 
$$\norm{ \Lambda_{\gamma_1}^{\partial\D} -\Lambda_{\gamma_2}^{\partial\D} }_{\mathcal{L}_{1/2,-1/2}(\D)}\leq C\norm{ \Lambda_{\gamma_1} -\Lambda_{\gamma_2} }_{\mathcal{L}_{1/2,-1/2}(\Omega)}.$$
In particular, the mapping
$$
\begin{matrix}
\Lambda(\mathcal{G}(K, \Omega, p, \omega)) & \to & \Lambda^{\partial\D}(\mathcal{G}(K, \Omega, p, \omega))\\
\Lambda_\gamma	& \mapsto & \Lambda_\gamma^{\partial\D}
\end{matrix}
$$
is Lipschitz continuous. Note that the reverse mapping is not continuous (see Remark \ref{remNotCompactlySupported}). Moreover, since $\Omega\subset \D$ it follows that $ \mathcal{G}(K,\Omega, p, \omega))\subset  \mathcal{G}(K,\D, p, \omega))$, so the inclusion mapping 
$$
\begin{matrix}
\Lambda^{\partial\D}(\mathcal{G}(K, \Omega, p, \omega)) & && \to && \Lambda^{\partial\D}(\mathcal{G}(K,\D, p, \omega))
\end{matrix}
$$
has norm 1. Thus, Theorem \ref{theoMainTheorem} is reduced to the following one.

\begin{theorem}\label{theoMainBisTheorem}
Let $K\geq 1$, let $2K< p <\infty$ and let $\omega$ be a modulus of continuity. Then $\mathcal{G}(K,\D, p, \omega)$ is an $L^2$-stable family of conductivities for  $\D$. 
\end{theorem}

\subsection{Complex geometric optics solutions}\label{secCGOS}
Next we present the main ingredients for the proof of Theorem \ref{theoMainBisTheorem}. We begin by recalling some results. In order to retain information of the whole DtN map, Astala and P\"aiv\"arinta define a parametrized family of solutions of the conductivity equation. Before that, the authors introduce Beltrami equation techniques to reach a good control of these solutions, so  they introduce the following parameterized family of functions (see \cite[Section 4]{AstalaPaivarinta}).
\begin{definition}\label{defFMu}
Given a Beltrami coefficient $\mu \in L^\infty$ supported in $\overline{\D}$ with $\kappa:=\norm{\mu}_{L^\infty}<1$ and a real number $2< p < p_\kappa$, we have that for each $k\in\C$ there exists a unique $f_\mu(\cdot,k)\in W^{1,p}_{loc}$ such that
\begin{equation}\label{eqBeltrami}
\overline{\partial} f_\mu(\cdot,k)= \mu \, \overline{\partial f_\mu(\cdot,k)}
\end{equation}
satisfying the asymptotics
\begin{equation}\label{eqBeltramiAssimptotics}
	f_\mu (z,k)= e^{i k z}M_\mu(z,k), \mbox{\quad with $M_\mu(\cdot,k)-1 = \mathcal{O}_{z\to\infty}(\frac{1}{z})$.}
\end{equation}
 Of course $f_\mu$ does not depend on our choice of $p$ and, since $M_\mu(\cdot, k)$ is continuous by the Sobolev embedding Theorem, we have that $f(\cdot, 0)\equiv1$.
We call $f_\mu$ the {\em Complex Geometric Optics Solution (CGOS) to the Beltrami equation} \rf{eqBeltrami}.

We define the {\it scattering transform} of $\mu$ as
\begin{equation*}
\tau_\mu(k) := \frac{1}{2\pi}\int_\D \partial_z\left( \overline{M_\mu(z,k)}-\overline{M_{-\mu}(z,k)}\right)\, dm(z).
\end{equation*}
\end{definition}

The definition of the scattering transform changes from one paper to the other. The one presented here agrees with \cite{AstalaPaivarinta}, as a short argument using Green's theorem and the Residue Theorem shows. Collecting the results in that paper, we have the following:

\begin{theorem}\label{theoConsequencesFMu}
Given $\mu \in L^\infty$  supported in $\overline{\D}$ with $\kappa:=\norm{\mu}_{L^\infty}<1$ and $2<p < p_\kappa$, the solution $f_\mu$ admits a representation
\begin{equation}\label{eqFMuRepresentation}
f_\mu(z,k)=e^{ik\varphi_\mu(z,k)},
\end{equation}
where $\varphi_\mu(\cdot,k)$ is a $K$-quasiconformal principal mapping (for $K$ defined in \rf{eqKappaK}). In particular $\varphi_{\mu}(\cdot,k)$ has uniform decay:
\begin{equation}\label{eqUniformDecay}
|z||\varphi_\mu(z,k) - z|\leq C_{\kappa}.
\end{equation}
 Moreover, we have that the Jost function $M_\mu$ satisfies that
\begin{equation}\label{eqJostSobolevBounded}
\norm{M_\mu(\cdot,k)-1}_{W^{1,p}(\C)}\leq e^{C_{\kappa,p}(1+|k|)},
\end{equation}
with $\real\left(\frac{M_\mu}{M_{-\mu}}\right)>0$ and for every $z\in\C$ the map
\begin{equation}\label{eqCInfty}
k \mapsto M_\mu(z,k) \mbox{\quad\quad is }C^\infty.
\end{equation}
In addition, 
\begin{equation}\label{eqTauBound}
\norm{\tau}_{L^\infty}\leq 1.
\end{equation}
\end{theorem}
\begin{proof}
See  {\cite[Theorems 7.1, 4.2, 4.3, and Proposition 6.3]{AstalaPaivarinta} and \cite[Proposition 2.6]{BarceloFaracoRuiz}}. The differentiability properties on the $k$ variable are shown in  the discussion following \cite[Lemma 5.3]{AstalaPaivarinta}.
\end{proof}

In Definition \ref{defFMu} we could consider the equation $\overline{\partial}f=\nu\partial f + \mu \overline{\partial f}$ with $\norm{|\nu|+|\mu|}_{L^\infty}<1$, and we would get the same results. However, to deal with the isotropic conductivity equation, we only need  the case presented above, and we only work with real-valued Beltrami coefficients. This comes from a natural bijection between the conductivity equation \rf{eqConductivity} and the Beltrami equation \rf{eqBeltrami} which, in the case of isotropic conductivities, reads as
$$\mu_\gamma:=\frac{1-\gamma}{1+\gamma},$$
which leads to a real-valued Beltrami coefficient. Then, the identity \rf{eqRelationfu} below links the solutions of both equations together. The interested reader may find a more detailed explanation on \cite[Chapter 16]{AstalaIwaniecMartin}. If it is clear from the context we will omit the subindex in $\mu$. Analogously, given a Beltrami coefficient $\mu$, we call $\gamma_\mu:=\frac{1-\mu}{1+\mu}$. Note that $\gamma_{\mu_\gamma}=\gamma$.

\begin{definition}
Given a conductivity coefficient $\gamma$ supported in $\overline{\D}$, and $k\in\C$ there is a unique complex valued solution $u_\gamma(\cdot,k) : \C\to\C$ of 
\begin{equation}\label{eqConductivityAssimptotics}
\begin{cases}
	\nabla \cdot \left(\gamma \nabla u_\gamma(\cdot,k)\right) \equiv 0, \\
	u_\gamma (z,k)= e^{i k z}\left(1 + R_\gamma(z,k) \right), \mbox{\quad with $R(\cdot,k)\in W^{1,p}(\C)$}
\end{cases}
\end{equation}
 for $p\in \left(2,p_\kappa\right)$, which we call {\it Complex Geometric Optics Solution to the conductivity equation} \rf{eqConductivityAssimptotics}. \end{definition}

\begin{proposition}[{see \cite[from (1.14) to (1.17)]{AstalaPaivarinta}}]\label{propoCGOS}
The CGOS to \rf{eqConductivityAssimptotics} is given by
\begin{equation}\label{eqRelationfu}
u_\gamma =\real(f_\mu)+i\,\imag(f_{-\mu})=\frac{1}{2} \left(f_\mu+f_{-\mu}+\overline{f_{\mu}}-\overline{f_{-\mu}}\,\right),
\end{equation}
where $\mu=\mu_\gamma$. In addition, $u_\gamma$ satisfies the equation
\begin{equation}\label{eqTransportTau}
\partial_{\overline{k}} u_\gamma(z,k)=-i \tau_\mu(k) \overline{u_\gamma}(z,k).
\end{equation}
\end{proposition}

\begin{remark}\label{remDifferentiabilityOfFunctions}
First note that the differentiability properties in Theorem \ref{theoConsequencesFMu} extend to $f_\mu$ and $u_\gamma$, so the derivative in \rf{eqTransportTau} can be understood in the classical sense. 
By the Abstract Monodromy Theorem (see \cite[Theorem 15.1.3]{ConwayComplexII}, for instance) there is a determination of the logarithm of $M_\mu$ which coincides with $ik(\varphi_\mu-z)$. Thus, the differentiability property of $M_\mu$ with respect to the second variable extends to the product $k\varphi_\mu$ as well.

On the other hand, note that $u_\gamma$ is close to $e^{i zk}$ for $z$ big. Indeed, since 
$$R_\gamma(z,k)=\frac12 \left(M_\mu(z,k)+M_{-\mu}(z,k)+ e^{-i(kz+\bar{k}\bar{z})}\left(\overline{M_\mu(z,k)}-\overline{M_{-\mu}(z,k)}\,\right)\right) -1,$$ 
we get a uniform control independent of $\gamma$ of the Sobolev norm
$$\norm{R_\gamma(\cdot,k)}_{W^{1,p}(\C)}\leq C(1+|k|)\left(\norm{M_\mu(\cdot,k)-1}_{W^{1,p}(\C)}+\norm{M_{-\mu}(\cdot,k)-1}_{W^{1,p}(\C)}\right) \leq e^{C(1+|k|)}.$$
By the Sobolev embedding Theorem, $R_\gamma(z,k)\to 0$ as $z\to \infty$. This tells us that the complex geometric optics solution spins as $z$ approaches infinity in the direction of $\bar{k}$ and it blows up in the direction of $-i\bar k$, but when $k\to 0$ this is done slower and slower.
\end{remark}

\begin{theorem}[{see \cite[Theorem 3.12]{BarceloFaracoRuiz}}]\label{theoMMuDerivativeTauDerivative}
Let $\mu \in L^\infty$ be supported in $\overline{\D}$ with $\kappa:=\norm{\mu}_{L^\infty}<1$ and let $2<p<p_\kappa$. There exists a constant $C$ depending on $\kappa$ and $p$ such that
\begin{equation}\label{eqMMuDerivative}
\norm{\nabla_k M_\mu(\cdot,k)}_{W^{1,p}(\C)}\leq e^{C(1+|k|)},
\end{equation}
and
\begin{equation}\label{eqTauDerivative}
|\nabla_k \tau_\mu(k)|\leq e^{C(1+|k|)}.
\end{equation}
\end{theorem}

\section{Decay of the complex geometric optics solution}\label{secDecay}
Next we focus on the sub-exponential behavior of the complex geometrics optics solution depending on the modulus of continuity of the solutions. To start, we note some properties of the modulus of continuity.

\subsection{Interaction of the modulus of continuity with operators.}\label{secInteractionCGOS}

\begin{lemma}\label{lemTranslationInvariant}
Let $T$ be a translation-preserving linear operator mapping $L^p$ to $L^q$ for certain $0<p,q\leq \infty$, with norm $\norm{T}_{p,q}$. The following holds:
$$\norm{Tf}_{B^\omega_q} \leq \norm{T}_{p,q} \norm{f}_{B^\omega_p}.$$
\end{lemma} 
\begin{proof}
Being translation invariant can be written as $T(\tau_y f)=\tau_y (Tf)$. Thus, for every $t>0$
\begin{equation*}
\modulus{q}{Tf}(t) = \sup_{|y|\leq t} \norm{Tf-\tau_y(Tf)}_{L^q}= \sup_{|y|\leq t} \norm{ Tf- T(\tau_yf)}_{L^q}=\sup_{|y|\leq t} \norm{ T(f-\tau_yf)}_{L^q}.
\end{equation*}
It follows that
\begin{equation}\label{eqModulusBoundedTranslationInvariant}
\modulus{q}{Tf}(t) \leq \norm{T}_{p,q} \modulus{p}{f} (t).
\end{equation}
\end{proof}

Note that the Beurling transform is translation invariant. Indeed, it follows from the definition of the Fourier transform that given a function $f\in L^1_{loc}$ and complex numbers $\xi,k\in \C$, we have that 
\begin{equation}\label{eqFourierTranslation}
\widehat{e_{k} f}(\xi)= \tau_{-\bar{k}}\widehat{f}(\xi) \mbox{\quad \quad and \quad \quad} \widehat{\tau_k f}(\xi) = e_{\bar{k}} (\xi) \widehat{f}(\xi).
\end{equation}
The same happens with the Cauchy transform.

Next we  get a quantitative version of \cite[Theorem 4]{Pego}.
\begin{lemma}\label{lemFourier}
Given $f\in L^p$, $1\leq p\leq 2$ and $R>0$, then 
\begin{equation}\label{eqFourierKeepingModulus}
\norm{ \hat f \chi_{|\xi|> R} }_{L^{p'}} \leq C(p) \modulus{p}{f}\left(\frac1R\right) .
\end{equation}
\end{lemma}
\begin{proof}
The first step is showing that for $1\leq q \leq \infty$ and $|\xi|> R$, 
\begin{equation}\label{eqSomethingFixed}
	\left(\fint_{|y|\leq \frac{1}{R}} \left| e_{\bar{\xi}}(y)  -1\right|^{q} \, dm(y) \right)^{\frac{1}{q}} \approx 1,
\end{equation}
with constants not depending on $p$. Note that $\norm{ e_{\bar{\xi}} -1}_{L^{\infty}} < 4$, so we need to prove that  
\begin{equation*}
	\fint_{|y|\leq \frac{1}{R}} \left| e_{\bar{\xi}}(y)  -1\right| \, dm(y)  \gtrsim 1,
\end{equation*}
and \rf{eqSomethingFixed} will follow by Jensen's inequality.

Indeed, for any given $\xi \in \C$, changing variables we obtain that
\begin{align*}
\fint_{|y|\leq \frac{1}{R}} \left| e_{\bar{\xi}}(y) -1\right| \, dm(y) 
	& \gtrsim  \fint_{\frac{1}{2R} \leq |y|\leq \frac{1}{R}} \left| e_{\bar{\xi}}(y)  -1\right| \, dm(y) 
		\approx  R^2 \int_{\frac{1}{2R}}^{\frac{1}{R}}\int_0^{2\pi} \left| e_{\bar{\xi}}(re^{i\theta}) -1\right| \, d\theta r dr	\\
	& \gtrsim \inf_{r\in \left(\frac{1}{2R},\frac{1}{R}\right)} \left|\left\{ \theta\in(0,{2\pi}) :  \left| e^{2 i (r \bar{\xi} \cdot e^{i\theta})} -1\right|>0.4 \right\}\right|,
\end{align*}
where $\bar{\xi} \cdot e^{i\theta}$ stands for the usual scalar product of two vectors. Thus,
\begin{align*}
\fint_{|y|\leq \frac{1}{R}} \left| e_{\bar{\xi}}(y)  -1\right| \, dm(y) 
	& \gtrsim \inf_{r\in \left(\frac{1}{2R},\frac{1}{R}\right)} \left|\left\{ \theta\in(0,{2\pi}) :  \dist(2 r \bar{\xi} \cdot e^{i\theta}, 2\pi \Z)>0.5 \right\}\right|.
\end{align*}
Writing $\bar{\xi}=|\xi| e^{it}$, we have that
\begin{align*}
 \left|\left\{ \theta\in(0,{2\pi}) :  \dist(2r \bar{\xi} \cdot e^{i\theta}, 2\pi \Z)> 0.5 \right\}\right| 
 	& =  \left|\left\{ \theta\in(0,{2\pi}) :  \dist(|r \xi| \cos(\theta-t) , \pi \Z)> 0.25 \right\}\right| \\
 	& =  \left|[0,2\pi] \cap \cos^{-1}\left( \bigcup_{j} \left(\frac{\pi j +0.25}{|r \xi|},\frac{\pi j + \pi - 0.25}{|r \xi|}\right)\right)\right|.
\end{align*}
But the cosine is a contraction and, therefore, 
\begin{align*}
 \left|\left\{ \theta\in(0,{2\pi}) :  \dist(2 r \bar{\xi} \cdot e^{i\theta},  2\pi \Z)> 0.5 \right\}\right| 
 	& >   \sum_{j} \left| \left(  \frac{\pi j + 0.25}{|r \xi|},  \frac{\pi j + \pi - 0.25}{|r \xi|} \right) \cap [-1,1] \right| \\
 	& \geq  \frac2{|r \xi|} \sum_{j\geq 0 : \pi j + 0.5 \leq |r\xi|} \left| \left( \pi j + 0.25, \pi j + 0.5 \right) \cap [0,|r \xi|] \right| .
\end{align*}
Whenever $j$ is in the summation range, the measure $\left| \left( \pi j + 0.25, \pi j + 0.5 \right) \cap [0,|r \xi|] \right| =0.25$. For every $|\xi|>R$ and $r>\frac{1}{2R}$, we have that $|r \xi|> 0.5$ and, thus, the number of elements in the sum bounded above and below by constants depending linearly on $|r\xi|$. Namely,
\begin{align*}
 \left|\left\{ \theta\in(0,{2\pi}) :  \dist(2 r \bar{\xi} \cdot e^{i\theta},  2\pi \Z)> 0.5 \right\}\right| 
	&  > \frac{2 \cdot 0.25}{|r \xi|}  \#\{j\geq 0 : \pi j + 0.5 \leq |r\xi|\} 
	 	\gtrsim 1,
\end{align*}
establishing \rf{eqSomethingFixed}.

By \rf{eqSomethingFixed}, for $p>1$ we have that 
\begin{align*}
\norm{ \hat f \chi_{|\xi|> R} }_{L^{p'}} 
	& \approx \left( \int_{|\xi|> R} \fint_{|y|< \frac1R}\left| e_{\bar{\xi}}(y)  - 1 \right|^{p'} \, dm(y) |\hat f(\xi)|^{p'} \, dm(\xi)\right)^\frac{1}{p'} \\
	& \leq \left( \sup_{|y|< \frac1R} \int_{|\xi|> R} \left|( e_{\bar{y}}(\xi)  - 1 )\hat f(\xi)\right|^{p'} \, dm(\xi)\right)^\frac{1}{p'} .
\end{align*}
Using \rf{eqFourierTranslation} and increasing the domain of integration we get
\begin{align*}
\norm{ \hat f \chi_{|\xi|>R} }_{L^{p'}} 
	& \leq  \sup_{|y|< \frac1R} \norm{ \widehat{\tau_y f - f}}_{L^{p'}} .
\end{align*}
By the Hausdorff-Young inequality, 
\begin{align*}
\norm{ \hat f \chi_{|\xi|>R} }_{L^{p'}} 
	& \leq C_p \sup_{|y|< \frac1R} \norm{ \tau_y f - f }_{L^{p}} 
		= C_p\, \modulus{p}{f}\left(\frac1R\right),
\end{align*}
that is, \rf{eqFourierKeepingModulus}. The case $p=1$ can be shown mutatis mutandis.
\end{proof}

%
 
\subsection{Solution to the linear equation}\label{secLinear}
Following the sketch of \cite[Section 5.1]{ClopFaracoRuiz}, we derive properties of the unique quasiconformal solution $\psi_k :\C\to\C$  to the linear equation
\begin{equation}\label{eqBeltramiLinearEquation}
\begin{cases}
	\bar\partial \psi_k (z) = - \frac{\bar{k}}{k} e_{-k}(z) \mu (z) \partial \psi_k (z), \\
	\psi_k(z) - z = \mathcal{O}_{z\to \infty}(1/z),
\end{cases}
\end{equation}
where  $\mu\in L^\infty$ is supported in $\D$ and $\norm{\mu}_{L^\infty}\leq \kappa<1$. In particular, we want to derive information on the decay of the $L^\infty$ norm of $\psi_k - Id$ when $k$ tends to infinity from the modulus of continuity of $\mu$.

Let us adapt the notion of the families of conductivities in Definition \ref{defFamiliesOfModuli} to the Beltrami equation context.
\begin{definition}\label{defFamiliesOfModuliRevisited}
Let $\mu \in L^\infty$ real valued and supported in $\overline{\D}$ with $\norm{\mu}_{L^\infty}\leq \kappa <1$. Let $1< p <\infty$ and assume  that for a certain modulus of continuity  $\omega$ we have the pointwise bound $\modulus{p}{\mu}\leq \omega$. Then we say that $\mu\in\mathcal{M}(\kappa, p,\omega)$. 
\end{definition}

\begin{proposition}\label{propoUpsilon}
Let  $q>2$, let  $\kappa < 1$, let $\omega$ be a modulus of continuity and let $\mu\in\mathcal{M}(\kappa,q,\omega)$. For every $k\in \C$, let $\psi_k$ be the unique quasiconformal solution to \rf{eqBeltramiLinearEquation}. Then, there exits a modulus of continuity $\upsilon$ depending only on $\kappa$, $q$ and $\omega$ such that
$$\norm{\psi_k-Id}_{L^\infty}\leq \upsilon(|k|^{-1}).$$
\end{proposition}

To show the proposition above we adapt the Neumann series expression \rf{eqPrincipalNeumann} to the parameterized Beltrami equation \rf{eqBeltramiLinearEquation}. Namely,
\begin{equation*}
\bar \partial \psi_k = \sum_{n=0}^\infty  \left(\frac{-\bar{k}}{k} e_{-k} \mu \Beurling\right)^n \left(\frac{-\bar{k}}{k}e_{-k}\mu\right)= \sum_{n=0}^\infty \left(\frac{-\bar{k}}{k}\right)^{n+1} \left(e_{-k} \mu \Beurling\right)^n \left(e_{-k}\mu\right),
\end{equation*}
where the series has, at least, $L^p$ convergence for $p$ in a certain open interval containing $2$.
We want to study the asymptotic behavior on $k$, and, thus, we rewrite this expression as
\begin{equation}\label{eqNeumannModified}
\bar \partial \psi_k = \sum_{n=0}^\infty G_{n,k}.
\end{equation}

To understand better the terms $G_{n,k}$, let
$$f_0(z) :=\mu(z)$$
and
\begin{equation}\label{eqFnDefinition}
f_n(z) := \mu(z) T_n (f_{n-1})(z) ,
\end{equation}
with $T_n$ defined as 
\begin{equation}\label{eqTnDefinition}
T_n(\varphi) = e_{nk} \Beurling (e_{-nk} \varphi).
\end{equation}

With a quick induction argument we show that
\begin{equation}\label{eqHypotesiInduction}
 f_{n}  = e_{(n+1)k}(e_{-k} \mu \Beurling)^{n} (e_{-k}\mu).
\end{equation}
Indeed, the case $n=0$ is trivial. Let $n>0$. We only need to show that if \rf{eqHypotesiInduction} holds for $n-1$, then it holds for $n$.
Thus, assuming that \rf{eqHypotesiInduction} holds for $f_{n-1}$, using \rf{eqFnDefinition} and \rf{eqTnDefinition}, we obtain
\begin{align*}
f_n
	 = \mu T_n (f_{n-1}) 
	 = \mu e_{nk} \Beurling (e_{-nk} e_{nk}(e_{-k} \mu \Beurling)^{n-1} (e_{-k}\mu)) 
	 = e_{(n+1)k} (e_{-k} \mu \Beurling)^{n} (e_{-k}\mu).
\end{align*}
Thus, identity \rf{eqHypotesiInduction} holds for every $n\geq 0$, establishing the following:
\begin{equation}\label{eqTermOfNeumann}
G_{n,k}(z) = \left(\frac{-\bar{k}}{k}\right)^{n+1} e_{-(n+1)k}(z) f_n(z).
\end{equation}

Next we list some properties of the operator $T_n$. By Definition \ref{defBeurling} and \rf{eqFourierTranslation}, we have that  
\begin{equation*}
\widehat{T_n \varphi}(\xi) =(e_{nk} \Beurling (e_{-nk} \varphi))\,\widehat{\,}\,(\xi) =(\Beurling (e_{-nk} \varphi))\,\widehat{\,}\,(\xi + n\bar{k}) =\frac{\bar{\xi}+nk}{\xi+n\bar{k}} \widehat {e_{-nk} \varphi} (\xi+ n\bar{k})=\frac{\bar{\xi}+nk}{\xi+n\bar{k}}  \widehat {\varphi} (\xi).
\end{equation*}
Moreover, $T_n$ is translation invariant, as the reader can check using either the same property of the Beurling transform or the Fourier multiplier definition just mentioned.
The boundedness of the Beurling transform in $L^p$ can be brought to $T_n$ as well, since
\begin{equation}\label{eqBoundTnLp}
\norm{T_n f}_{L^p}=\norm{\Beurling(e_{-nk} f)}_{L^p}\leq \norm{\Beurling}_{p,p}\norm{e_{-nk} f}_{L^p} = \norm{\Beurling}_{p,p}\norm{f}_{L^p}.
\end{equation}

\begin{lemma}\label{lemModulusIOfIterates}
Let $1<p<q<\infty$ and let $\frac1{r}+\frac1q=\frac1p$. Then
\begin{equation}\label{eqModulusFn}
\norm{f_n}_{\dot B^\omega_p}\leq \left(2 \pi\right)^\frac1r \kappa^n M(p,q)^n \norm{\mu}_{\dot B^\omega_q} ,
\end{equation}
where $M(p,q):=\norm{\Beurling}_{r,r} + \norm{\Beurling}_{p,p}$. 
\end{lemma}

\begin{proof}
We begin by studying the modulus of continuity of a product. Let $g,h\in L^1_{\rm loc}$ and let $t>0$. Then, using the expression
$h(z-y)g(z-y)-h(z)g(z) = (h(z-y)-h(z))g(z-y)+h(z)(g(z-y)-g(z))$, we get
\begin{align*}
\modulus{p}{(hg)}(t)
	&  \leq \sup_{|y|<t}\norm{(\tau_yh -h)\tau_y g}_{L^p}+  \sup_{|y|<t}\norm{h (\tau_y g -g)}_{L^p}
\end{align*}
for $t>0$. Whenever $\frac1p=\frac 1{q_1}+\frac 1{r_1} =\frac 1{q_2}+\frac 1{r_2}$, using the H\"older inequality in each term we get
\begin{align}\label{eqModulusProduct}
\modulus{p}{(hg)}(t)
	&  \leq \modulus{q_1}{h}(t) \norm{g}_{L^{r_1}} + \modulus{q_2}{g}(t) \norm{h}_{L^{r_2}},
\end{align}
and dividing by $\omega(t)$ and taking supremum in $t>0$, we obtain the generalized Leibniz' rule
\begin{align}\label{eqNormModulusProduct}
\norm{hg}_{\dot B^\omega_p}
	&  \leq \norm{h}_{\dot B^\omega_{q_1}} \norm{g}_{L^{r_1}} + \norm{g}_{\dot B^\omega_{q_2}}\norm{h}_{L^{r_2}}.
\end{align}

To show \rf{eqModulusFn} we will argue by induction. For the case $n=0$, note that
\begin{align*}
\nonumber \modulus{p}{f_0}(t)
	& =\modulus{p}{\mu}(t) \leq \sup_{|z|<t} \norm{ \mu - \tau_z \mu}_{L^p} \leq  \sup_{|z|<t} \norm{ \mu - \tau_z \mu}_{L^q \left(\D\cup (\D+z)\right)} \left(2 \pi\right)^\frac1r 
	 \leq   (2\pi)^\frac1r \modulus{q}{\mu}(t). 
\end{align*}

Let us assume that \rf{eqModulusFn} is shown for $n-1$. We need to show that it holds for $n$. Recall that $f_n=\mu T_n(f_{n-1})$.  Using \rf{eqNormModulusProduct} and Lemma \ref{lemTranslationInvariant} combined with \rf{eqBoundTnLp}, we get
\begin{align*}
\norm{f_n}_{\dot B^\omega_p} 
	& \leq \norm{\mu}_{\dot B^\omega_q} \norm{T_n(f_{n-1})}_{L^{r}} + 
\kappa \norm{T_n(f_{n-1})}_{\dot B^\omega_p} \\
	& \leq \norm{\Beurling}_{r,r} \norm{f_{n-1}}_{L^{r}} 
\norm{\mu}_{\dot B^\omega_q} + \kappa \norm{\Beurling}_{p,p}
\norm{f_{n-1}}_{\dot B^\omega_p}.
\end{align*}
The $L^r$ norm of $f_{n-1}$ can be bounded by \rf{eqBoundTnLp}:
\begin{equation}\label{eqLrBound}
\norm{f_{n-1}}_{L^{r}} \leq \kappa \norm{\Beurling}_{r,r} \norm{f_{n-2}}_{L^{r}}\leq \cdots \leq \kappa^{n-1}\norm{\Beurling}_{r,r}^{n-1} \norm{f_{0}}_{L^{r}}=\kappa^{n-1}\norm{\Beurling}_{r,r}^{n-1} \norm{\mu}_{L^{r}}.
\end{equation}
Moreover, using the essential supremum bound of $\mu$, it is clear that
\begin{equation}\label{eqNormMuLr}
\norm{\mu}_{L^{r}}\leq  \kappa \pi^{\frac1r}.
\end{equation}

Thus, using \rf{eqLrBound}, \rf{eqNormMuLr} and the hypothesis of induction, we obtain
\begin{align*}
\norm{f_n}_{\dot B^\omega_p} 
	& \leq \norm{\Beurling}_{r,r} \kappa^{n-1}\norm{\Beurling}_{r,r}^{n-1} \kappa \pi^\frac1r \norm{\mu}_{\dot B^\omega_q} + \kappa \norm{\Beurling}_{p,p} \left(2 \pi\right)^\frac1r \kappa^{n-1} M(p,q)^{n-1} \norm{\mu}_{\dot B^\omega_q} \\
	& \leq \left(2 \pi\right)^\frac1r \kappa^n  \left(\norm{\Beurling}_{r,r} M(p,q)^{n-1}  +  \norm{\Beurling}_{p,p}  M(p,q)^{n-1} \right)
	\norm{\mu}_{\dot B^\omega_q}\\
	& =  \left(2 \pi\right)^\frac1r \kappa^n  M(p,q)^{n}\norm{\mu}_{\dot B^\omega_q} .
\end{align*}
\end{proof}

\begin{lemma}\label{lemBreakPsi}
Let ${N}\in\N$, let $1<s<\infty$,  where $\kappa\norm{\Beurling}_{s,s}<1$ and let $k\in \C$.  There exists a decomposition $\bar\partial \psi_k=g_k+h_k$ such that the following holds:
\begin{itemize}
\item $\norm{h_k}_{L^s} \leq  \frac{\kappa \pi^\frac1s }{1- \kappa \norm{\Beurling}_{s,s}} \left( \kappa \norm{\Beurling}_{s,s}\right)^{{N}+1}  $.
\item $\norm{g_k}_{L^s}\leq  \frac{\kappa \pi^\frac1s}{1- \kappa \norm{\Beurling}_{s,s}}$.
\item If $1<p\leq 2$, $p<q<\infty$ and $0<R<|k|$, then 
$$\norm{\widehat{g_k}}_{L^{p'}(\D_R)} \leq \frac{ \left(2\pi \right)^\frac1r M(p,q)^{{N}}}{1-\kappa}   \norm{\mu}_{  \dot B^\omega_q}\omega\left(\frac1{|k|-R}\right), $$
where $M(p,q):=\norm{\Beurling}_{r,r} + \norm{\Beurling}_{p,p}$, with $\frac1p=\frac1q+\frac1r$.
\end{itemize}
\end{lemma}

\begin{proof}
Take the partial sum
$$g_k:=\sum_{n=0}^{{N}} G_{n,k},$$
with $G_{n,k}= \left(\frac{-\bar{k}}{k}\right)^{n+1} e_{-(n+1)k} f_n$ as in \rf{eqTermOfNeumann}. Then, by  \rf{eqLrBound} and \rf{eqNormMuLr}, whenever $n\in\N$ we have that 
$$\norm{ G_{n,k}}_{L^s}=\norm{ f_n}_{L^s}\leq \left(\kappa \norm{\Beurling}_{s,s}\right)^n \kappa \pi^\frac1s.$$
Thus, the first and the second properties come from the geometric series formula. 


From \rf{eqFourierTranslation} we know that the Fourier transform of $G_{n,k}$ is
$$\widehat{G_{n,k}}(\xi) = \left(\frac{-\bar{k}}{k}\right)^{n+1} \widehat{f_n}(\xi - (n+1)\bar{k}).$$
Thus,
\begin{align*}
\norm{\widehat{G_{n,k}}}_{L^{p'}(\D_R)}
	& \leq \left(\int_{|\xi|<R} \left|\left(\frac{-\bar{k}}{k}\right)^{n+1} \widehat{f_n}(\xi - (n+1)\bar{k})\right|^{p'} \, dm(\xi)\right)^{\frac1{p'}}\\
	& \leq \left(\int_{|\zeta + (n+1)\bar{k}|<R} \left|\widehat{f_n}(\zeta)\right|^{p'} \, dm(\zeta)\right)^{\frac1{p'}}.
\end{align*}
Note that $|\zeta + (n+1)\bar{k}|<R$ implies that $|\zeta|> (n+1)|k| - R$. Thus, using lemmas \ref{lemFourier} and \ref{lemModulusIOfIterates} (and the fact that $|k|>R$ to have a meaningful expression for $n\geq 0$), we get
\begin{align*}
\norm{\widehat{G_{n,k}}}_{L^{p'}(\D_R)}
	& \leq    \norm{\widehat{f_n}}_{L^{p'}\left(\D_{(n+1)|k| - R}^c\right)} 
		 \lesssim_p  \modulus{p}{f_n}\left(\frac1{(n+1)|k| - R}\right)\\ 
	& \leq \left(2\pi \right)^\frac1r  \kappa^n M(p,q)^n  \norm{\mu}_{ \dot B^\omega_q}\omega\left(\frac1{(n+1)|k| - R}\right).
\end{align*}
Since the modulus of continuity is increasing, 
\begin{align*}
\norm{\widehat{g_k}}_{L^{p'}(\D_R)} 
	& \leq \sum_{n=0}^{{N}} \norm{\widehat{G_{n,k}}}_{L^{p'}(\D_R)}
		 \leq  \left(2\pi \right)^\frac1r M(p,q)^{{N}}   \norm{\mu}_{ \dot  B^\omega_q}\omega\left(\frac1{|k| - R}\right)\sum_{n=0}^{{N}}  \kappa^n \\
	& \leq  \left(2\pi \right)^\frac1r M(p,q)^{{N}}   \norm{\mu}_{  \dot B^\omega_q}\omega\left(\frac1{|k| - R}\right)\frac{1}{1-\kappa}.
\end{align*}

\end{proof}

We can compute the Cauchy transform of a function supported in $\overline{\D}$ using a cut-off kernel $\Phi$ such that
$$\mbox{$\Phi(z)=\frac{1}{\pi z}$ when $|z|<4$,  $\Phi$ vanishes in $\D_5^c$ and $\Phi \in C^\infty(\C\setminus \{0\})$}.$$
Indeed, for every $z\in \D$ and every function $A$ supported in $\overline{\D}$ with integrability of order greater than $2$,  
\begin{equation}\label{eqCauchyTruncated}
\Cauchy A (z) =\frac1\pi \int_{\D} \frac{A(w)}{z-w}\, dm(w) = \int_{\D} A(w)\Phi(z-w) \, dm(w)=A*\Phi (z).
\end{equation}

\begin{proposition}\label{propoCompactKernelIsGood}
The kernel $\Phi \in L^p$ for $1\leq p<2$. Moreover, 
$$\norm{\Phi}_{B^\varepsilon_{p,q}}< C_{s,p,q}$$
for every $1\leq p<2$, $\varepsilon < \frac2p-1$, and $0<q \leq \infty$.
\end{proposition}
\begin{proof}
See  \cite[Lemma 2.3.1/1]{RunstSickel} for the absolute value case. The proof for $\Phi$ runs parallel to that one.
\end{proof}

\begin{proof}[Proof of Proposition \ref{propoUpsilon}]
Let $q>2$ be given and let $\mu \in \mathcal{M} (\kappa,q,\omega)$. Recall that $\psi_k$ stands for the unique quasiconformal solution to \rf{eqBeltramiLinearEquation}, which, by Theorem \ref{theoMeasurableRiemann} is given by 
\begin{equation}\label{eqPsiAsCauchy}
\psi_k(z)=z+\Cauchy(\bar\partial \psi_k),
\end{equation}
and $\bar\partial\psi_k$ can be found using the Neumann series \rf{eqNeumannModified}.

 We want to define a function $\upsilon :\R_+ \to \R+$ depending only on $\kappa$, $q$ and $\omega$ such that
$$\norm{\psi_k-Id}_{L^\infty}\leq \upsilon(|k|^{-1}),$$
and then show that
$$\lim_{t\to0} \upsilon(t)=0.$$

First of all, note that $\psi_k-Id$ is holomorphic outside the unit disk, vanishing at infinity. Moreover, $\psi_k-Id$ is continuous everywhere. Using both the maximum principle in $\D^c$ and the continuity in $\overline{\D}$, we obtain that
\begin{equation}\label{eqReductionSupport}
\norm{\psi_k-Id}_{L^\infty} = \norm{\psi_k-Id}_{L^\infty(\D)}.
\end{equation}

Consider two parameters ${N} \in\N$ and $\delta > 0$ to be fixed along the proof and $R=\frac{|k|}{2}>0$. We will use ${N}$ to cut the tail in the Neumann series, $\delta$ to use an approximation of the identity smoothing out the kernel of the Cauchy transform and we will use the radius $R$ to separate the low and high frequencies to use Lemma \ref{lemBreakPsi}. 

Let $g_k$ and $h_k$ be the functions given in Lemma \ref{lemBreakPsi} (depending on ${N}$) with $\bar\partial \psi_k=h_k+g_k$. Then, by \rf{eqPsiAsCauchy} we have that
$$\norm{\psi_k-Id}_{L^\infty (\D)}=\norm{\Cauchy \bar\partial \psi_k}_{L^\infty (\D)}\leq \norm{\Cauchy h_k}_{L^\infty (\D)}+\norm{\Cauchy g_k}_{L^\infty (\D)}.$$
Both $h_k$ and $g_k$ are supported in $\overline{\D}$. Thus, we can use the truncated kernel in \rf{eqCauchyTruncated}.  Combining this and identity \rf{eqReductionSupport}, we get that
\begin{equation}\label{eqBreakPsi}
\norm{\psi_k-Id}_{L^\infty} \leq \norm{\Phi * h_k}_{L^\infty (\D)}+\norm{\Phi * g_k}_{L^\infty (\D)}.
\end{equation}

First we address the term corresponding to the Neumann tail $h_k$. Chose $s = s(\kappa) >2$ such that $\norm{\Beurling}_{s,s} = \frac{1}{\sqrt{\kappa}}$. Then, by the Young inequality, the boundedness of the Cauchy kernel in $L^{s'}_{loc}$ (see Proposition \ref{propoCompactKernelIsGood})  and the first property in Lemma \ref{lemBreakPsi}, we get
\begin{equation}\label{eqBoundH}
\norm{\Phi * h_k}_{L^\infty} \leq \norm{\Phi}_{L^{s'}} \norm{h_k}_{L^s} \leq  C(s') C(s)\frac{\kappa}{1-\sqrt{\kappa}} \left(\kappa \norm{\Beurling}_{s,s} \right)^{{N}+1} = C(\kappa) \kappa^{\frac{{N}}2}.
\end{equation}

Next we check the main term in \rf{eqBreakPsi}. Here, the idea is to use Plancherel's identity to study the $L^1$ norm of $\widehat{\Phi}\widehat{g_k}$. We will use Lemma \ref{lemBreakPsi} for the low frequencies and Lemma \ref{lemFourier} for the high frequencies, but we cannot control at the same time the $L^p$ norm of $\widehat{\Phi}$ and the $L^{p'}$ norm of $\widehat{g_k}$ because the Fourier transform is bounded from $L^p$ to $L^{p'}$ only for $1\leq p \leq 2$ and $\Phi$ is not $2$-integrable. We need an extra term, which we will add using an approximation of the identity. 

Consider a test function $\varphi \in C^\infty_c(\C)$ with $\int \varphi \, dm=1$ and define the approximation of the identity $\varphi_\delta(y):=\frac{1}{\delta^2} \varphi \left(\frac{y}{\delta}\right)$ for $y\in\C$ and  
$$\Phi_\delta := \Phi * \varphi_\delta.$$
By the triangle inequality, we obtain
\begin{equation}\label{eqBreakPsiDelta}
\norm{\Phi * g_k}_{L^\infty} \leq \norm{(\Phi-\Phi_\delta) * g_k}_{L^\infty} + \norm{\Phi_\delta * g_k}_{L^\infty} .
\end{equation}

Let us study the first term in the right-hand side of \rf{eqBreakPsiDelta}. By the Young inequality again, we have that
$$\norm{(\Phi-\Phi_\delta) * g_k}_{L^\infty} \leq \norm{\Phi-\Phi_\delta}_{L^{s'}} \norm{g_k}_{L^{s}},$$
where we chose again $s = s(\kappa) > 2$ such that $\norm{\Beurling}_{s,s} = \frac{1}{\sqrt{\kappa}}$. On one hand, by the second property in Lemma \ref{lemBreakPsi}, we have that
$$\norm{g_k}_{L^{s}}\leq C(\kappa).$$
On the other hand, since $\int \varphi = 1$, we have that
$$\Phi_\delta(x)-\Phi(x) = \int (\Phi(x-y)-\Phi(x)) \varphi_\delta(y)\, dm(y)$$
for $x\neq 0$ and, by Minkowski's inequality, 
$$\norm{\Phi_\delta-\Phi}_{L^{s'}} \leq \int |\varphi_\delta(y)| \left( \int |\Phi(x-y)-\Phi(x)|^{s'} \, dm(x)\right)^\frac{1}{s'}  dm(y) \leq \int |\varphi_\delta(y)| \modulus{s'}{\Phi}(|y|) dm(y).$$
Now, take $0<\varepsilon(\kappa)< \frac{2}{s'}-1$ (say $ \varepsilon = \frac{1}{s'}-\frac12=\frac12-\frac{1}{s}$). By the Cauchy-Schwarz inequality
\begin{align*}
\norm{\Phi_\delta-\Phi}_{L^{s'}} 
	& \leq \left(\int |\varphi_\delta(y)|^2 |y|^{2\varepsilon+2}\, dm(y) \right)^\frac12 \left(\int \frac{\modulus{s'}{\Phi}(|y|)^2}{|y|^{2\varepsilon+2}}   dm(y)\right)^\frac12.
\end{align*}
The second integral coincides with a Besov norm of the kernel $\Phi$, while the first, by rescaling, can be easily controlled. Namely, by Proposition \ref{propoCompactKernelIsGood} we have that
\begin{equation*}
\norm{\Phi_\delta-\Phi}_{L^{s'}} 
	 \leq C_{s,\varepsilon} \left(\int |\varphi(x)|^2 \delta^{2\varepsilon} \, dm(x) \right)^\frac12 \norm{\Phi}_{B^\varepsilon_{s',2}} = C_\kappa \norm{\varphi}_{L^2} \delta^{\varepsilon} \norm{\Phi}_{B^\varepsilon_{s',2}}
\end{equation*}
and, thus,
\begin{equation}\label{eqBoundApproximation}
\norm{(\Phi-\Phi_\delta) * g_k}_{L^\infty} 
	 \leq C(\kappa) \delta^{\frac12-\frac{1}{s}}.
\end{equation}
Putting together \rf{eqBreakPsi}, \rf{eqBreakPsiDelta}, \rf{eqBoundH} and \rf{eqBoundApproximation} we get that
\begin{align}\label{eqSecondStage}
\norm{\psi_k-Id}_{L^\infty} 
\nonumber	& \leq \norm{\Phi * h_k}_{L^\infty}  + \norm{(\Phi-\Phi_\delta) * g_k}_{L^\infty} + \norm{\Phi_\delta * g_k}_{L^\infty} \\
	& \leq C(\kappa)\left( \kappa^{\frac{{N}}2} +  \delta^{\frac12-\frac{1}{s}} \right) + \norm{\Phi_\delta * g_k}_{L^\infty}.
\end{align}

Next, we deal with the (smoothed) main term in the Fourier side separating low and high frequencies with respect to the parameter $R$. Namely,
\begin{equation}\label{eqBreakSmoothedG}
\norm{\Phi_\delta * g_k}_{L^\infty} \leq \norm{\widehat{\Phi_\delta}\widehat{g_k}}_{L^1}\leq \norm{\widehat{\Phi_\delta}\widehat{g_k}}_{L^1(\D_R)}+\norm{\widehat{\Phi_\delta}\widehat{g_k}}_{L^1(\D_R^c)} .
\end{equation}
The low-frequency part can be bounded using the last property described in Lemma \ref{lemBreakPsi}. 
Since $\widehat{\Phi_\delta}=\widehat{\Phi}\widehat{\varphi_\delta}$, for $\frac1{p_1}+\frac1{p_2}+\frac1{p_3}=1$, by the H\"older inequality
$$\norm{\widehat{\Phi_\delta}\widehat{g_k}}_{L^1(\D_R)} \leq  \norm{\widehat{\Phi}}_{L^{p_1}}\norm{\widehat{g_k}}_{L^{p_2}(\D_R)}\norm{\widehat{\varphi_\delta}}_{L^{p_3}}.$$
The term corresponding to the mollification is bounded by $\norm{\widehat{\varphi_\delta}}_{L^{p_3}}=\delta^{-\frac2{p_3}} \norm{\widehat{\varphi}}_{L^{p_3}}<\infty$. Using the boundedness of the Fourier transform for $p_1=s=s(\kappa)> 2$ and Lemma \ref{lemBreakPsi} for $p_2:= 2$ (recall that we defined $R=\frac{|k|}2$), fixing  $\frac1{p_3}=\frac12-\frac1s$ and $\delta<1$, we get
\begin{align*}
\norm{\widehat{\Phi_\delta}\widehat{g_k}}_{L^1(\D_R)}
	& \leq C(\kappa) \norm{\Phi}_{L^{s'}} M(2,q)^{{N}} \norm{\mu}_{\dot B^\omega_q} \omega\left(\frac2{|k|}\right) \delta^{-\left(1-\frac2s\right)},
\end{align*}
and using the AM-GM inequality, we obtain
\begin{align}\label{eqBoundGLow}
\norm{\widehat{\Phi_\delta}\widehat{g_k}}_{L^1(\D_R)}
	& \leq C(\kappa) \left( M(2,q)^{{2N}} \omega \left(\frac2{|k|}\right) + \omega \left(\frac2{|k|}\right) \delta^{-\left(2-\frac4s\right)} \right) .
\end{align}

Finally, the high frequency part can be controlled using H\"older inequality and the boundedness of the Fourier transform. Arguing as before, we have that
$$\norm{\widehat{\Phi_\delta}\widehat{g_k}}_{L^1(\D_R^c)} 
	\leq \norm{\widehat{g_k}}_{L^{2}}  \norm{\widehat{\Phi}}_{L^{s}(\D_R^c)} \norm{\widehat{\varphi_\delta}}_{L^{p_3}} 
	\leq C(p_1) \norm{g_k}_{L^{2}} \norm{\widehat{\Phi}}_{L^{s}(\D_R^c)} \delta^{-\frac2{p_3}} \norm{\widehat{\varphi}}_{L^{p_3}}  .$$

The norm corresponding to $g_k$ satisfies that $\norm{g_k}_{L^{2}}\leq C(\kappa)$ by Lemma \ref{lemBreakPsi}.  To control the norm of the kernel $\Phi$, by Lemma \ref{lemFourier} we have that 
\begin{align*}
\norm{\widehat{\Phi}}_{L^{s}(\D_R^c)}
	& \leq C(s) \modulus{s'}{\Phi}\left(\frac1R\right)
		\leq C(s,\varepsilon) \frac{\norm{\Phi}_{B^\varepsilon_{s',\infty}}}{R^\varepsilon} . 
\end{align*}
Thus, we use that by Proposition \ref{propoCompactKernelIsGood} the kernel is in any Besov space $B^\varepsilon_{s',\infty}$ as long as $\varepsilon<\frac2{s'}-1$.  Take $\varepsilon=\frac12-\frac1s$ so that $C(s,\varepsilon) \norm{\Phi}_{B^\varepsilon_{s',\infty}}\leq C(\kappa)$. Since $\frac1{p_3}=\frac12-\frac2s$, combining the previous facts we get that
\begin{equation}\label{eqBoundGHigh}
\norm{\widehat{\Phi_\delta}\widehat{g_k}}_{L^1(\D_R^c)} 
	\leq C_\kappa  \frac{1}{\delta^{1-\frac2s}R^{\frac12-\frac1s}}  .
\end{equation}
By \rf{eqSecondStage}, \rf{eqBreakSmoothedG}, \rf{eqBoundGLow} and \rf{eqBoundGHigh}, since  we defined $R=\frac{|k|}2$, we have shown that
\begin{align}\label{eqFinalStage}
\norm{\psi_k-Id}_{\infty} 
	& \leq C_\kappa \left( \kappa^{\frac{{N}}2} +  \delta^{\frac12-\frac{1}{s}} + M(2,q)^{{2N}} \omega\left(\frac2{|k|}\right)  +   \frac{1}{\delta^{2-\frac4s}} \left(\omega\left(\frac2{|k|}\right) + \left(\frac{2}{|k|}\right)^{\frac12-\frac1s}\right)\right)
\end{align}
for an appropriate constant $C_\kappa$.

Let us define $|s^*|=\frac12-\frac{1}{s}$ and
$$\upsilon(t):= C_\kappa \left(\inf_{N\in\N} \left( \kappa^{\frac{N}2}  + M(2,q)^{2N} \omega\left(2t\right) \right)+\inf_{\delta>0} \left(\frac{(2t)^{|s^*|}+ \omega\left(2t\right) }{\delta^{4 |s^*|}} +  \delta^{|s^*|} \right)\right),$$
with the constant chosen as in \rf{eqFinalStage}. We have seen that  
$$\norm{\psi_k-Id}_{L^\infty}\leq \upsilon(|k|^{-1}).$$
To show that $\upsilon$ is increasing and $\lim_{t\to0} \upsilon(t)=0$ we only need to use the counterparts for $\omega$ and solve a standard optimization problem (see Remark \ref{remQuantitativeUpsilon} below). \end{proof}

To end this section, we study how $\upsilon$ depend on the original modulus of continuity.
\begin{lemma}
Let $0<\kappa < 1$, $M>1$ and $\lambda, N>0$. Then
$$\inf_{y\in\R} \left(\kappa^y + \lambda M^y\right) = \lambda^\frac{-\log \kappa}{\log (M/\kappa)}\left[ \left(\frac{ \log M}{-\log \kappa} \right)^\frac{-\log \kappa}{\log (M/\kappa)} +  \left(\frac{ \log M}{-\log \kappa} \right)^\frac{-\log M}{\log (M/\kappa)}\right],$$
and
$$\inf_{x>0} \left(x+\frac{\lambda}{x^N}\right) = \lambda^{\frac1{N+1}} \frac{N+1}{N^{\frac{N}{N+1}}}.$$
\end{lemma}
\begin{proof}
The second infimum can be obtained by a standard optimization exercise. The first one is a consequence of the second after the change $x:= \kappa^y$, $y=\frac{\log x}{\log \kappa}$, so 
$$\inf_{y\in\R} \left(\kappa^y + \lambda M^y\right)=\inf_{x>0} \left(x+\frac{\lambda}{x^{\frac{\log M}{-\log \kappa}}}\right).$$
\end{proof}

\begin{remark}\label{remQuantitativeUpsilon}
By the previous lemma, 
\begin{align*}
\inf_{N\in\N} \left( \kappa^{\frac{N}2}  + M(2,q)^{2N} \omega\left(2t\right) \right)
	&  \leq \inf_{y\geq 0} \max_{x\in [0,1)} \left( \kappa^{\frac{y+x}2}  + M(2,q)^{2(y+x)} \omega\left(2t\right) \right) \\
	& \leq M(2,q)^2 \inf_{y\geq 0} \left( \kappa^{\frac{y}2}  + M(2,q)^{2y} \omega\left(2t\right) \right)\leq  C_{\kappa,q} \, \omega\left(2t\right)^{\alpha},
\end{align*}
where $\alpha={\frac{- \log\kappa}{4\log M(2,q) - \log\kappa}}$.
On the other hand, 
$$\inf_{\delta>0} \left(\frac{(2t)^{|s^*|}+ \omega\left(2t\right) }{\delta^{4 |s^*|}}+  \delta^{|s^*|} \right) = \left((2t)^{|s^*|}+\omega(2t)\right)^{ \frac1{5}} \frac{5}{4^{\frac{4}{5}}}, $$
where $s=s(\kappa)>2$ is given by the relation $\norm{\Beurling}_{s,s}=\frac1{\sqrt{\kappa}}$ with $|s^*|=\frac12-\frac1s$.

Abusing notation, we write $\alpha=\min\left\{ \frac{-\log\kappa}{4\log M(2,q) - \log\kappa}, \frac15\right\}$, to state that
$$\upsilon(t) \lesssim_{\kappa,q} \, \omega\left(2t\right)^{\alpha}+ C_\kappa t^{|s^*|/5}.$$
\end{remark}

\subsection{Decay of the nonlinear solution}\label{secNonlinear}
Next we see how does interact the composition with a quasiconformal mapping with Lebesgue spaces. 
\begin{lemma}\label{lemCompositionQCLebegue}[see {\cite[Lemma 2.4]{OlivaPrats}}]\label{lemCompositionLebesgue}
 Let $K\geq 1$ and $0<p \leq \infty$ and $\frac1q > \frac{K}{p}$. Given  $f \in L^p(\D)$ and a $K$-quasiconformal principal mapping $\phi$, we have that
 $$\norm{f \circ \phi}_{L^q(\phi^{-1}(\D))}\leq C_{K,q,p} \norm{f}_{L^p(\D)}.$$
\end{lemma}

\begin{lemma}\label{lemCompositionModuli}
Let $\phi$ be a $K$-quasiconformal principal mapping and let $\mu\in L^\infty$ be supported in $\overline{2\D}$. Consider $0<p \leq \infty$ and $\frac1q > \frac{K}{p}$. For $t$ small enough (depending on $K$), there exist constants $C_{K,q,p}$ and $C_K$ such that
$$\modulus{q}{(\mu\circ \phi)}(t) \leq C_{K,q,p} \, \modulus{p}{\mu}(C_K t^\frac1K).$$
\end{lemma}

\begin{proof}
First we will show the following claim:
\begin{claim}\label{claimComparableModuli}
For every $g\in L^1_{loc}(\C)$, we have that
$$\modulus{p}{g}(t) \approx_{p} \left( \int_{\C} \fint_{B(x,t)} |g(x)-g(y)|^p\, dm(y)\, dm(x) \right)^\frac1p.$$
\end{claim}
\begin{figure}
 \centering
\begin{subfigure}{0.5 \textwidth}
 \centering{\includegraphics[width=\textwidth]{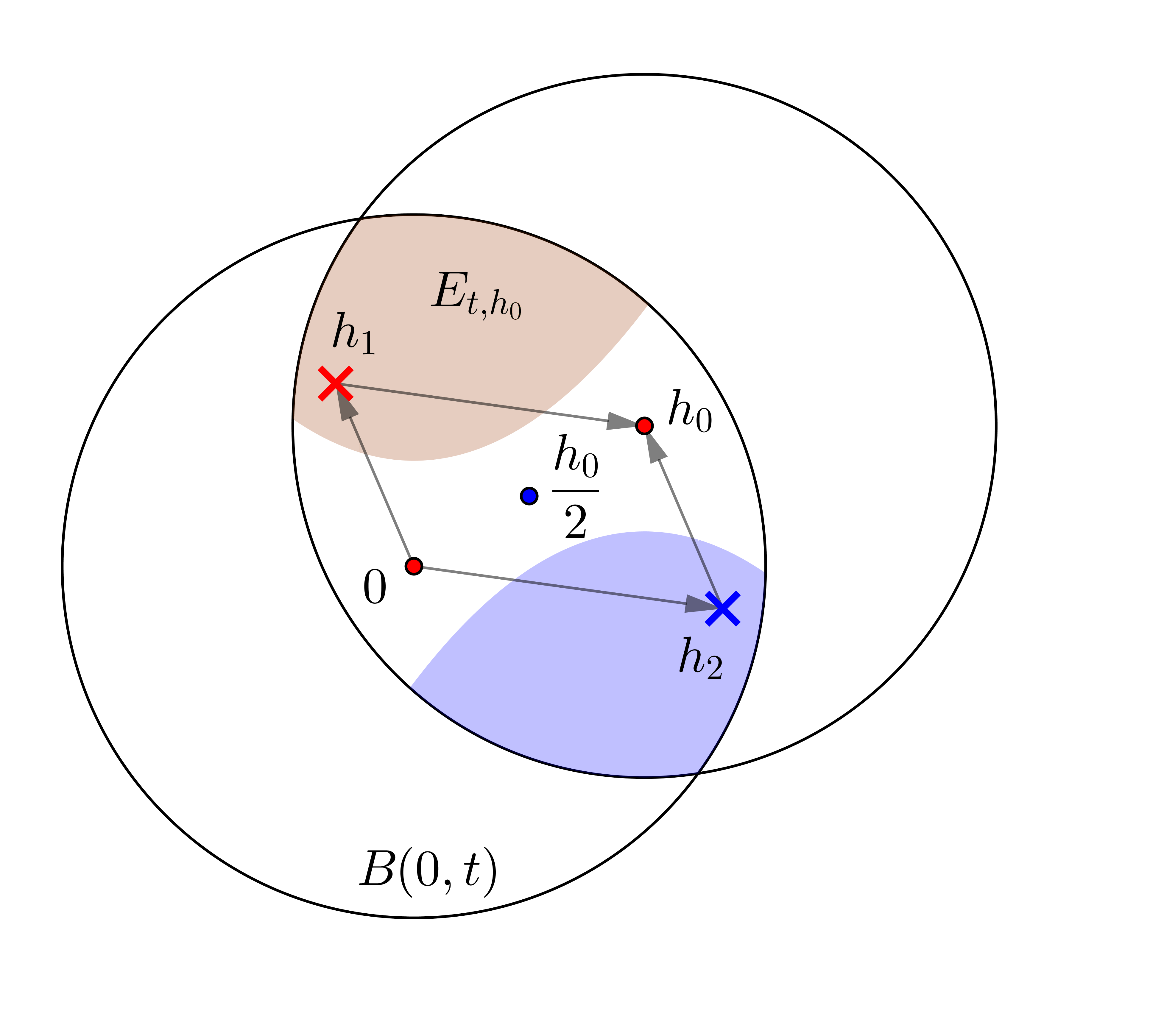}}
\end{subfigure}
\caption{The symmetrized region is included in the complement by the triangle inequality, see the proof of Claim \ref{claimComparableModuli}.}\label{figDuesZones}
\end{figure}

Indeed, the function 
$$\modulus{p}{g}(h):= \left( \int |g(x)-g(x+h)|^p\right)^\frac1p \mbox{\quad\quad for }h\in \C,$$
is measurable by Tonelli's theorem. Given $|h_0|<t$, the set $E_{t,h}:=\{|h|<t, |h-h_0|<t : \modulus{p}{g}(h)< \frac12 \modulus{p}{g}(h_0) \}$ is measurable, and its measure is bounded by $|\{|h|<t, |h-h_0|<t\}|/2$: using the triangle inequality $\modulus{p}{g}(h_1+h_2)\leq\modulus{p}{g}(h_1)+\modulus{p}{g}(h_2)$ we obtain that the set $\{|h|<t, |h-h_0|<t\} \setminus E_{t,h}$ contains the symmetrization of $E_{t,h}$ with respect to $h_0/2$ (see Figure \ref{figDuesZones}). Note that the argument is valid even if  $\modulus{p}{g}(t)=\infty$. Thus, there is a measurable region $E_0\subset B(0,t)$ with Lebesgue measure $|E_0| \geq \frac{\pi }6t^2$ with  $\modulus{p}{g}(h)\geq \frac12 \modulus{p}{g}(h_0)$ for $h\in E_0$. Thus, 
$$\modulus{p}{g}(t) 
	= \sup_{|h_0|\leq t} \modulus{p}{g}(h_0) \leq 2 \left(  \fint_{E_0} \modulus{p}{g}(h)^p  dm(h) \right)^\frac1p
	 \lesssim_{p} \left(  \fint_{B(0,t)} \modulus{p}{g}(h)^p  dm(h) \right)^\frac1p.$$
The converse inequality is trivial. Claim \ref{claimComparableModuli} follows now by Tonelli's theorem. 

Now,  using Claim \ref{claimComparableModuli}, and adding and subtracting $\fint_{B(\phi(x),Ct^{1/K})} \mu \, dm $ with $C$ to be fixed, we obtain
\begin{align*}
\modulus{q}{(\mu\circ \phi)}(t) 
	& \approx \left( \int_{\C} \fint_{B(x,t)} |\mu\circ \phi(x)-\mu\circ \phi(y)|^q\, dm(y)\, dm(x) \right)^\frac1q\\
	& \lesssim \left( \int_{\C} \left|\mu\circ\phi (x)- \fint_{B(\phi(x),Ct^{1/K})} \mu \, dm \right|^q\, dm(x) \right)^\frac1q \\
	& \quad + \left( \int_{\C} \fint_{B(x,t)} \left|\mu\circ\phi (y)- \fint_{B(\phi(x),Ct^{1/K})} \mu \, dm \right|^q\, dm(y)\, dm(x) \right)^\frac1q .
\end{align*}
Since $\phi$ maps the unit disk to $2\D$ (being a principal mapping, this follows from Koebe's theorem) and it is $\frac1K$-H\"older continuous, we obtain that for $x,y\in 2\D$ with $|x-y|<t$, $|\phi(x)-\phi(y)|\leq C_K |x-y|^\frac1K < C_K t ^\frac1K$, that is, $\phi(x)\in B(\phi(y), C_Kt^{\frac1K})$. 
Therefore, for $t<1$, if we choose above $C=C_K$, taking absolute values inside the integral and using Jensen's inequality we get that
\begin{align*}
\modulus{q}{(\mu\circ \phi)}(t) 
	& \lesssim \left( \int_{\C}  \fint_{B(\phi(x),Ct^{1/K})} \left|\mu\circ\phi (x)  - \mu(w)  \right|^q \, dm(w)\, dm(x) \right)^\frac1q \\
	& \quad + \left( \int_{\C}  \fint_{B(\phi(y),2Ct^{1/K})} \left|\mu\circ\phi (y) - \mu (w)\right|^q \, dm(w) \, dm(y) \right)^\frac1q .
\end{align*}
Applying Lemma \ref{lemCompositionLebesgue} to the function $g: z\mapsto \left(\fint_{B(z,2Ct^{1/K})} \left|\mu(z)  - \mu(w)  \right|^q \, dm(w) \right)^\frac1q$, we obtain
$$ \left( \int_{\D} g\circ \phi (x)^q\, dm(x) \right)^\frac1q\leq  C_{K,q,p} \left( \int_{2\D} g(z)^p \, dm(z) \right)^\frac1p. $$
Since $q<p$, applying Jensen's inequality to $g$ and Claim \ref{claimComparableModuli} the lemma follows.
\end{proof}

\begin{theorem}\label{theoDecayOfCGOS}
Let $\kappa < 1$,  let $2K<p<\infty$, where $K$ is defined by \rf{eqKappaK} and let $\omega$ be a modulus of continuity. There exits a modulus of continuity $\upsilon$ depending only on $\kappa$, $p$ and $\omega$ such that for every $\mu \in \mathcal{M}(\kappa,p,\omega)$, we have that
$$\norm{\varphi_\mu (\cdot,k)-Id}_{L^\infty}\leq \upsilon(|k|^{-1}),$$
where $\varphi_\mu(\cdot,k)$ stands for the quasiconformal mapping defined in Theorem \ref{theoConsequencesFMu}. 
\end{theorem}

\begin{proof}
Consider $\psi_k:= \varphi_\mu(\cdot, k)^{-1}$, that is, the quasiconformal map which is inverse to $\varphi_\mu(\cdot, k)$, which solves 
\begin{equation}\label{eqVarphi}
 \overline{\partial} (ik\varphi_\mu(\cdot,k))=  \mu \,e_{-k}(\varphi_\mu(\cdot,k))  \overline{\partial (ik\varphi_\mu(\cdot,k))}.
\end{equation}
 Then, using the inversion formulas \cite[(2.49) and (2.50)]{AstalaIwaniecMartin} for the Wirtinger derivatives, we get
$$ \overline{\partial} \psi_k(\cdot) = - \frac{ \bar{k}}{k}\mu\circ\psi_k(\cdot) \, e_{-k}(\cdot) \, \partial \psi_k(\cdot).$$
By Lemma \ref{lemCompositionModuli}, for every $q>0$ and $p$ in the range $0<\frac1p < \frac 1 {Kq}$, we have that 
$$\modulus{q}{(\mu\circ\psi_k)}(t)\leq C_{K,q,p} \, \modulus{p}{\mu}(C_K t^\frac1K)\leq C_{K,q,p} \, \omega(C_K t^\frac1K).$$
That is, $\mu\circ\psi_k\in \mathcal{M}(\kappa,q,\widetilde{\omega})$ for
\begin{equation}\label{eqModulusModified}
\widetilde{\omega}(t)=C_{K,q,p} \, \omega(C_K t^\frac1K).
\end{equation}

Choosing $q>2$, by Proposition \ref{propoUpsilon} we get that there exits a modulus of continuity $\upsilon$ depending only on $K$, $q$, $\omega$ and $C_{K,q,p}$ such that
$$\norm{\psi_k-Id}_{L^\infty}\leq \upsilon(|k|^{-1}).$$
But quasiconformal mappings preserve the essential supremum norm because they preserve null sets. Thus,
$$\norm{\psi_k-Id}_{L^\infty}=\norm{\psi_k\circ \varphi_\mu(\cdot,k) - \varphi_\mu(\cdot,k)}_{L^\infty}=\norm{Id-\varphi_\mu(\cdot,k)}_{L^\infty}.$$
\end{proof}

We end the section with some remarks on the last results above. Lemma \ref{lemCompositionModuli}  is non-sharp, at least for the Besov spaces. Indeed, $\mu\in \dot B^{s}_{p,\infty}$ if and only if $\modulus{p}{\mu}(t)\leq C t^s$ for $t<1$. Applying Lemma \ref{lemCompositionModuli} we obtain that $\modulus{q}{(\mu\circ \phi)}(t)\leq C C_K^s t^\frac sk$, that is, we get that $\mu\circ\phi \in \dot B^\frac sK_{q,\infty}$. Moreover, if we write $\frac1{p^*_s}=\frac1p-\frac s2$, since $B^s_{p,\infty}\subset L^r$ for $r<p^*_s$ (see \cite[Section 2.7]{TriebelTheory}, for instance), Lemma \ref{lemCompositionQCLebegue} implies that $\mu\circ\phi \in L^q$ for $\frac1q>\frac Kp- \frac{sK}2$, and, as a consequence, for $\frac1q>\frac Kp$ we get that $\mu\circ\phi \in B^\frac sK_{q,\infty}$. However, we already  know from \cite{OlivaPrats} that $\mu\circ\phi \in B^s_{q,\infty}$ for $\frac1q>\frac1p + \frac{K-1}2 \left(\frac2p-s\right)$, which implies our result by elementary embeddings (see \cite[Section 2.2]{TriebelTheory}).

The reason for this result to be so vague is that we do not use the relation between the jacobian determinant and the distortion at every point, but only the H\"older regularity in a quite naive way. The question here is whether it is possible to improve the result above or not to get, for instance, 
$$\modulus{p}{(\mu\circ \phi)}(t) \leq C_1 \, \modulus{p}{\mu}(C_2 t^\frac1K),$$
that is, preserving the integrability and losing only smoothness, or 
$$\modulus{q}{(\mu\circ \phi)}(t) \leq C_1 \, \modulus{p}{\mu}(C_2 t) \mbox{\quad\quad for every } \frac1q>\frac Kp,$$ 
which does not recover the Besov case but it would mean that there is no ``loss of smoothness''.

In any case, Lemma \ref{lemCompositionModuli} covers any modulus of continuity and, as a consequence, it can be applied to every $\mu \in L^\infty_c$, which is our purpose in the present paper.

As it has been noticed above, there are some particular moduli for which we know better estimates for $\widetilde{\omega}$ than \rf{eqModulusModified}. If we restrict to Beltrami coefficients in subcritical Besov spaces $B^s_{p,\infty}$, i.e., if $\omega (t)= C t^s$ with $0<s<1$ and $sp<2$, then by \cite[Theorem 2.25]{OlivaPrats} we can deduce that $\mu \circ \psi_k \in B^s_{q,\infty}$ as long as $1>\frac1q > \frac1p + (K-1)\left(\frac1p-\frac{s}{2}\right)=\frac Kp - (K-1)\frac{s}{2}$.
If, instead, the Besov space is supercritical, then we can take $1>\frac1q > \frac1p + \frac{(K-1)}{K}\left(\frac{s}{2}-\frac1p\right)$.  Therefore, in these cases we can write 
\begin{equation*}
\widetilde{\omega}(t)=C_{K,q,p} \,  t^s.
\end{equation*}
Note that the modulus $\widetilde{\omega}$ obtained here has a better decay at the origin, and the integrability parameter $q$ has been improved as well. 

To recover the previously known results (see \cite[Theorem 3.7]{BarceloFaracoRuiz} and \cite[Theorem 5.6]{ClopFaracoRuiz}), it suffices to switch the loss from the integrability to the smoothness. Namely, in the subcritical setting, where $sp<2$, by the embedding theorems we have $\mu \in B^{\frac sK}_{Kp,\infty}$ and
$$\modulus{p}{(\mu\circ\psi_k)}(t) \leq C_{K,q,p}  t^\frac sK.$$ 
In the supercritical case $sp>2$, we can find $0<\alpha_\kappa<1$ depending on $K$ and $p$ such that 
$$\modulus{p}{(\mu\circ\psi_k)}(t) \leq C_{K,q,p}  t^{\alpha_\kappa s}.$$ 
Moreover, it has the additional advantage that  $\mu\circ\psi_k \in C^{\beta_\kappa s}$ is granted by the Embedding theorems (see \cite[Section 2.7]{TriebelTheory}, for instance) for a certain $0<\beta_\kappa<1$ which depends on $K$ and $p$ as well, that is 
$$\modulus{\infty}{(\mu\circ\psi_k)}(t) \leq C_{K,q,p}  t^{\beta_\kappa s}.$$ 
As a consequence, if we restrict to $C^s$ H\"older spaces, with $0<s<1$, i.e., if $\omega (t)= C t^s$ and $p=\infty$, we can take $\frac1q >\frac{ (K-1) s}{2K}$ and $\widetilde{\omega}(t)=C_{K,q,p} \,  t^s$, which agrees with the fact that $\mu \circ \phi \in B^{s}_{q,\infty}\subset C^{\beta_\kappa s}$ (any $\beta_\kappa<\frac1K$ will do by the embedding theorems). Thus, we can take $q=\infty$ as well with a worse modulus of continuity $\widetilde{\omega}(t)=C_{K,q,p} \,  t^{\beta_\kappa s}$.

\subsection{Decay of the conductivity solution}\label{secDecayConductivity}

Equation \rf{eqRelationfu} can be used to deduce a useful relation between the complex geometric optics solutions to the conductivity equation $u_\gamma$ and the family of solutions to the Beltrami equations with rotated coefficients $f_{\lambda\mu}$.

\begin{lemma}\label{lemUgammaFLambdaMu}
 Given a Beltrami coefficient $\mu \in L^\infty$ supported in $\overline{\D}$ with $\norm{\mu}_{L^\infty}<1$. Then
 \begin{equation}\label{eqPhiXiFMu}
u_\gamma(z,k)=f_{\lambda_{\mu}(z,k) \mu}(z,k)=e^{ik\varphi_{\lambda_{\mu}(z,k) \mu}(z,k)},
\end{equation}
where $\lambda_{\mu}(z,k)\in\partial\D$ is defined as
\begin{equation}\label{eqDefinitionXi}
\lambda_{\mu}(z,k)=\begin{cases}
	\frac{\overline{f_{\mu}(z,k)-f_{-\mu}(z,k)}}{f_\mu(z,k)-f_{-\mu}(z,k)} & \mbox{if } f_\mu(z,k) \neq f_{-\mu}(z,k),\\
	1 & \mbox{otherwise.}
	\end{cases}
\end{equation}
\end{lemma}
\begin{proof}
First we recover an argument from \cite[Lemma 8.2]{AstalaPaivarinta}. Consider $\lambda \in \partial\D$ and let
$$\Phi_\lambda :=\left(\frac{1+\lambda}{2}f_\mu+\frac{1-\lambda}{2}f_{-\mu}\right)=\frac{f_\mu+f_{-\mu}}{2}+ \lambda \frac{f_\mu-f_{-\mu}}{2}. $$
By \rf{eqBeltramiAssimptotics} we have that
$$\Phi_\lambda=e^{ikz}\left(1+\mathcal{O}_{z\to\infty}\left(\frac{1}{z}\right)\right),$$
and, since $\lambda\bar{\lambda}=1$,  it is immediate to check that
$$\bar\partial \Phi_\lambda = \left(\frac{1+\lambda}{2}\bar\partial f_\mu+\frac{1-\lambda}{2} \bar\partial f_{-\mu}\right) = \left(\frac{\lambda(1+\bar{\lambda})}{2}\mu\overline{\partial f_\mu}- \frac{\lambda(-1+\bar{\lambda})}{2} \mu\overline{\partial f_{-\mu}}\right) = \lambda \mu  \overline{\partial\Phi_\lambda}.$$
By Definition \ref{defFMu} and \rf{eqFMuRepresentation} we have that 
$\Phi_\lambda=f_{\lambda \mu} = e^{ik\varphi_{\lambda\mu}}.$
By \rf{eqBeltramiAssimptotics} and Theorem \ref{theoConsequencesFMu} we have that $\real\left(\frac{f_\mu}{f_{-\mu}}\right)>0$ and, therefore, $f_\mu+f_{-\mu}\neq 0$. Using the definition of $\Phi_\lambda$ above, this equality reads as
\begin{align}\label{eqExpressionNewHope}
f_{\lambda \mu}=\frac{ \lambda \left((\bar{\lambda}+1)f_\mu+(\bar{\lambda}-1)f_{-\mu}\right)}{2} \frac{(f_\mu+f_{-\mu})}{(f_\mu+f_{-\mu})}=\left(\frac{f_\mu-f_{-\mu}}{f_\mu+f_{-\mu}}+\lambda^{-1} \right)\lambda \frac{f_\mu+f_{-\mu}}{2}.
\end{align}

Fix $z,k\in\C$. By \rf{eqRelationfu}, we have that
\begin{align*}
u_\gamma(z,k)
	& =\frac{1}{2} \left(f_\mu+f_{-\mu}+\overline{f_{\mu}}-\overline{f_{-\mu}}\,\right)(z,k).
\end{align*}
If $f_\mu(z,k)\neq f_{-\mu}(z,k)$, with some algebraic manipulation and using \rf{eqDefinitionXi} we get
\begin{align*}
u_\gamma(z,k)
	& =\frac{f_\mu+f_{-\mu}}{2}\left(1+\frac{\overline{f_{\mu}}-\overline{f_{-\mu}}}{f_\mu+f_{-\mu}}\right) (z,k)
		=\frac{f_\mu+f_{-\mu}}{2}\lambda_{\mu}\left((\lambda_{\mu})^{-1}+\frac{f_\mu-f_{-\mu}}{f_\mu+f_{-\mu}}\right)(z,k).
\end{align*}
Otherwise,  since we defined $\lambda(z,k)=1$ this identity is satisfied as well (in fact, we can fix any $\lambda\in\partial\D$ without restriction in this case).
Substituting the pointwise equality \rf{eqExpressionNewHope}, we get
\begin{equation*}
u_\gamma(z,k) =f_{\lambda_\mu(z,k) \mu}(z,k).
\end{equation*}
\end{proof}
 \begin{remark}
 Note that in the previous lemma we have seen two key facts. First, $f_\mu$ and $f_{-\mu}$ determine the solutions $f_{\lambda\mu}$ of any rotation of the Beltrami coefficient $\mu$ via \rf{eqExpressionNewHope}. Secondly, they also determine the possible values of $u_\gamma$, so any bound (either above or below), decay at infinity,... that we find for $f_{\lambda\mu}$ which is uniform on $\norm{\mu}$  or, even better, just in $\lambda$, is automatically applied to $u_\gamma$.
 \end{remark}

Let us introduce a new family of conductivities adapted to Definition \ref{defFamiliesOfModuliRevisited}.
\begin{definition}\label{defFamiliesOfModuliGamma}
Let $0<\kappa<1$, $1< p <\infty$ and let $\omega$ be a modulus of continuity. We define
$$\mathcal{G}(\kappa, p,\omega):=\{\gamma: \gamma=\gamma_\mu \mbox{ for a certain }\mu\in\mathcal{M}(\kappa, p,\omega) \}.$$
\end{definition}
Note that this definition does not coincide with  $\mathcal{G}(K, \D, p,\omega)$ from Definition \ref{defFamiliesOfModuli}. However, it is not difficult to find constants $C_j$ which depend only on $\kappa$ so that 
$$\mathcal{G}(K,\D,p,C_1\omega)\subset \mathcal{G}(\kappa, p,\omega)\subset  \mathcal{G}(K, \D,p, C_2\omega)$$
whenever \rf{eqKappaK} is satisfied.

To lift \rf{eqPhiXiFMu} to the logarithms of both CGOS, we need to check the continuity of $k\varphi_{\lambda\mu}$. We will use the following version of the Liouville theorem:
\begin{theorem}\label{theoArgumentPrinciple}
Let $\kappa<1$, let $2<p <p_\kappa$, let $M \in L^p(\D)$ and $E\in L^2(\D)$ be positive functions and let $F\in W^{1,2}_{loc}(\C)$ satisfy the differential inequality
\begin{equation*}
|\bar\partial F|\leq \chi_{\D} \left(\kappa |\partial F| + M|F| + E\right).
\end{equation*}
If
$$\lim_{z\to \infty}F(z)=0,$$
then for $2<q<\infty$ we have that
$$\norm{F}_{L^{q}(\C)} \leq e^{C\left(1+ \norm{M}_{L^p(\D)}\right)}\norm{E}_{L^2(\D)}.$$
If, moreover, $F\in L^q(\D)$ with $2<q<p_\kappa$ , then
\begin{equation*}
\norm{F}_{W^{1,q}(\C)} \leq e^{C\left(1+\norm{M}_{L^p(\D)}\right)}\norm{E}_{L^q(\D)}.
\end{equation*}
The constants depend only on $p$, $q$ and $\kappa$.
\end{theorem}

We skip the  proof because it  is exactly the same as in \cite[Theorem 2.1]{BarceloFaracoRuiz}.

\begin{lemma}\label{lemVarphiIsContinuous}
Let $k\in\C$ and let $\mu\in L^\infty$ be a compactly supported Beltrami coefficient. The function $z \mapsto \varphi_{\lambda_{\mu}(z,k) \mu}(z,k)$ as defined in Lemma \ref{lemUgammaFLambdaMu} is continuous, and  $k \mapsto k\varphi_{\lambda_{\mu}(z,k) \mu}(z,k)$ is differentiable.
\end{lemma}

\begin{proof}
First we check that $\lambda \mapsto k\varphi_{\lambda \mu}(z,k)$ is a continuous function. Let $z_0\in\C$ and $k\in \C\setminus\{0\}$.
By \rf{eqVarphi}, 
\begin{equation}\label{eqVarphiLambda}
 \overline{\partial} \varphi_{\lambda\mu}(\cdot,k)= - \frac{ \bar{k}}{k}{\lambda\mu} \,e_{-k}(\varphi_{\lambda\mu}(\cdot,k))  \overline{\partial \varphi_{\lambda\mu}(\cdot,k)}.
\end{equation}

Let $\lambda_0$, $\lambda_1 \in \partial\D$ be given and write $\varphi_j:=\varphi_{\lambda_j\mu}$. From \rf{eqVarphiLambda} we derive that
\begin{align*}
 \overline{\partial} (\varphi_0-\varphi_1) (\cdot,k)
 	& =  \frac{ \bar{k}}{k}{\lambda_1\mu} \,e_{-k}(\varphi_{1}(\cdot,k))  \overline{\partial \varphi_{1}(\cdot,k)}- \frac{ \bar{k}}{k}{\lambda_0\mu} \,e_{-k}(\varphi_{0}(\cdot,k))  \overline{\partial \varphi_{0}(\cdot,k)}\\
 	& =  \frac{ \bar{k}}{k}{(\lambda_1-\lambda_0)\mu} \,e_{-k}(\varphi_{1}(\cdot,k))  \overline{\partial \varphi_{1}(\cdot,k)}\\
 	& \quad +  \frac{ \bar{k}}{k}{\lambda_0\mu} \,\left(e_{-k}(\varphi_{1}(\cdot,k)) -e_{-k}(\varphi_{0}(\cdot,k)) \right) \overline{\partial \varphi_{1}(\cdot,k)}\\
 	& \quad +  \frac{ \bar{k}}{k}{\lambda_0\mu} \,e_{-k}(\varphi_{0}(\cdot,k))  \overline{\partial \left(\varphi_{1}(\cdot,k)-\varphi_{0}(\cdot,k)\right)}.
\end{align*}
Writing $ F(z)=(\varphi_0-\varphi_1) (z,k)$ and taking absolute values, we get
\begin{align*}
 \left|\overline{\partial} F \right|
 	& \leq \chi_\D \left(|\lambda_1-\lambda_0| |\partial \varphi_{1}| + \left| e_{-k}(\varphi_{1}) -e_{-k}(\varphi_{0}) \right| \left|\partial \varphi_{1} \right| +   \left|{\mu}  \right||\partial F|\right).
\end{align*}
The function $z\mapsto e_{k}(z)=e^{2i k\cdot z}$ is Lipschitz with constant $2|k|$. Thus, 
\begin{align*}
 \left|\overline{\partial} F \right|
 	& \leq  \chi_\D \left(|\lambda_1-\lambda_0| |\partial \varphi_{1}| + |F| 2|k|\left|\partial \varphi_{1} \right| +   \left|{\mu}  \right||\partial F|\right).
\end{align*}
By \rf{eqCauchyCompactlyBounded}, we have that $F=\varphi_0-\varphi_1 =\Cauchy(\bar\partial\varphi_0 -\bar\partial \varphi_1) \in W^{1,p}(\C)$ for $2<p<p_\kappa$. By Theorem \ref{theoArgumentPrinciple}, we get that 
$$\norm{\varphi_0(\cdot,k)-\varphi_1(\cdot,k)}_{L^\infty (\C)} \leq \norm{\varphi_0(\cdot,k)-\varphi_1(\cdot,k)}_{W^{1,p}(\C)} \leq e^{C\left(1+2|k|\norm{\partial\varphi_1}_{L^p(\D)}\right)}|\lambda_1-\lambda_0|\norm{\partial \varphi_{1}}_{L^p(\D)}, $$
that is, $\norm{\varphi_0(\cdot,k)-\varphi_1(\cdot,k)}_{L^\infty (\C)} \leq C_{\kappa,|k|} |\lambda_1-\lambda_0|$, and 
\begin{equation}\label{eqLambdaVarphiContinuous}
\lambda \mapsto \varphi_{\lambda \mu}(z_0,k) \mbox{\quad\quad  is a Lipschitz continuous function in } \partial\D,
\end{equation}
 as we wanted to check.

Next, let $z \in \C$. By the triangle inequality
\begin{align}\label{eqVarepsilonBreak}
\nonumber\left|\varphi_{\lambda_{\mu}(z,k) \mu}(z,k)-\varphi_{\lambda_{\mu}(z_0,k) \mu}(z_0,k)\right| 
	& \leq \left|\varphi_{\lambda_{\mu}(z,k) \mu}(z,k)-\varphi_{\lambda_{\mu}(z,k) \mu}(z_0,k)\right|\\
	& \quad   + \left|\varphi_{\lambda_{\mu}(z,k) \mu}(z_0,k)-\varphi_{\lambda_{\mu}(z_0,k) \mu}(z_0,k)\right|.
\end{align}
The first term is always bounded by the H\"older regularity of $K$-quasiconformal principal mappings:
$$\left|\varphi_{\lambda_{\mu}(z,k) \mu}(z,k)-\varphi_{\lambda_{\mu}(z,k) \mu}(z_0,k)\right| \leq C_\kappa |z-z_0|^{\frac1K}.$$

If $f_\mu(z_0,k) = f_{-\mu}(z_0,k)$, then $f_{\lambda\mu}(z_0,k)=f_{\mu}(z_0,k)$ for every $\lambda\in\partial\D$ by \rf{eqExpressionNewHope}. In that case, the last term in the right-hand side of \rf{eqVarepsilonBreak} vanishes. Thus, we can assume that $f_\mu(z_0,k) \neq f_{-\mu}(z_0,k)$. By the continuity of $f_\mu$ and $f_{-\mu}$ in the $z$ variable, we get that $f_\mu(z,k) \neq f_{-\mu}(z,k)$ for $z\in B(z_0,r_0)$ for $r_0=r_0(z_0,k)$ small enough. 

From \rf{eqDefinitionXi} we obtain that $z\mapsto \lambda_{\mu}(z,k)$ is continuous in $B(z_0,r_0)$, and using that $\lambda \mapsto \varphi_{\lambda \mu}(z_0,k)$ is continuous by \rf{eqLambdaVarphiContinuous}, we get that 
$$z\mapsto \varphi_{\lambda_{\mu}(z,k) \mu}(z_0,k)$$
is continuous  (and this continuity is uniform in a neighborhood of $k$). Therefore, back to \rf{eqVarepsilonBreak} we get
\begin{align*}
\left|\varphi_{\lambda_{\mu}(z,k) \mu}(z,k)-\varphi_{\lambda_{\mu}(z_0,k) \mu}(z_0,k)\right|\xrightarrow{z\to z_0} 0.
\end{align*}

The continuity of $k \mapsto k\varphi_{\lambda_{\mu}(z,k) \mu}(z,k)$ follows by an analogous reasoning using the continuity of the quasiconformal mapping $k\varphi_\mu$ described in Remark \ref{remDifferentiabilityOfFunctions}. By \rf{eqPhiXiFMu}, the function $ik\varphi_{\lambda_{\mu}(z,k) \mu}(z,k)$ is a determination of the logarithm given by the Abstract Monodromy Theorem and, thus, it inherits the differentiability properties of $u_\gamma$. The second statement of the theorem follows from \rf{eqCInfty} and the considerations in Remark \ref{remDifferentiabilityOfFunctions} again.
\end{proof}

\begin{proposition}\label{propoDecayDelta}
Let $0\leq \kappa<1$, let $2K\leq p\leq \infty$, where $K$ is defined by \rf{eqKappaK}, let $\omega$ be a modulus of continuity and let $\gamma\in\mathcal{G}(\kappa, p,\omega)$. Then, the function $\delta_\gamma(z,k) := ik \varphi_{\lambda_\mu (z,k) \mu}(z,k)$ satisfies that
$$u_\gamma (z,k) = e^{\delta_\gamma(z,k)},$$
with $\delta_\gamma(z,k)-i zk=\mathcal{O}_{z\to\infty}(\frac{1}{z})$ and $\delta(\C,k)=\C$;
and there exists a modulus of continuity $\upsilon$ depending only on $(\kappa, q,\omega)$ such that
\begin{equation*}
|\delta_\gamma(z,k)-izk|\leq  |k| \upsilon(|k|^{-1}).
\end{equation*}
\end{proposition}

\begin{proof}
By \rf{eqUniformDecay} and Theorem \ref{theoDecayOfCGOS}, 
$$|\delta_\gamma(z,k)-izk| \leq |k| |\varphi_{\lambda_\mu (z,k) \mu}(z,k) - z|\leq |k| \min \left\{ \frac{C_\kappa}{|z|}, \upsilon(|k|^{-1}) \right\}.$$
Moreover, by Lemma \ref{lemUgammaFLambdaMu}, we have that
$$e^{\delta_\gamma(z,k)}=e^{ik \varphi_{\lambda_\mu (z,k) \mu}(z,k)}=f_{\lambda_\mu (z,k) \mu}(z,k)=u_\gamma(z,k).$$
Moreover,  $\delta_\gamma$ is continuous with respect to $z$ by Lemma \ref{lemVarphiIsContinuous}. The asymptotic behavior of $\delta_\gamma(\cdot,k)$ implies that $\delta_\gamma(\C,k)=\C$ by a classical homotopy argument.
\end{proof}

\section{Deriving stability from the scattering transform}\label{secScatteringToUInfty}
Next we consider two conductivities $\gamma_1, \gamma_2$ and we will write $u_j:=u_{\gamma_j}$, $\delta_j:=\delta_{\gamma_j}$, $\mu_j:=\mu_{\gamma_j}$, $\tau_j:=\tau_{\mu_j}$ and so on for $j\in\{1,2\}$. Barcel\'o, Faraco and Ruiz showed that there is unconditional stability from the DtN map to the scattering transform, that is, we need no a priori assumptions on the conductivities to show the following stability result.

\begin{theorem}[{see \cite[Corollary 4.5]{BarceloFaracoRuiz}}]\label{theoStabilityFromDtnToScattering}
Let $\mu_j \in L^\infty$ real valued and supported in $\overline{\D}$ with $\kappa:=\norm{\mu}_{L^\infty}<1$  for $j\in\{1,2\}$. Let $\rho:=\norm{\Lambda_1-\Lambda_2}_{H^{1/2}(\partial\D)\to H^{-1/2}(\partial\D)}$. Then, there exists a constant $C$ such that
$$ \left|\tau_1(k)-\tau_2(k)\right|\leq C\rho e^{C|k|}.$$
\end{theorem}

In this section we will follow the notation described above to polish the arguments presented in \cite[Section 5]{BarceloFaracoRuiz} and adapt them to the context of modulus of continuity.

\subsection{Main result and sketch of the proof}\label{secGauge}

\begin{theorem}\label{theoUInftyIsSmall}
Let $\kappa \in(0,1)$, let $2K< p <\infty$ with $K$ defined by \rf{eqKappaK}, let $\omega$ be a modulus of continuity and let ${\mathcal{G}}:=\mathcal{G}(\kappa, p, \omega)$.
There exists a modulus of continuity $\iota_{\mathcal{G}}$ depending only on $(\kappa, p, \omega, |k|)$ such that for every $\gamma_1, \gamma_2\in {\mathcal{G}}$ we have the estimate
\begin{equation*}
\norm{u_1(\cdot,k)-u_2(\cdot,k)}_{L^\infty(2\D)}\leq  \iota_{\mathcal{G}}(\rho),
\end{equation*}
where $\rho:=\norm{\Lambda_1-\Lambda_2}_{H^{1/2}(\partial\D)\to H^{-1/2}(\partial\D)}$. Moreover, if $\omega$ is upper semi-continuous, then there are constants $C_{\kappa,p},C_{\kappa}, C_{K},\alpha_K, b_{\kappa,p}$ such that
\begin{equation*}
\iota_{\mathcal{G}}(\rho)
	\leq   C_{\kappa,p} \,  \omega\left(\frac{C_K}{|\log(\rho)|^\frac1K}\right)^{b_{\kappa,p}} +  \frac{C_\kappa}{|\log(\rho)|^{\alpha_K}}
\end{equation*}
for $\rho$ small enough.
\end{theorem}

\begin{proof}
Let $z\in 2\D$ and $k\in \C$. If $k=0$, then the asymptotic condition in \rf{eqConductivityAssimptotics} implies $u_1(z,k)=u_2(z,k)=1$ and there is nothing to prove, so we can assume that $k\neq 0$. 

Proposition \ref{propoDecayDelta} grants the existence of a modulus of continuity $\upsilon$ depending only on $(\kappa, q,\omega)$ such that
\begin{equation*}
|\delta_j(z,k)-izk|\leq  |k| \upsilon(|k|^{-1}).
\end{equation*}
On the other hand, by Theorem \ref{theoStabilityFromDtnToScattering},
$$ \left|\tau_1(k)-\tau_2(k)\right|\leq C\rho e^{C|k|}.$$

By Proposition \ref{propoDecayDelta}, we have that $u_j(z,k)=e^{\delta_j(z,k)}$, with $\delta_j(\cdot,k)$ onto. Thus, we can find $w\in\C$ such that $\delta_1(z,k)=\delta_2(w,k)$. Hence, $u_1(z,k)=u_2(w,k)$ and $g(z,w,k):=\delta_1(z,k)-\delta_2(w,k)=0$.

Thus, the conditions for Proposition \ref{propoGDoesNotVanishFarAway} below are satisfied, and there exists a modulus of continuity $\iota$  depending only on $(\kappa,\upsilon)$ such that  $|z-w|\leq \iota_1(\rho)$ whenever $k\neq 0$ and $g(z,w,k)= 0$. It is well known that $K$-quasiconformal mappings are $\frac1K$-H\"older continuous (see \cite[Corollary 3.10.3]{AstalaIwaniecMartin}, for instance). Thus, the complex geometric optics solution $u_1$ is H\"older continuous in $\D(0,2)$ with constant depending on $|k|$ by \rf{eqRelationfu} and \rf{eqFMuRepresentation}. Therefore
$$|u_1(z,k)-u_2(z,k)|=|u_2(w,k)-u_2(z,k)|\leq C(|k|)|w-z|^{\frac1K}\leq C(|k|)\iota_1(\rho)^\frac1K.$$

 From \rf{eqQuantitativeIota} below,  if $\omega$ is upper semi-continuous then
 \begin{equation*}
\iota_{\mathcal{G}}(\rho)=\iota_1(\rho)^\frac1K
	\leq 4 \upsilon \left(\frac{1}{C_1 |\log(\rho)|-C_2}\right)^\frac1K + \rho^\frac{1}{C_3} .
\end{equation*}
Moreover, by Remark \ref{remQuantitativeUpsilon}, \rf{eqModulusModified} and Theorem \ref{theoDecayOfCGOS}, choosing $q=q(p,\kappa)$ satisfying $\frac K{p} < \frac {1}{q}<\frac12$, we get that $\upsilon$ is bounded above by 
$$\upsilon(t)\leq  C_{\kappa,p} \,  \widetilde{\omega}(2t)^{b_{\kappa,p}} + C_\kappa t^{\alpha_K} \leq C_{\kappa,p} \, \left(C_{K,q,p} \, \omega(C_K (2t)^\frac1K)\right)^{b_{\kappa,p}} + C_\kappa t^{\alpha_K},$$
where the constant $b_{\kappa,p}=\frac{- \log \kappa}{4 \log (\norm{1+ \Beurling}_{\ell,\ell}) -  \log\kappa}$ with $\frac1\ell=\frac12-\frac1q$, and the constant $\alpha_K=\frac1{10}-\frac1{5s}$ with $s=s(\kappa)>2$ given by the relation $\norm{\Beurling}_{s,s}=\frac1{\sqrt{\kappa}}$. Thus,
 \begin{equation*}
\iota_{\mathcal{G}}(\rho)
	\leq 4 C_{\kappa,p} \,  \omega\left(C_K \left(\frac{1}{C_1 |\log(\rho)|-C_2}\right)^\frac1K\right)^\frac{b_{\kappa,p}}K + C_\kappa \left(\frac{1}{C_1 |\log(\rho)|-C_2}\right)^\frac{\alpha_K}K + \rho^\frac{1}{C_3} .
\end{equation*}
For $\rho$ small enough, this estimate can be written as
 \begin{equation*}
\iota_{\mathcal{G}}(\rho)
	\leq C_{\kappa,p} \,  \omega\left(\frac{C_K}{|\log(\rho)|^\frac1K}\right)^{b_{\kappa,p}} +  \frac{C_\kappa}{|\log(\rho)|^{\alpha_K}}.
\end{equation*}
\end{proof}

The proof of Theorem \ref{theoUInftyIsSmall} has been reduced to \rf{eqQuantitativeIota} below and the following result:
\begin{proposition}\label{propoGDoesNotVanishFarAway}
Let $K\geq 1$, let $\upsilon$ be a modulus of continuity,  and let $\gamma_1, \gamma_2 \in \mathcal{G}(K,\D)$ be such that 
\begin{itemize}
\item for a certain $0<\rho\leq 1/2$ and every $k\in\C$, we have
\begin{equation}\label{eqBoundInScattering}
|\tau_1(k)-\tau_2(k)|\leq \rho e^{C|k|},
\end{equation}
\item and for every $z\in\C$ and $j\in\{1,2\}$ we have 
\begin{equation}\label{eqDecayDelta}
|\delta_j(z,k)-izk|\leq |k| \upsilon(|k|^{-1}).
\end{equation}
\end{itemize} 
Consider $z\in 2\D$, $w\in\C$ and  the function $g(k)=g(z,w,k):=\delta_1(z,k)-\delta_2(w,k)$. Then there exists a modulus of continuity $\iota_1$  depending only on $(K,\upsilon)$ such that $g(z,w,k)\neq 0$ whenever $k\neq 0$ and $|z-w|\geq \iota_1(\rho)$. 
\end{proposition}

\begin{proof}[Sketch of the proof]
Let  $z\in 2\D$ and let $w\in\C$ with $|z-w|$ far enough (in fact, $|z-w| >\iota_1(\rho)$ with $\iota_1$ to be determined). We are interested in bounds for $g$ in terms of $|z-w|$. For each distance $s>0$, we will find a circle of radius 
\begin{equation}\label{eqRLambda}
\frac{1}{R(s)}:=\begin{cases}
\upsilon^{-1}\left(s\right) & \mbox{ if }s\leq 1, \\
\upsilon^{-1}(1) & \mbox{ otherwise,}
\end{cases}
\end{equation}
where $\upsilon^{-1}(y) = \sup \{ x\in\R: \upsilon(x) \leq y\}$, and define
$$R_{z,w}:=R\left(\frac{|z-w|}4\right).$$
This is big enough so that we have some control on the decay of certain functions. Namely, inequality \rf{eqDecayDelta} grants that $|g(z,w,k)-i{(z-w)} k|\leq 2  |k| \upsilon(|k|^{-1})$ and thus, if $|k| > R_{z,w}$, a short computation shows that 
\begin{equation}\label{eqDeltaControlOnTheCircle}
\frac{|g(z,w,k)-i{(z-w)} k|}{|{(z-w)} k|}\leq\frac{2 \upsilon(|k|^{-1})}{|z-w|}\leq \frac{1}{2}.
\end{equation}

Thus, $g$ is homotopic to ${(z-w)} k$ in $\partial\D_{R_{z,w}}$ (we omit $z$, $w$ and $k$ in the notation when their values are  clear from the context).  In particular, its Browder degree
$$\deg(g, \partial \D_{R_{z,w}}, 0 )=1.$$
In \rf{eqEtaDefinition}--\rf{eqFDefinition} we will define functions $F$, $S$ and  $\varsigma$, with $F$ holomorphic on $\D(0,R_{z,w})$, $S$ small in the Lipschitz norm and $\varsigma$ continuous, so that $g=e^{\varsigma}(F+S)$. The continuity of $\varsigma$ grants that $g$ is homotopic to $F+S$ as well in $\partial\D_{R_{z,w}}$, so 
$$\deg(F+S, \partial \D_{R_{z,w}}, 0 )=1.$$

In Proposition \ref{propoFHasOneZero} we will see that whenever $|z-w| >\iota_2(\rho)$, this function $F$ vanishes only when $k=0$ and we need to control the zeros which $S$ may add.

Denoting the zeroes of $H=F+S$ by $Z(H)$, from Lemma \ref{lemZerosSmall} below if $|z-w| >\iota_3(\rho)$ then $Z(H)\subset\D_{d_0}$. Moreover, by Proposition \ref{propoDetDHNot0} whenever $|z-w| >\iota_4(\rho)$ we know that ${\rm sgn} \left(\det (DH)\right)=1 $ in $\D_{d_0}$. Since $H\in C^1$ by \rf{eqControlNablaS}, we can write
$$1=\deg(H, \partial \D_{R_{z,w}}, 0 )=\sum_{k_i\in Z(H)} {\rm ind} H(k_i)\geq \sum_{k_i\in Z(H)} {\rm sgn} \left(\det (DH)\right)(k_i) =\# Z(H).$$
Therefore, setting $\iota_1:=\max\{\iota_2,\iota_3,\iota_4\}$, we get that $H$ has only one zero, that is in $k=0$ by \rf{eqEverythingVanishes}, and the same happens to $g$.
\end{proof}

\subsection{A topologic argument}\label{secDetails}
The present section is devoted to providing the missing details in the proof of Proposition \ref{propoGDoesNotVanishFarAway} above. 
First, a technical lemma to be used in the subsequent proofs follows:
\begin{lemma}\label{lemExistsC}
Let $\alpha\in \R_+^4$, and $\upsilon:\R_+ \to \R_+$ increasing.
 There exists a modulus of continuity $\iota = \iota_\alpha$ depending on $\alpha$ and $\upsilon$ such that for $d > \iota(\rho)$ we have 
\begin{equation*}
 \rho^{\alpha_1}  < \frac{d^{\alpha_2}}{\alpha_3  e^{\alpha_4 R(d)}}
\end{equation*}
for every $\rho\in\left(0,\frac12\right]$, where $R(d)$ is defined as in  \rf{eqRLambda} with that given $\upsilon$. In addition, if $\rho$ is small enough and $\upsilon$ is upper semi-continuous, then
\begin{equation}\label{eqQuantitativeIota}
\iota(\rho) 
	\leq \upsilon \left(\frac{\alpha_4}{\alpha_1 |\log(\rho^\frac12)|-\log(\alpha_3)}\right) +   \rho^\frac{\alpha_1}{2\alpha_2} .
\end{equation}
\end{lemma}

\begin{proof}
Indeed, it is enough to show that 
$$ \alpha_1 \log(\rho) < \alpha_2 \log(d)-\log (\alpha_3) - \alpha_4 R(d).$$
Since $0<\rho\leq \frac12$, this is equivalent to
$$  |\log(\rho)| > \frac{\alpha_4 R(d) - \alpha_2 \log(d) + \log (\alpha_3) }{\alpha_1}= \frac{R(d)}{\beta_1} - \beta_2 \log(d) + \beta_3.$$

The term in the right-hand side is a strictly decreasing function on $d$, and it tends to infinity as $d$ goes to zero. Thus,
$$\iota(t) :=  \inf \left\{ x \in\R_+: \frac{R(x)}{\beta_1} - \beta_2 \log(x) + \beta_3< |\log(t)|\right\}$$
is an increasing function on $t$, and it satisfies that
$$\lim_{t\to 0} \iota(t) = 0.$$
Moreover, 
\begin{align*}
\iota(t) 
	& \leq  \inf \left\{ x \in\R_+: \frac{R(x)}{\beta_1} + \beta_3<\frac{ |\log(t)|}2\right\} \cap \left\{ x \in\R_+:  - \beta_2 \log(x) < \frac{|\log(t)|}2 \right\} \\
	& = \max \left\{ \inf \left\{ x \in\R_+: \frac1{R(x)} > \frac1{\beta_1 |\log(t^\frac12)|-\beta_1\beta_3}\right\}, \inf \left\{ x \in\R_+:  x >  t^\frac1{2\beta_2} \right\} \right\}.
\end{align*}
Note that we used the fact that the intersection of two rays containing $+\infty$ is one of the rays. Under the assumption of upper semicontinuity of $\upsilon$, since it is increasing, we get that for every $y>0$, $\upsilon\circ \upsilon^{-1}(y) \geq y$. By \rf{eqRLambda}, when $x<1$ we defined  $R(x)^{-1}=\upsilon^{-1}(x)$. For $t$ small enough we get
\begin{align*}
\iota(t) 
	& \leq \max \left\{ \upsilon \left(\frac1{\beta_1 |\log(t^\frac12)|-\beta_1\beta_3}\right),  t^\frac1{2\beta_2}  \right\} .
\end{align*}
\end{proof}

\begin{example}
If $\upsilon(t)= t^{a}$ with $0<a<1$, then 
\begin{equation*}
R(s):=\begin{cases}
s^{-\frac1a} & \mbox{ if } s\leq 1, \\
1 & \mbox{ otherwise.}
\end{cases}
\end{equation*}
In this case, for $t$ small enough, $\iota(t) \leq (4+\varepsilon) \left(\frac{\alpha_1 }{\alpha_4} |\log(t)| \right)^{-a} $.
\end{example}

The remaining part of this section follows closely the scheme in \cite[Section 5]{BarceloFaracoRuiz}, and the expert reader may skip it. We include the details for the sake of completeness and to keep track of the moduli of continuity.

Let us begin with the construction of $F$, $S$ and $\varsigma$. First of all let us introduce two auxiliary functions which arise from equation \rf{eqTransportTau} written in terms of $g$. Since $u_j=e^{\delta_j}$, we have that 
\begin{equation}\label{eqDeltaBarDerivative}
\partial_{\bar{k}}\delta_j = \frac{\partial_{\bar{k}}  u_j}{u_j}=-i\tau_j \frac{\overline{u_j}}{u_j}=-i\tau_j e^{\overline{\delta_j}-\delta_j}
\end{equation}
Therefore,
$$\partial_{\bar{k}} g=\partial_{\bar{k}}\delta_1 -\partial_{\bar{k}}\delta_2 = -i(\tau_1- \tau_2) e^{\overline{\delta_1}-\delta_1} -i \tau_2 \left(e^{\overline{\delta_1}-\delta_1}-e^{\overline{\delta_2}-\delta_2}\right),$$
leading to 
\begin{equation}\label{eqGEquation}
\partial_{\bar{k}} g = \sigma g + E,
\end{equation}
where
\begin{equation*}
\sigma(z,w,k): = 
	\begin{cases}
		-i \tau_2(k) \frac{\left(e^{\overline{\delta_1(z,k)}-\delta_1(z,k)}-e^{\overline{\delta_2(w,k)}-\delta_2(w,k)}\right)}{\delta_1(z,k)-\delta_2(w,k)} & \mbox{ if } \delta_1(z,k)\neq \delta_2(w,k),\\
		0 & \mbox{ otherwise.}
		\end{cases}
\end{equation*}
and
\begin{equation*}
E(z,k): = -i(\tau_1(k)- \tau_2(k)) e^{\overline{\delta_1(z,k)}-\delta_1(z,k)}.
\end{equation*}

Before going on we need to establish some bounds in these functions, in order to define their Cauchy transforms.
\begin{lemma}\label{lemBoundsOnSigmaE}
Under the hypothesis of Proposition \ref{propoGDoesNotVanishFarAway}, for every $z\in 2\D$ and $w\in\C$ we have that
\begin{equation}\label{eqSigmaInftyBound}
\norm{\sigma(z,w,\cdot)}_{L^\infty(\C)}\leq2\norm{\tau_2}_{L^\infty(\C)}\leq 2,
\end{equation}
\begin{equation}\label{eqEInftyBound}
|E(z,k)|\leq \rho e^{C|k|},
\end{equation}
and
\begin{equation}\label{eqDEInftyBound}
|\nabla_kE(z,k)|\leq e^{C(1+|k|)}.
\end{equation}
\end{lemma}

\begin{proof}
Let us begin bounding $\sigma$. Note that $\overline{\delta_1}-\delta_1\in i\R$, and the Lipschitz constant $1$ on $e^{i\theta}$  in $\theta\in\R$ leads to 
$$|\sigma(z,w,k)|\leq |\tau_2|\frac{|\overline{\delta_1}-\delta_1-(\overline{\delta_2}-\delta_2)|}{|\delta_1-\delta_2|}\leq 2|\tau_2|.$$
Using \rf{eqTauBound},  we obtain \rf{eqSigmaInftyBound}.
Moreover, we have $|e^{i\theta}|=1$ and, by hypothesis, $|\tau_1(k)-\tau_2(k)|\leq \rho e^{C|k|}$. By these reasons, we get that
\begin{equation}\label{eqUncalled}
|E(z,k)|\leq \min\{2,\rho e^{C|k|}\},
\end{equation}
proving \rf{eqEInftyBound}. Finally, 
\begin{align*}
|\partial_k E(z,k)|
	& =  \left|(\partial_k\tau_1 - \partial_k \tau_2) e^{\overline{\delta_1}-\delta_1}+(\tau_1- \tau_2) e^{\overline{\delta_1}-\delta_1}\left(\partial_k\overline{\delta_1}-\partial_k\delta_1\right)\right|\\
	& \leq  \left|\partial_k\tau_1\right| + \left| \partial_k \tau_2\right| +\left|\tau_1- \tau_2\right| \left|\partial_k\overline{\delta_1}\right|+\left|\tau_1- \tau_2\right|\left|\partial_k\delta_1\right|.
\end{align*}
Using  \rf{eqTauDerivative}, \rf{eqBoundInScattering}, \rf{eqDeltaBarDerivative} and then \rf{eqTauBound} and the hypothesis $\rho\leq1/2$, we get
\begin{align*}
|\partial_k E(z,k)|
	& \leq  e^{C(1+|k|)} + \rho e^{C|k|}\left( \left|\tau_1 e^{\overline{\delta_1}-\delta_1}\right|+\left|\partial_k\delta_1\right|\right) \leq  e^{C(1+|k|)} \left( 1+\left|\partial_k\delta_1\right|\right).
\end{align*}
Arguing analogously, we get that $|\partial_{\bar{k}} E(z,k)| \leq  e^{C(1+|k|)} \left( 1+\left|\partial_k\delta_1\right|\right)$.
Since $z\in 2\D$, in order to show \rf{eqDEInftyBound} it suffices to show that 
\begin{equation}\label{eqObjectivePartialDelta}
|\partial_k\delta_1|\leq e^{C(1+|k|)}(1+|z|).
\end{equation}
To do so, by means of \rf{eqRelationfu} and Lemma \ref{lemUgammaFLambdaMu} we get
\begin{equation*}
\left|\partial_{{k}}\delta_1\right| = \frac{|\partial_{{k}}  u_1|}{|u_1|}\leq \frac{  |\nabla_k  f_{\mu}|+ |\nabla_k  f_{- \mu}|}{\inf_{\lambda\in\partial\D} |f_{\lambda \mu}|}.
\end{equation*}
We write this expression in terms of the Jost functions using \rf{eqBeltramiAssimptotics} and by \rf{eqFMuRepresentation} we get
\begin{align}\label{eqBoundDelta1}
\left|\partial_{{k}}\delta_1\right| 
\nonumber	& \lesssim \frac{\sup_{\lambda\in\partial\D} |e^{ikz}| |\nabla_k  M_{\lambda \mu}|+|ze^{ikz}| | M_{\lambda \mu}|}{\inf_{\lambda\in\partial\D} |e^{ikz}||M_{\lambda\mu}|} \leq \frac{\sup_{\lambda\in\partial\D} \norm{\nabla_k  M_{\lambda \mu}}_{L^\infty}+ |z|  |e^{ik(\varphi_{\lambda\mu}-z)}|}{\inf_{\lambda\in\partial\D} |e^{ik(\varphi_{\lambda\mu}-z)}|}\\
	& \leq   \left(\sup_{\lambda\in\partial\D} \norm{\nabla_k  M_{\lambda \mu}}_{L^\infty} +  |z|  \right) e^{2|k|\sup_{\lambda\in\partial\D} |\varphi_{\lambda \mu}(z,k)-z|} .
\end{align}
By the Sobolev Embedding Theorem and \rf{eqMMuDerivative} we get that
\begin{equation}\label{eqBoundOnSupM}
\sup_{\lambda\in\partial\D} \norm{\nabla_k  M_{\lambda \mu}(\cdot,k)}_{L^\infty}\leq e^{C(1+|k|)}.
\end{equation}
On the other hand,  the quasiconformal principal mapping $\varphi=\varphi_\mu(\cdot,k)$ with Beltrami coefficient $\nu=\frac{\bar{\partial}\varphi}{\overline{\partial\varphi}}$ such that $\varphi(z) - z= \Cauchy(\bar{\partial}\varphi)(z)$ satisfies that for $p>2$ and close enough to $2$,
\begin{equation}\label{eqBoundOnVarphi}
\norm{\varphi - Id}_{L^\infty(\C)}\lesssim \norm{\bar{\partial}\varphi}_{L^{p}(\C)}\leq \norm{(I-\nu\Beurling)^{-1}}_{(p,p)}\norm{\nu}_{L^p}\lesssim C_\kappa.
\end{equation}
Note that in the first step we used the Sobolev embedding Theorem again and \rf{eqCauchyCompactlyBounded}, and then the Neumann series \rf{eqPrincipalNeumann}. By \rf{eqBoundDelta1}, \rf{eqBoundOnSupM} and \rf{eqBoundOnVarphi}, we get \rf{eqObjectivePartialDelta}.
\end{proof}

Next, for each $R$ we consider a bump function $\varphi_{R}\in C^\infty_0(\C)$ such that $\chi_{\D(0,R)}\leq \varphi_R\leq\chi_{\D(0,2R)}$ and $\norm{\nabla\varphi_R}_{L^\infty}\lesssim \frac{1}{R}$. Now take
\begin{equation}\label{eqEtaDefinition}
\varsigma(z,w,k):=\Cauchy\left(\sigma(z,w,\cdot) \varphi_{R_{z,w}}(\cdot)\right)(k),
\end{equation}
which is well defined by \rf{eqSigmaInftyBound} and \rf{eqCauchyCompactlyBounded}, and it is locally H\"older continuous by \rf{eqCauchyBounded}. Thus, using \rf{eqEInftyBound} we can define
\begin{equation}\label{eqSDefinition}
S(z,w,k):=\Cauchy\left( e^{-\varsigma(z,w,\cdot)}E(z,\cdot) \varphi_{R_{z,w}}(\cdot)\right)(k)-\Cauchy\left(e^{-\varsigma(z,w,\cdot)}E(z,\cdot) \varphi_{R_{z,w}}(\cdot)\right)(0).
\end{equation}
The purpose of all this procedure is to get a holomorphic approximation of $g$ on $\D(0,R_{z,w})$: take
\begin{equation}\label{eqFDefinition}
F(z,w,k):=e^{-\varsigma(z,w,k)}g(z,w,k) - S(z,w,k).
\end{equation}
Using the fact that $\bar{\partial}\circ\Cauchy = Id$ on Sobolev functions and \rf{eqGEquation} we get
\begin{align*}
\partial_{\bar{k}}F
	& =e^{-\varsigma}\partial_{\bar{k}}g-  e^{-\varsigma}g\, \partial_{\bar{k}}\varsigma -\partial_{\bar{k}}S =e^{-\varsigma} \left(\sigma g +E  - \sigma \varphi_{R_{z,w}} g\right) -  e^{-\varsigma}E \varphi_{R_{z,w}} \\
	& = \left(1-\varphi_{R_{z,w}}\right) e^{-\varsigma} \left(\sigma g +E\right).
\end{align*}
By Weyl's Lemma, F is holomorphic on $\D(0,R_{z,w})$. Note that putting together Definition \ref{defFMu}, \rf{eqSDefinition} and \rf{eqFDefinition}, regardless of the values of $z,w$,  we have that
\begin{equation}\label{eqEverythingVanishes}
F(z,w,0)=S(z,w,0)=g(z,w,0)=0.
\end{equation}
Let us complete the bounds on $\varsigma$ and on $S$ and its derivatives.
\begin{lemma}
Under the hypothesis of Proposition \ref{propoGDoesNotVanishFarAway}, for every $z\in 2\D$ and $w\in\C$ we have that 
\begin{equation}\label{eqEtaLInftyBound}
\norm{\varsigma(z,w,\cdot)}_{L^\infty}\leq C R_{z,w}.
\end{equation}
Moreover,
\begin{equation}\label{eqSLInftyBound}
\norm{S(z,w,\cdot)}_{L^\infty} \leq \rho e^{CR_{z,w}} 
\end{equation}
and, for every $\theta < 1$ there exists $C$ depending on $\theta$ and $\upsilon^{-1}(1)$ such that 
\begin{equation}\label{eqSW1InftyBound}
\norm{\nabla_k S(z,w,\cdot)}_{L^\infty}\leq\rho^\theta e^{C(R_{z,w}+1)} .
\end{equation}
\end{lemma}

\begin{proof}
The first equation follows easily from \rf{eqEtaDefinition}, and \rf{eqSigmaInftyBound}:
\begin{equation*}
\norm{\varsigma(z,w,\cdot)}_{L^\infty}\lesssim\int_{\D_{2R_{z,w}}}\frac{|\sigma(t)|}{|k-t|} \, dm(t) \lesssim R_{z,w}.
\end{equation*}
Using that $\partial  = \Beurling\circ \bar{\partial}$, the boundedness of the Beurling transform in any $L^p$ space for $1<p<\infty$ and $\bar{\partial}\Cauchy=Id$, together with \rf{eqSigmaInftyBound} again, we get
\begin{equation}\label{eqEtaW1InftyBound}
\norm{\partial_k \varsigma}_{L^p}=\norm{\Beurling(\partial_{\bar k} \varsigma)}_{L^p}\leq C_p \norm{\partial_{\bar k} \varsigma}_{L^p}=C_p \norm{\sigma \varphi_{R_{z,w}}}_{L^p}\leq C R_{z,w}^{\frac{2}{p}}
\end{equation}
for every such $p$.

Regarding $S$, by \rf{eqEtaLInftyBound} and \rf{eqEInftyBound} we obtain
\begin{align*}
|S(z,w,k)|
\nonumber	& =\left|\Cauchy(e^{-\varsigma} E\varphi_{R_{z,w}})(k)-\Cauchy(e^{-\varsigma} E\varphi_{R_{z,w}})(0)\right|\\
	& \leq \int_{2\D_{R_{z,w}}}e^{|\varsigma(z,w,t)|}|E(z,w,t)|\left(\frac{1}{|k-t|} + \frac{1}{|t|}\right)\, dm(t)\\
	& \leq C e^{CR_{z,w}} \rho e^{CR_{z,w}}\int_{2\D_{R_{z,w}}}\left(\frac{1}{|k-t|} + \frac{1}{|t|}\right)\, dm(t) \leq  C \rho e^{CR_{z,w}}R_{z,w} ,
\end{align*}
which implies \rf{eqSLInftyBound}.

On the other hand, by the properties of the Cauchy transform $\partial_{\bar{k}} S=e^{-\varsigma} E\varphi_{R_{z,w}}$. Using \rf{eqEtaLInftyBound} and \rf{eqEInftyBound} we get that
$$\norm{\partial_{\bar{k}}S}_{L^\infty}=\norm{e^{-\varsigma} E\varphi_{R_{z,w}}}_{L^\infty}\leq  \rho e^{CR_{z,w}},$$
and for every $1<p<\infty$,  arguing as before
$$\norm{\partial_k S}_{L^p}=\norm{\Beurling(\partial_{\bar k} S)}_{L^p}\leq C_p \norm{\partial_{\bar k} S}_{L^p}\leq C_p R_{z,w}^{\frac{2}{p}} \norm{\partial_{\bar k} S}_{L^\infty}\leq \rho e^{C(R_{z,w}+1)}.$$
However, we need an $L^\infty$ bound, which we will obtain by interpolation. Combining \rf{eqEtaLInftyBound} and \rf{eqEtaW1InftyBound} with \rf{eqEInftyBound} and \rf{eqDEInftyBound}, we get
\begin{align*}
\norm{\nabla_k(e^{-\varsigma} E\varphi_{R_{z,w}})}_{L^p}
	& \leq e^{\norm{\varsigma}_{\infty}}\left(\norm{E}_{L^\infty(\D_{2R})}\norm{\nabla\varphi_{R_{z,w}}}_{L^p}+\norm{E}_{L^\infty(\D_{2R})}\norm{\nabla_k\varsigma}_{L^p}+\norm{\nabla_k E}_{L^p(\D_{2R})}\right)\\
	& \leq  e^{CR_{z,w}} \left( \rho e^{CR_{z,w}}  R_{z,w}^{\frac{2}{p}-1}+\rho e^{CR_{z,w}}C R_{z,w}^{\frac{2}{p}}+e^{C(R_{z,w}+1)}R_{z,w}^{\frac{2}{p}}\right)\\
	& \leq   R(1)^{\frac2p-1}e^{C(R_{z,w}+1)},
\end{align*}
where $\D_{2R}:=\D_{2R_{z,w}}$. Note that we have used implicitly the hypothesis $z\in 2\D$ to control $\norm{\nabla_k E}_{L^p(\D_{2R})}$ and the fact that $\rho\leq 1/2$ to keep the constants in the exponent. To end, let $\theta<1$ and choose $\frac{2}{1-\theta}<p<\infty$. Then by the embedding properties of fractional Sobolev spaces and interpolation between these spaces (see \cite[Theorems 2.4.7, 2.5.6, 2.5.7 and 2.7.1]{TriebelTheory}, for instance) we have that 
\begin{align}\label{eqControlNablaS}
\norm{\nabla_k S}_{L^\infty}
\nonumber	& \leq C\norm{\nabla_k S}_{W^{1-\theta,p}}
			\leq C_{\theta,p} \norm{\nabla_k S}_{W^{1,p}}^{1-\theta}\norm{\nabla_k S}_{L^p}^\theta \\
	& \leq C e^{C(R_{z,w}+1)(1-\theta)} \rho^{\theta} e^{C(R_{z,w}+1)\theta}.
 \end{align}
\end{proof}

Next we see that $F$ has only one zero if $|z-w|$ is big enough with respect to $\rho$ (i.e., when the perturbation $S$ is small with respect to $F$).
\begin{proposition}\label{propoFHasOneZero}
Under the hypothesis of Proposition \ref{propoGDoesNotVanishFarAway},  there exists a modulus of continuity $\iota_2$ such that for every $z\in 2\D$ and $w\in\C$ with $|z-w| > \iota_2(\rho)$, then 
$$F(z,w,k)=0 \iff k=0$$
and the zero is simple.
\end{proposition}
\begin{proof}
By \rf{eqEverythingVanishes} the ``if'' implication is trivial. Since $F$ is holomorphic in $\D_{R_{z,w}}$, the function $k\mapsto \frac{F(k)}{k}$ is holomorphic in $\D_{R_{z,w}}$ as well.

Let us assume that $|z-w|\geq \iota_2(\rho)$ with $\iota_2$ to be fixed and $|k|\geq R_{z,w}$. We will check that
\begin{equation}\label{eqFAndEtaStayCloseInTheCircleBelow}
\real\left(\frac{F(z,w,k)}{i{(z-w)} k e^{-\varsigma}}\right)\geq\frac{1}{4} .
\end{equation}
Indeed, using \rf{eqFDefinition} we can write
$$\frac{F(k)}{i{(z-w)} k e^{-\varsigma}}=\frac{[i{(z-w)} k + (g-i{(z-w)} k) ]e^{-\varsigma} -S}{i{(z-w)} k e^{-\varsigma}}=1+\frac{ g(k)-i{(z-w)} k}{i{(z-w)} k}-\frac{S(k)}{i{(z-w)} k e^{-\varsigma}}.$$
Therefore, by \rf{eqDeltaControlOnTheCircle}, \rf{eqEverythingVanishes} and \rf{eqSLInftyBound}, we have that
$$\real\left(\frac{F(k)}{i{(z-w)} k e^{-\varsigma}}\right)\geq 1-\left|\frac{ g-i{(z-w)} k}{i{(z-w)} k}\right|-\left|\frac{S(k)-S(0)}{k }\frac{1}{{(z-w)} e^{-\varsigma}}\right| \geq \frac{1}{2}-\frac{\norm{\nabla_k S}_{L^\infty}}{|z-w| e^{-\norm{\varsigma}_{L^\infty}}}.$$
Thus, to show \rf{eqFAndEtaStayCloseInTheCircleBelow} it suffices to see that 
\begin{equation}\label{eqNablaSImprovingStayClose}
\norm{\nabla_k S}_{L^\infty}\leq\frac{ |z-w|}{4 e^{\norm{\varsigma}_{L^\infty}}}
\end{equation}
 for a convenient $\iota_2$.

By \rf{eqSW1InftyBound} we have that $\norm{\nabla_k S}_{L^\infty}\leq \rho^\frac12 e^{C(R_{z,w}+1)}$. Therefore, for \rf{eqNablaSImprovingStayClose} to hold we only need that
\begin{equation*}
 \rho^\frac12  \leq \frac{|z-w|}{\widetilde{C} e^{(C+\norm{\varsigma}_{L^\infty}) R_{z,w}}}
\end{equation*}
for a convenient $\iota_2=\iota_{(\frac12,1,\widetilde{C},C+\norm{\varsigma}_{L^\infty})}$,  and $|z-w|\geq \iota_2(\rho)$, and this is a consequence of Lemma \ref{lemExistsC}. The proof of \rf{eqFAndEtaStayCloseInTheCircleBelow} is complete.

Inequality \rf{eqFAndEtaStayCloseInTheCircleBelow} implies that, in particular, for $|k|\geq R_{z,w}$, the function $F(k)$ does not vanish. To end, we need to see that when $0<|k|<R_{z,w}$ still $F(k)\neq 0$. Now, for $t\in(0,1)$, if
$$t \frac{F(k)}{k} + (1-t) i{(z-w)} e^{-\varsigma}=0,$$
then $\frac{F(k)}{i{(z-w)} k e^{-\varsigma}}=-\frac{1-t}{t}$, which is impossible when  $|k|\geq R_{z,w}$ by \rf{eqFAndEtaStayCloseInTheCircleBelow}.  Thus, $\frac{F(k)}{k}$ is holomorphic and homotopic to $i{(z-w)} e^{-\varsigma}$ in $\partial\D_{R_{z,w}}$, and the latter is homotopic to the constant function $k\mapsto i {(z-w)}$. Thus, the three of them have the same number of zeros, that is, none of them has zeroes. 
\end{proof}

\begin{remark}
The same reasoning used above to show \rf{eqFAndEtaStayCloseInTheCircleBelow}  can be used to see that for $|k|\geq R_{z,w}$ with $|z-w|\geq  \iota_2(\rho)$, the estimate
\begin{equation}\label{eqFAndEtaStayCloseInTheCircleAbove}
\left|\frac{F(z,w,k)}{i{(z-w)} k e^{-\varsigma}}\right|\leq \frac{7}{4} 
\end{equation}
holds.
\end{remark}

Next we want to see that $F$ behaves as ${(z-w)} k$ near the origin.

\begin{lemma}\label{lemNuIsBounded}
Under the hypothesis of Proposition \ref{propoGDoesNotVanishFarAway}, there exists a constant $M_0$ such that for every $z\in 2\D$ and $w\in\C$ with $|z-w|\geq  \iota_2(\rho) $,  there exists a holomorphic function $\nu$ on $\D_{R_{z,w}}$ with 
\begin{equation}\label{eqFAsAnExponential}
F(k)=i{(z-w)} k e^{\nu(k)} \mbox{\quad\quad for every }k\in\D_{R_{z,w}},
\end{equation}
 and $|\nu(k)|\leq M_0 (R_{z,w}+1)$.
\end{lemma}

\begin{figure}[ht]
\center
\begin{tikzpicture}[line cap=round,line join=round,>=triangle 45,x=1.6cm,y=1.6cm]
\draw[->,color=black] (-2.387559499715807,0.) -- (2.590889088635819,0.);
\foreach \x in {-2.,-1.5,-1.,-0.5,0.5,1.,1.5,2.,2.5}
\draw[shift={(\x,0)},color=black] (0pt,2pt) -- (0pt,-2pt) node[below] {\footnotesize $\x$};
\draw[->,color=black] (0.,-2.0050839418500694) -- (0.,2.020592589075628);
\foreach \y in {-2.,-1.5,-1.,-0.5,0.5,1.,1.5,2.}
\draw[shift={(0,\y)},color=black] (2pt,0pt) -- (-2pt,0pt) node[left] {\footnotesize $\y$};
\draw[color=black] (0pt,-10pt) node[right] {\footnotesize $0$};
\clip(-2.387559499715807,-2.0050839418500694) rectangle (2.590889088635819,2.020592589075628);
\draw [shift={(0.,0.)},fill=black,fill opacity=0.25]  plot[domain=-1.4274487578895307:1.4274487578895312,variable=\t]({1.*1.75*cos(\t r)+0.*1.75*sin(\t r)},{0.*1.75*cos(\t r)+1.*1.75*sin(\t r)});
\draw (0.25,1.7320508075688772)-- (0.25,-1.7320508075688774);
\end{tikzpicture}
\caption{Region containing the image of $\partial \D_{R_{z,w}}$ under $e^{\nu(k)+\varsigma(k)}$. }\label{figRecorregut}
\end{figure}
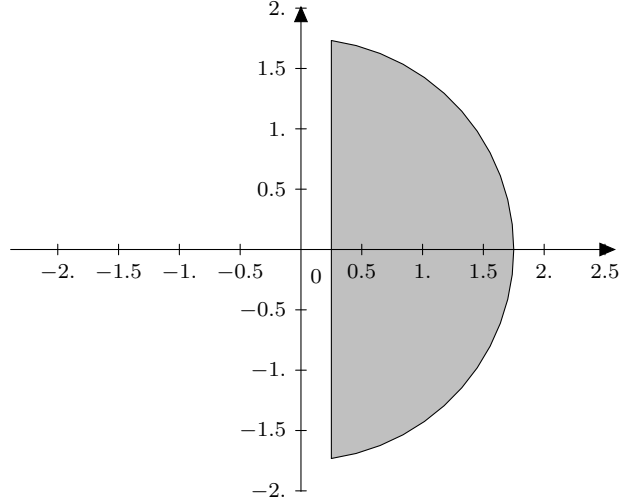

\begin{proof}
By Proposition \ref{propoFHasOneZero}, the function $k\mapsto \frac{F(k)}{i{(z-w)} k}$ is analytic with no zeroes on $\D_{R_{z,w}}$. Thus, $\frac{F(k)}{i{(z-w)} k}=e^{\nu(k)}$ with $\nu$ holomorphic on the considered domain. 

Let $|k|=R_{z,w}$. By \rf{eqFAndEtaStayCloseInTheCircleAbove}, we have that
$$ \left|e^{\nu(k)+\varsigma(k)}\right|
	=  \left|\frac{F(k)}{i{(z-w)} k e^{-\varsigma}}\right|
	\leq \frac{7}{4}.$$
Moreover, by \rf{eqFAndEtaStayCloseInTheCircleBelow} we have that 
$$\real \left(e^{\nu(k)+\varsigma(k)}\right)
	=\real \left( \frac{F(k)}{i{(z-w)} k e^{-\varsigma}}\right)
	\geq \frac{1}{4}$$
(see Figure \ref{figRecorregut}).	
	
Combining both estimates and using the principal branch of the logarithm, we can choose a determination of $\nu$ so that $\nu+\varsigma$ maps $\partial\D_{R_{z,w}}$ to the rectangle
$$\left\{\xi \in\C: -\log(4)\leq \real(\xi)\leq \log\left(\frac{7}{4}\right) \mbox{ and } |\imag(\xi)|<\frac{\pi}{4}\right\}.$$
Thus,  we have that $|\nu(k)|\leq 4 + \norm{\varsigma}_{L^\infty}\leq M_0 (R_{z,w}+1)$ in the circumference $|k|=R_{z,w}$ by \rf{eqEtaLInftyBound}. By the maximum principle   this result extends to the whole disk.
\end{proof}

\begin{lemma}\label{lemBigDerivativeAndKoebe}
Under the hypothesis of Proposition \ref{propoGDoesNotVanishFarAway}, there exists $d_0<\frac12$ such that for every $z\in 2\D$ and $w\in\C$ with $|z-w|\geq  \iota_2(\rho)$ as in the previous lemma the following statements hold true:
\begin{enumerate}[i)]
\item For every $\delta>0$, we have $F^{-1}(\D(0,\delta)) \subset \D\left(0,\frac{\delta e^{M_0( R_{z,w}+1)}}{|z-w|}\right)$,
\item $\inf_{|k|<d_0} |F'(k)|>\frac{1}{2}|z-w|e^{-M_0 (R_{z,w}+1)}$ and
\item $\sup_{|k|\leq R_{z,w}} |F'(k)|\leq |z-w| e^{C (R_{z,w}+1)}$.
\end{enumerate}
\end{lemma}

\begin{proof}
The first statement follows from the definitions. Indeed, let $k$ be such that $|F(k)|<\delta$. Then, by \rf{eqFAsAnExponential} we have that  $\left|{(z-w)} k e^{\nu(k)}\right| <\delta$ and $|k|<\frac{\delta e^{M_0(R_{z,w}+1)}}{|z-w|}$.

The second and the third can be shown by means of the Cauchy integral formula. Namely, for $|k|\leq d_1<\min\left\{ \frac{1}{2}, \frac{R_{z,w}}{2}\right\}$,  we have that
$$|\nu'(k)|\leq \left|\frac{1}{2\pi}\int_{\partial\D_{R_{z,w}/4}}\frac{\nu(k)}{(\xi - k)^2}\, d\xi \right|\leq \frac{M_0 (R_{z,w}+1)}{2\pi (R_{z,w}/4)^2}2\pi R_{z,w}/4=4M_0\left(1+\frac1{R_{z,w}}\right).$$

Since $F'(k)={(z-w)} e^{\nu(k)} \left(1  +  k \nu'(k)\right)$, we have $|F'(k)|=|z-w| \left|e^{\nu(k)}\right| \left|1  +  k \nu'(k)\right|$. By Lemma \ref{lemNuIsBounded}, the two claims follow choosing  $d_0< \min \left\{ \frac{1}{8M_0 (1+\upsilon^{-1}(1))}, d_1\right\}$.
\end{proof} 

The idea to conclude comes from the fact that ${(z-w)} k$ cannot be intersected twice by $S$ which is small in $W^{1,\infty}$. By Lemma \ref{lemBigDerivativeAndKoebe} one expects the same for $F$. In the following lemma we deal with the distribution of the zeroes of $g(z,w,\cdot)$ when $z$ and $w$ are far from each other. Consider the set of zeroes
$$Z(g):=\{k:g(z,w,k)=0\}.$$

\begin{lemma}\label{lemZerosSmall}
Let us assume the hypothesis of Proposition \ref{propoGDoesNotVanishFarAway}.  There exists a modulus of continuity $\iota_3$ such that if $|z-w| > \iota_3(\rho)$, then $Z(g)\subset \D(0,d_0)$.
\end{lemma}

\begin{proof}
Let $|z-w|> \iota_2(\rho)$. From \rf{eqFDefinition} we have that $g=e^{\varsigma}(F+S)$. Thus, if $k\in Z(g)$ then $F(k)=-S(k)$ and we can apply \rf{eqSLInftyBound} to get $|F(k)|\leq \rho e^{CR_{z,w}}$. That is, $k\in F^{-1}\left(\D\left(0,\rho e^{CR_{z,w}}\right)\right)$ and, by the first statement in Lemma \ref{lemBigDerivativeAndKoebe}, we have $|k|\leq \rho \frac{e^{(C+M_0)(R_{z,w}+1)}}{|z-w|}$. 

We only need to see that there exists a convenient $\iota_3$, so that $\rho \leq \frac{d_0 \, |z-w|}{e^{(C+M_0)(R_{z,w}+1)}}$ when $|z-w|> \iota_3(\rho)$, which can be done by  Lemma \ref{lemExistsC}.
\end{proof}

The last ingredient in the proof of Proposition  \ref{propoGDoesNotVanishFarAway} above is to compute the Jacobian determinant of 
$$H:=e^{-\varsigma}g = F+S.$$
\begin{proposition}\label{propoDetDHNot0}
Let us assume the hypothesis of Proposition \ref{propoGDoesNotVanishFarAway}. There exists a modulus of continuity $\iota_4$ such that  if $|z-w|>\iota_4(\rho)$ then $\det(D_k H)(k)> 0$ for $k\in\D(0,d_0)$.
\end{proposition}

\begin{proof}
Let $|z-w|>\iota_3(\rho)$. Then, since $F$ is holomorphic in $\D_{d_0}$, by the   cosine formula we get
\begin{align*}
\det (D_k H)
	& =|\partial_k H|^2-|\partial_{\bar{k}}H|^2
		= |\partial_{{k}}F + \partial_{{k}}S|^2-|\partial_{\bar{k}}S|^2\\
	& \geq |F'|^2 + |\partial_{{k}}S|^2-2 |F'||\partial_{{k}}S|-|\partial_{\bar{k}}S|^2
	\geq |F'|^2 - 2 |F'||\nabla S|-|\nabla S|^2 .
\end{align*}
By \rf{eqSW1InftyBound} and the second and the third statements of Lemma \ref{lemBigDerivativeAndKoebe} we have that
\begin{align*}
\det (D_k H)
	& \geq  \left(\frac{1}{2}|z-w|e^{-M_0 (R_{z,w}+1)}\right)^2 - \rho^\theta e^{C(R_{z,w}+1)} \left(2  |z-w| e^{C (R_{z,w}+1)}+\rho^\theta e^{C(R_{z,w}+1)} \right),
\end{align*}
so the condition 
\begin{align*}
\rho^\theta 
	& < \frac{|z-w|^2 }{8 e^{(2C+M_0)(R_{z,w}+1)}  } \min\left\{\frac{1 }{2|z-w|},\frac{1}{\rho^\theta }\right\}
		\leq\frac{|z-w|^2 }{4 e^{(2C+M_0) (R_{z,w}+1)} \left(2 |z-w| +\rho^\theta \right)}\\
	&	= \frac{\left(\frac{1}{2}|z-w|e^{-M_0 (R_{z,w}+1)}\right)^2 }{e^{C(R_{z,w}+1)} \left(2 |z-w| e^{C (R_{z,w}+1)} + \rho^\theta e^{C(R_{z,w}+1)} \right)}
\end{align*}
implies that $\det (D_k H)>0$, and, therefore, the proposition follows. Again, this is a consequence of Lemma \ref{lemExistsC}.
\end{proof}

\section{Final approach}\label{secFinal}
To end we need to combine the  a priori uniform elliptic estimates on the conductivities with the control obtained in Theorem \ref{theoUInftyIsSmall} to obtain estimates in the distance between conductivities. Since interpolation is not possible in our setting, we will perform a subtle argument combining the division in high and low frequencies of the derivatives of the CGOS in the Fourier side with Lemma \ref{lemFourier}. Thus, we need some control on the integral moduli of the CGOS in terms of the moduli of the conductivities.

\subsection{Caccioppoli inequalities}\label{secCaccioppoli}
In Lemma \ref{lemModulusIOfIterates} we have studied the modulus of continuity of a Neumann series. However, the Beltrami equation together with \rf{eqModulusProduct} gives us bounds for the modulus of continuity of the Complex Geometric Optics Solution derivatives, as we will see below, by means of a Caccioppoli inequality.

\begin{theorem}\label{theoCaccioppoliModulus}
Let $\mu,\nu \in L^\infty$ be compactly supported with $\norm{|\mu|+|\nu|}_{L^\infty} \leq \kappa< 1$. Let $f$ be a quasiregular solution to 
$$	\bar\partial f = \mu \,  \partial f + \nu \,  \overline{\partial f }. $$
Let $1<p<p_\kappa$ satisfy that $\kappa \norm{\Beurling}_{L^p\to L^p}<1$, let $r\in [p, p_\kappa)$ and let $q$ defined by $\frac1p=\frac1q+\frac1r$. Then, for every real-valued, compactly supported Lipschitz function $\varphi$, we have that
\begin{equation}\label{eqOptimalModulus}
\modulus{p}{(\varphi \bar\partial f)}(t)
	 \leq  C_{\kappa,r,p} \norm{f \nabla \varphi}_{L^r}\left(  \modulus{q}{\mu}(t)+  \modulus{q}{\nu}(t)\right)  +C_{\kappa,p} \, \modulus{p}{(f \nabla\varphi)}(t).
\end{equation}
\end{theorem}

\begin{proof}
We will show the case $\nu=0$, leaving the general case to the reader. 

Let $F:=\varphi f$. We have $\bar\partial F =  \varphi \bar\partial f + f \bar\partial \varphi = \varphi \mu \partial f + f \bar\partial \varphi  = \mu \partial F + f (\bar\partial \varphi - \mu \partial \varphi)$.  Since $F\in W^{1,p}(\C)$, it follows that $\Beurling (\bar\partial F) = \partial F$. Thus,
$$\bar\partial F  = \mu \Beurling (\bar\partial F) + f \bar\partial \varphi -  \mu f \partial \varphi.$$
Taking modulus of continuity for $t>0$, we get that
$$\modulus{p}{(\bar\partial F)}(t) \leq \modulus{p}{(\mu \Beurling(\bar\partial F))}(t) + \modulus{p}{(f \bar\partial \varphi)}(t)+\modulus{p}{(\mu f  \partial \varphi)}(t),$$
and, using \rf{eqModulusProduct}, we get
\begin{align*}
\modulus{p}{(\bar\partial F)}(t)
	& \leq  \modulus{q}{\mu}(t) \norm{ \Beurling(\bar\partial F)}_{L^{r}} + \modulus{p}{ \Beurling(\bar\partial F)}(t) \norm{\mu}_{L^\infty} + \modulus{p}{(f \bar\partial \varphi)}(t)\\
	& \quad + \norm{ \mu}_{L^\infty} \modulus{p}{(\partial \varphi f)}(t)+\norm{\partial \varphi f}_{L^r} \modulus{q}{\mu}(t).
\end{align*}

By \rf{eqModulusBoundedTranslationInvariant} we have that $\modulus{p}{\Beurling(\bar\partial F)}(t)\leq \norm{\Beurling}_{L^p\to L^p}\, \modulus{p}{(\bar\partial F)}(t)$
\begin{align*}
\modulus{p}{(\bar\partial F)}(t)
	& \leq  \modulus{q}{\mu}(t)  \norm{ \Beurling}_{L^{r}\to L^r}\norm{\bar\partial F}_{L^{r}} +\kappa\, \norm{\Beurling}_{L^p\to L^p}\, \modulus{p}{(\bar\partial F)}(t)  + \modulus{p}{(f \bar\partial \varphi)}(t)\\
	& \quad + \kappa\, \modulus{p}{(\partial \varphi f)}(t)+\norm{\partial \varphi f}_{L^r} \modulus{q}{\mu}(t).
\end{align*}

Note that $\bar \partial F$ is compactly supported and $p$-integrable (see Theorem \ref{theoSelfImprovement}). Thus, $\modulus{p}{(\bar\partial F)}(t)$ is finite and we can infer that
\begin{align*}
\modulus{p}{(\bar\partial F)}(t)
	& \leq  \frac{\modulus{q}{\mu}(t) \left(\norm{ \Beurling}_{L^{r}\to L^r}\norm{\bar\partial F}_{L^{r}} + \norm{f \nabla \varphi}_{L^r}  \right)+ (1+\kappa) \modulus{p}{(f \nabla\varphi)}(t)}{1-  \kappa \norm{\Beurling}_{L^p\to L^p} }.
\end{align*}
By \rf{eqCaccioppoli} we have that $\norm{\bar\partial F}_{L^r}\leq C_{\kappa,r} \norm{f \nabla \varphi}_{L^r}$. Thus 
\begin{align*}
\modulus{p}{(\bar\partial F)}(t)
	& \leq  C_{\kappa,r,p} \norm{f \nabla \varphi}_{L^r}  \modulus{q}{\mu}(t)  +C_{\kappa,p} \, \modulus{p}{(f \nabla\varphi)}(t).
\end{align*}
Using that $\bar\partial F= \varphi \bar\partial f + f \bar\partial \varphi$, we get \rf{eqOptimalModulus}.
\end{proof}

\begin{corollary}\label{coroModulusCGOSBounded}
Let $\mu \in L^\infty$ be compactly supported in $\overline{\D}$ with $\norm{\mu}_{L^\infty} \leq \kappa <1$, let $2 \leq p <\infty$  with $\kappa \norm{\Beurling}_{L^p\to L^p}<1$, let $r\in [p, p_\kappa)$ and $q$ defined by $\frac1p=\frac1q+\frac1r$, and let $f_\mu$ be the complex geometric optics solution from Definition \ref{defFMu}.  Then
\begin{equation*}
\modulus{p}{(\bar\partial f_\mu )}(t)
	 \leq   e^{C_{\kappa,r,p} (1+|k|)} \left(  \modulus{q}{\mu}(t) + |t|^{1-\frac2p}\right).
\end{equation*}
\end{corollary}

\begin{proof}
Take $\chi_\D\leq \varphi\leq \chi_{2\D}$ with $|\nabla \varphi | \lesssim 1$. Since $\bar\partial f_\mu = \varphi \bar\partial f_\mu$, Theorem \ref{theoCaccioppoliModulus} leads to
\begin{align*}
 \modulus{p}{(\bar\partial f_\mu)}(t)
	& \leq C_{\kappa,r,p} \norm{f_\mu \nabla \varphi}_{L^r}  \modulus{q}{\mu}(t)  +C_{\kappa,p} \, \modulus{p}{(f_\mu \nabla\varphi)}(t).
\end{align*}
But for $|h|\leq t$, using the Sobolev embedding we have that
\begin{align}\label{eqControlOffdiagonalCaccioppoli}
\modulus{p}{(f_\mu \nabla\varphi)}(h) 
	& = \left(\int_{2\D} |f_\mu(x) \nabla\varphi(x) - f_\mu(x+h) \nabla\varphi(x+h)|^p\, dx\right)^\frac1p \\
\nonumber	& \leq \norm{f_\mu \nabla\varphi}_{C^{1-\frac2p}(2\D)} \left(\int_{2\D} |h|^{p-2} \, dx\right)^\frac1p 
		\leq C_p \norm{f_\mu \nabla\varphi}_{W^{1,p}(2\D)}|h|^{1-\frac2p}  \left(2\pi \right)^\frac2p\\
\nonumber	& \leq C_p \norm{f_\mu}_{W^{1,p}(2\D)} |t|^{1-\frac2p} .
\end{align}
On the other hand, by the Sobolev embedding theorem for subcritical indices, $\norm{f_\mu}_{L^r(2\D)}\leq C \norm{f_\mu}_{W^{1,p}(2\D)}$. Thus,
\begin{align*}
\modulus{p}{(\bar\partial f_\mu )}(t)
	& \leq  C_{\kappa,r,p}  \norm{f_\mu}_{W^{1,p}(2\D)}  \left(  \modulus{q}{\mu}(t) + |t|^{1-\frac2p}\right).
\end{align*}

But \rf{eqBeltramiAssimptotics} implies 
$$\norm{f_\mu}_{W^{1,p}(2\D)}  \leq C_p\norm{e^{ik\cdot}}_{W^{1,p}(2\D)} \norm{M_\mu}_{W^{1,p}(2\D)}, $$
and \rf{eqJostSobolevBounded} yields
\begin{equation}\label{eqSobolevForCGOS}
\norm{f_\mu}_{W^{1,p}(2\D)}  \leq  e^{C_{\kappa,p} (|k|+1)}.
\end{equation}
\end{proof}
Arguing analogously one gets Theorem \ref{theoCaccioppoliModulusModified}.

\subsection{Interpolation}\label{secInterpolation}
\begin{proof}[Proof of Theorem \ref{theoMainBisTheorem}]
Let $\kappa \in(0,1)$, let $2K< p <\infty$ with $K$ defined by \rf{eqKappaK}, and let $\omega$ be a modulus of continuity. We will show that there exists a modulus of continuity $\eta=\eta_{\mathcal{M}}$  depending only on $(\kappa, p, \omega)$  so that for every pair $\mu_1, \mu_2\in {\mathcal{M}}:=\mathcal{M}(\kappa, p, \omega)$, the estimate
\begin{equation}\label{eqSStability}
\norm{\mu_1-\mu_2}_{L^{s}(\D)} \leq C_{\kappa,s} \eta (\rho)
\end{equation}
holds for a fixed $\frac1s > \frac12 + \frac{K-1}2$, where $\rho:=\norm{\Lambda_1-\Lambda_2}_{H^{1/2}(\partial\D)\to H^{-1/2}(\partial\D)}$. By standard interpolation with $L^\infty$, we get $L^s$ stability for $0<s<\infty$, and Theorem \ref{theoMainBisTheorem} follows in particular.

To show \rf{eqSStability}, let $f_j=f_{\mu_j}(\cdot,1)$ (from now on we fix $k=1$). Note that $\bar\partial f_j=\mu_j \overline{\partial f_j}$. Thus, we have the almost everywhere identity
$$|\mu_1-\mu_2|=\frac{|\bar \partial f_1 \overline{\partial f_2}- \overline{\partial f_1 }\bar\partial f_2|}{|\partial f_1 \partial f_2|} = \frac{|\bar \partial f_1( \overline{\partial f_2}-  \overline{\partial f_1})+\overline{\partial f_1}( \bar \partial f_1 - \bar\partial f_2)|}{|\partial f_1 \partial f_2|} ,$$
so
$$|\mu_1-\mu_2| \leq \frac{|\mu_1||\partial ( f_2 -  f_1)| + |\bar\partial ( f_2 -  f_1)|}{|\partial f_2|}.$$
Let $s,s^*>0$ to be fixed, satisfying $\frac{1}{s^*}+\frac12=\frac1s$. Then we can apply H\"older's inequality to obtain
\begin{equation}\label{eqBreakMus}
\norm{\mu_1-\mu_2}_{L^{s}(\D)} \leq C \norm{|\bar\partial ( f_2 -  f_1)| +\kappa |\partial ( f_2 -  f_1)|}_{L^2(\D)}\norm{|\partial f_2|^{-1}}_{L^{s^*}(\D)}.
\end{equation}
By \cite[Lemma 4.6]{ClopFaracoRuiz}  the last (quasi)norm  is bounded by 
\begin{equation}\label{eqControlJacobianIntegrability}
\norm{|\partial f_2|^{-1}}_{L^{s^*}}\leq C_{s,\kappa}
\end{equation}
as long as $s^*< \frac{2}{K-1} $, that is, whenever $\frac1s > \frac12 + \frac{K-1}2$. 

We need to control the $L^2$-norm of the gradient of the difference in \rf{eqBreakMus}. To do so, we define a bump function $\chi_\D\leq \varphi \leq \chi_{2\D}$, and write $F=\varphi (f_1-f_2)$. Then, we can use the Plancherel's identity to state that
$$\norm{\partial ( f_2 -  f_1)}_{L^2(\D)} \leq \norm{\partial F}_{L^2(\C)} = \norm{\bar\partial F}_{L^2(\C)}= \norm{\widehat{ \bar\partial F }}_{L^2(\C)}.$$
and, by similar reasons, 
$$\norm{\bar\partial ( f_2 -  f_1)}_{L^2(\D)} \leq \norm{\widehat{ \bar\partial F }}_{L^2(\C)}.$$
Take $R$ to be fixed depending on $\rho$. Then
\begin{equation}\label{eqBreakDeltaFMinusF}
 \norm{|\bar\partial ( f_2 -  f_1)| +\kappa |\partial ( f_2 -  f_1)|}_{L^2(\D)} \leq 2\norm{\widehat{ \bar\partial F }}_{L^2(\D_R)}+2\norm{\widehat{ \bar\partial F }}_{L^2(\D_R^c)}.
\end{equation}

On one hand, for the low frequencies we use that $\left|\widehat{ \bar\partial F }(\xi)\right| \approx |\xi| \left|\widehat{F}(\xi)\right|$ and, thus,
$$\norm{\widehat{ \bar\partial F }}_{L^2(\D_R)}\lesssim R \norm{\widehat{ F }}_{L^2(\C)} = R \norm{ \varphi ( f_2 -  f_1)}_{L^2(\C)} \leq R \norm{f_2 -  f_1}_{L^2(2\D)} .$$
Since we have fixed $k=1$, by Theorem \ref{theoUInftyIsSmall},  there exists a modulus of continuity $\iota_{\mathcal{M}}$ depending only on $\kappa$, $p$ and $\omega$  so that
$$\norm{f_2 -  f_1}_{L^2(2\D)}\leq \norm{f_2 -  f_1}_{L^\infty(2\D)} (2\pi)^\frac12\leq   \iota_{\mathcal{M}}(\rho).$$
Thus, we get that
\begin{equation}\label{eqBoundLowFrequencies}
\norm{\widehat{ \bar\partial F }}_{L^2(\D_R)} \leq  R \, \iota_{\mathcal{M}}(\rho) .
\end{equation}

For the high frequencies we use Lemma \ref{lemFourier}, which implies that
\begin{equation*}
\norm{\widehat{ \bar\partial F }}_{L^2(\D_R^c)} \leq C\left( \modulus{2}{(\bar\partial(\varphi f_2 ))}\left(\frac1R\right)+ \modulus{2}{(\bar\partial(\varphi f_1))}\left(\frac1R\right)\right) .
\end{equation*}
Let $j\in \{1,2\}$.  By Corollary \ref{coroModulusCGOSBounded} and \rf{eqControlOffdiagonalCaccioppoli} we have that 
\begin{align*}
\modulus{2}(\bar\partial(\varphi f_j ))(R^{-1}) 
	& \leq \modulus{2}{(\varphi \bar\partial f_j )}(R^{-1}) +\modulus{2}{( f_j \bar\partial\varphi)}(R^{-1}) \\
	& \leq e^{2C_{\kappa,p}} \left(  \modulus{p}{\mu}(R^{-1}) + |R^{-1}|^{1-\frac2p}\right) + C_p \norm{f_j}_{W^{1,p}(2\D)} |R^{-1}|^{1-\frac2p} .
\end{align*}
By \rf{eqSobolevForCGOS},
\begin{equation}\label{eqBoundHighFrequencies}
\norm{\widehat{ \bar\partial F }}_{L^2(\D_R^c)}
	 \leq C_{\kappa,p} \left(  \omega(R^{-1}) + R^{\frac2p-1}\right).
\end{equation}

Combining \rf{eqBreakMus}, \rf{eqControlJacobianIntegrability}, \rf{eqBreakDeltaFMinusF}, \rf{eqBoundLowFrequencies} and \rf{eqBoundHighFrequencies}, we obtain
\begin{equation*}
\norm{\mu_1-\mu_2}_{L^{s}(\D)} \leq C_{s,\kappa} \inf_{R} \left( R\iota_{\mathcal{M}}(\rho) + C_{\kappa,p} \left( \omega(R^{-1})  + R^{\frac2p-1}\right)\right).
\end{equation*}
Defining
$$\eta(t):= \inf_{R} \left( C_{s,\kappa} R\iota_{\mathcal{M}}(t) + C_{\kappa,p,s} \left( \omega(R^{-1}) + R^{\frac2p-1}\right)\right),$$
we obtain the result. 
To see that $\eta $ tends to zero as $\rho\to 0$, it is enough to check that this happens with $R:=\iota_{\mathcal{M}}(t)^{-1/2}$. Note that using this value for $R$ in the last expression we get 
$$\eta (t) \leq  \left( C_{s,\kappa} \iota_{\mathcal{M}}(t)^\frac12 + C_{\kappa,p,s} \left( \omega(\iota_{\mathcal{M}}(t)^\frac12) + \iota_{\mathcal{M}}(t)^{\frac12-\frac1p}\right)\right).$$
The theorem follows combining that $\lim_{t\to 0}\omega (t)=0$ and $\lim_{t\to0}\iota_{\mathcal{M}}(t)=0$.

Whenever $\omega$ is upper semi-continuous, by Theorem \ref{theoUInftyIsSmall} we obtain the quantitative estimate in Theorem \ref{theoMainTheorem}.
\end{proof}

\bibliography{../../../bibtex/Llibres}
\end{document}